\documentclass[1 [leqno,11pt]{amsart}
\usepackage{amssymb, amsmath,amsmath,latexsym,amssymb,amsfonts,amsbsy, amsthm,mathtools,graphicx,CJKutf8,CJKnumb,CJKulem,color}

\numberwithin{equation}{section}

\allowdisplaybreaks

\let\al=\alpha
\let\b=\beta
\let\g=\gamma

\let\f=\frac

\let\om=\omega
\let\G= \Gamma

\let\Om=\Omega

\let\na=\nabla
\let\th=\theta
\let\pa=\partial

\def\R{\mathbf R}

\def\tphi{\widetilde{\phi}}

\def\no{\noindent}

\def\eqdef{\buildrel\hbox{\footnotesize def}\over =}
\def\ef{\hphantom{MM}\hfill\llap{$\square$}\goodbreak}

\def\bbT{\mathbb{T}}

\newcommand{\beq}{\begin{equation}}
\newcommand{\eeq}{\end{equation}}
\newcommand{\ben}{\begin{eqnarray}}
\newcommand{\een}{\end{eqnarray}}
\newcommand{\beno}{\begin{eqnarray*}}
\newcommand{\eeno}{\end{eqnarray*}}

\newtheorem{theorem}{Theorem}[section]
\newtheorem{definition}[theorem]{Definition}
\newtheorem{lemma}[theorem]{Lemma}
\newtheorem{proposition}[theorem]{Proposition}

\newtheorem{remark}[theorem]{Remark}

\begin{document}

\title{Linear inviscid damping for a class of monotone shear flow in Sobolev spaces}

\author{Dongyi Wei}
\address{School of Mathematical Science, Peking University, 100871, Beijing, P. R. China}
\email{jnwdyi@163.com}

\author{Zhifei Zhang}
\address{School of Mathematical Science, Peking University, 100871, Beijing, P. R. China}
\email{zfzhang@math.pku.edu.cn}

\author{Weiren Zhao}
\address{Department of Mathematics, Zhejiang University, 310027, Hangzhou, P. R. China}
\email{zjzjzwr@126.com}
\date{\today}

\maketitle

\begin{abstract}
In this paper, we prove the decay estimates of the velocity and $H^1$ scattering for the 2D linearized Euler equations around a class of
 monotone shear flow in a finite channel. Our result is consistent with the decay rate predicted by Case in 1960.
\end{abstract}

\section{introduction}
In this paper, we consider the 2D incompressible Euler equations in a finite channel $\big\{(x,y): x\in \bbT, y\in [0,1]\big\}$:
\beq
\label{eq:Euler}
\left\{
\begin{array}{l}
\pa_tV+V\cdot\nabla V+\nabla P=0,\\
\na\cdot V=0,\\
V^2(t,x,0)=V^2(t,x,1)=0,\\
V|_{t=0}=V_0(x,y).
\end{array}\right.
\eeq
where $V=(V^1,V^2)$ and $P$ denote the velocity and the pressure of the fluid respectively.
Let $\omega=\pa_xV^2-\pa_yV^1$ be the vorticity, which satisfies
\beq
\label{equ:vorticity}
\omega_t+V\cdot\nabla \omega=0.
\eeq

It is well-known that the 2D incompressible Euler equations are globally well-posed for smooth data \cite{Che, Maj}. However, the long time behaviour
of the solution is widely open. We refer to \cite{Kis, Den} for recent relevance results.

We are concerned with the asymptotic stability of the 2D linearized Euler equations around the shear flow
$(u(y),0)$, which is a steady solution of 2D Euler equations.
The linearized Euler equations around a shear flow $(u(y),0)$ take
\ben\label{eq:Euler-L}
\left\{
\begin{array}{l}
\pa_t\omega+\mathcal{L}\omega=0,\\
\om|_{t=0}=\om_0(x,y),
\end{array}\right.
\een
where $\mathcal{L}=u(y)\pa_x+u''(y)\pa_x(-\Delta)^{-1}$.

The stability of 2-D Euler equations is a very active field in Physics and Mathematics \cite{Dra}, especially for shear flows \cite{Sch}.
Rayleigh's inflection point theorem gives a necessary condition for linear stability of shear flow: $u(y)$ has no infection points \cite{Ray}.
Arnold's theorem gives a sufficient condition for nonlinear Liapunov stability of shear flow \cite{Arnold}:
\beno
0<c_1\le \f {u(y)} {u''(y)}\le c_2<+\infty.
\eeno
Lin \cite{Lin-SIAM} provided a large classes of unstable shear flows.
We refer to \cite{Bar, Lin-IN, Gre, Vish} and references therein for nonlinear instability.

In 1907, Orr \cite{Orr} observed that the velocity tends to zero as $t\rightarrow +\infty$ for the linearized Euler equations around
Couette flow $(y,0)$. In this case, the linearized vorticity equation becomes
\beno
\om_t+y\pa_x\om=0.
\eeno
Thus, $\om(t,x,y)=\om_0(x-ty,y)$. Especially, in the case of the infinite channel $\bbT\times \R$, the velocity can be explicitly solved.
Indeed, let $\psi$ be the stream function, i.e., $(V^1,V^2)=(\pa_y\psi,-\pa_x\psi)$. Hence,
\ben\label{eq:stream}
-\Delta\psi=\om.
\een
Since  we deduce by taking Fourier transform to (\ref{eq:stream}) that
\beno
\widehat{\psi}(t,\al,\xi)=\f {\widehat{\om}_0(\al,\xi+\al t)} {\al^2+|\xi|^2},
\eeno
which implies that $V^1$ decays at $t^{-1}$, while $V^2$ decays at $t^{-2}$. In the case of finite channel, Case \cite{Case} gave a
formal proof of $t^{-1}$ decay for the velocity. Recently, Lin and Zeng \cite{LZ} present the optimal linear decay estimates of the velocity
for the data in Sobolev space. More precisely, if $\int_{\bbT}\om_0(x,y)dx=0$, then it holds that
\begin{itemize}
\item[1.] if $\om_0(x,y)\in H^{-1}_xH^1_y$, then
\beno
\|V(t)\|_{L^2}=O\big(\f 1 t\big),
\eeno
\item[2.] if $\om_0(x,y)\in H^{-1}_xH^2_y$, then
\beno
\|V^2(t)\|_{L^2}=O\big(\f 1 {t^2}\big).
\eeno
\end{itemize}

Such inviscid damping is surprising for a time reversible system. The basic mechanism leading to this phenomena is vorticity mixing driven by shear flow,
which may be related to the appearance of coherent structures in 2D turbulence. This behaviour is similar to Landau damping
\cite{Lan}, which predicted the rapid decay of the electric field of the linearized Vlasov equation around homogeneous equilibrium.

It is a very difficult problem to extend linear damping to nonlinear damping.
Mouhot and Villani \cite{Mou} made a breakthrough and proved nonlinear Landau damping for the perturbation in Gevrey class(see also \cite{Bed2}).
Motivated by \cite{Mou}, Bedrossian and Masmoudi also proved the nonlinear inviscid damping of 2D Euler equations around
Couette flow in infinite channel still for the perturbation in Gevrey class. Lin and Zeng \cite{LZ, LZ-CMP} also show that
nonlinear damping is not true for the perturbation in low regularity Sobolev spaces.

The goal of this paper is to prove linear damping for the 2D Euler equations around general shear flow.
In this case, there are few rigorous mathematical results.
Case \cite{Case} gave the formal prediction for the decay of the velocity by using the Laplace transform and the leading singularity of the
resolvent.  Rosencrans and Sattinger \cite{Ros} gave $t^{-1}$ decay of the stream function with a continuous spectrum projection for analytic monotone shear flow.
Stepin \cite{Ste} proved $t^{-\nu}(\nu<\mu_0)$ decay of the stream function for the monotone shear flow $u(y)\in C^{2+\mu_0}(\mu_0>\f12)$ without inflection point. In a very interesting paper
\cite{Bou}, Bouchet and Morita predicted similar decay estimates of the velocity for a class of stable shear flow with stationary streamlines by using the Laplace transform
and an important observation: depletion phenomena of the vorticity at the stationary streamlines. More precisely, they formally proved that
\beno
\om(t,x,y)\sim \om_\infty(x,y)\exp(-iku(y)t)+O(t^{-\gamma})\quad \textrm{as }t\rightarrow +\infty,
\eeno
where $\om_\infty(x,y_c)=0$ at stationary points $y_c$ of $u(y)$.

In a recent paper, C. Zillinger \cite{Zill} proved the same decay estimates as those of Couette flow given by Lin and Zeng \cite{LZ} for a class of monotone shear flow
in Sobolev spaces. However, his result imposed a strong assumption that $L\|u''\|_{W^{3,\infty}}$ is small, where $L$ is the wave-length with respect to $x$.
Moreover, he also required that the initial vorticity vanishes on $y=0,1$ in the case of finite channel. Thus, the linear inviscid damping
is still open for general monotone shear flow. \medskip

The main result of this paper is stated as follows.

\begin{theorem}
\label{Thm:main}
Let $u(y)\in C^4([0,1])$ be a monotone function. Suppose that the linearized operator $\mathcal{L}$
has no embedding eigenvalues. Assume that $\int_{\bbT}\om_0(x,y)dx=0$ and $P_{\mathcal{L}}\om_0=0$,
where $P_{\mathcal{L}}$ is the spectral projection to $\sigma_d\big(\mathcal{L}\big)$.
Then it holds that
\begin{itemize}

 \item[1.] if $\om_0(x,y)\in H^{-1}_xH^1_y$, then
\beno
\|V(t)\|_{L^2}\leq \frac{C}{\langle t\rangle}\|\omega_0\|_{H^{-1}_xH_y^1};
\eeno
\item[2.] if $\om_0(x,y)\in H^{-1}_xH^2_y$, then
\beno
\|V^2(t)\|_{L^2}\leq \frac{C}{\langle t\rangle^2}\|\omega_0\|_{H^{-1}_xH_y^2};
\eeno
\item[3.] if $\om_0(x,y)\in H^{-1}_xH^k_y$ for $k=0,1$, there exists $\om_\infty(x,y)\in H^{-1}_xH^k_y$ such that
\beno
\|\om(t,x+tu(y),y)-\om_\infty\|_{L^2}\longrightarrow 0\quad \textrm{as}\quad t\rightarrow +\infty.
\eeno
\end{itemize}
\end{theorem}

Let us give some remarks on Theorem \ref{Thm:main}.

\begin{itemize}

\item If $u(y)$ has no inflection points, then $\mathcal{L}$ has no eigenvalues. In Remark \ref{rem:spectrum}, we will present a sufficient condition on $u(y)$ in the case when $u(y)$ has inflection points so that $\mathcal{L}$ has no eigenvalues.

\item If the wave-length $L$ with respect to $x$ is suitably small, then $\mathcal{L}$ has no embedding eigenvalues.
In Lemma \ref{lem:spectrum}, we will present a sufficient and necessary condition on $u(y)$ so that $\mathcal{R}_\al$ has no embedding eigenvalues.

\item By Zillinger's recent result \cite{Zill2},
the $L_x^2H_y^2$ norm of $W(t,x,y)\triangleq\om(t,x+tu(y),y)$ may blow up. Thus, it is in general unexpected for the $H^2$ scattering.

\item Our proof strongly relies on the monotonicity of $u(y)$. Thus, it remains unknown whether the decay estimates predicted by Bouchet and Morita \cite{Bou}
for stable shear flow with stationary streamlines  can be justified.

\item Nonlinear inviscid damping is a challenging question even for the analytic perturbation.

\end{itemize}

\section{Sketch of the proof}

Our proof is based on the representation formula of the solution
\beno
\widehat{\psi}(t,\al,y)=\frac{1}{2\pi i}\int_{\partial\Omega}
e^{-i\al tc}(c-\mathcal{R}_\al)^{-1}\widehat{\psi}(0,\al,y)dc,
\eeno
where $\mathcal{R}_\al$ is the Rayleigh operator defined by (\ref{def:ray ope}).
Under the assumption that $\mathcal{L}$ has no embedding eigenvalues and $P_{\mathcal{L}}\om_0=0$, the asymptotic behaviour of the solution is only
related to the continuous spectrum. Thus, we only need to study
\beno
\widehat{\psi}(t,\al,y)=\frac{1}{2\pi i}\int_{u(0)}^{u(1)}
e^{-i\al tc}\lim_{\epsilon\rightarrow 0+}\big((c-i\epsilon-\mathcal{R}_\al)^{-1}-(c+i\epsilon-\mathcal{R}_\al)^{-1}\big)\widehat{\psi}(0,\al,y)dc.
\eeno

To establish the estimates of the resolvent $(c-\mathcal{R}_\al)^{-1}$, we need to study the inhomogeneous Rayleigh equation
\beno
\left\{
\begin{aligned}
&\Phi''-\al^2\Phi-\frac{u''}{u-c}\Phi=f,\\
&\Phi(0)=\Phi(1)=0,
\end{aligned}
\right.
\eeno
with $f=\f {\widehat{\om}_0(\al,y)} {i\al(u-c)}$. Indeed, it holds that
\beno
&&\lim_{\epsilon\rightarrow 0+}\big((c-i\epsilon-\mathcal{R}_\al)^{-1}-(c+i\epsilon-\mathcal{R}_\al)^{-1}\big)\widehat{\psi}(0,\al,y)\\
&&\quad=i\al\lim_{\epsilon\rightarrow 0+}\big(\Phi(y,c+i\epsilon)-\Phi(y,c-i\epsilon)\big)\triangleq i\al\widetilde{\Phi}(y,c).
\eeno
Thus, we obtain
\ben\label{eq:psi-sket}
\widehat{\psi}(t,\al,y)=\frac{1}{2\pi}\int_{u(0)}^{u(1)}
e^{-i\al tc}\al\widetilde{\Phi}(y,c)dc.
\een

Formally, if one can show that $\widetilde{\Phi}(y,c)\in W^{2,1}$ in $c$, then integration by parts gives
\beno
\widehat{\psi}(t,\al,y)\sim O(t^{-2})\int_{u(0)}^{u(1)}e^{-i\al tc}\pa_c^2\widetilde{\Phi}(y,c)dc\sim O(t^{-2}).
\eeno
One of main difficulties is
that the solution $\Phi(y,c)$ of the inhomogeneous Rayleigh equation has a singularity of order $(y-y_c)\log|y-y_c|$ with $y_c=u^{-1}(c)$(see \cite{Case, Bou}).
Thus, $\pa_c^2\widetilde{\Phi}(y,c)\sim \f 1 {y-y_c}\notin L^1$. This may be the main reason why the authors in \cite{Case, Ros, Ste}
only obtained the $O(t^{-1})$ decay of the stream function even in the analytic framework.

Indeed, in Section 6 and 7, we will show that
$\al\widetilde{\Phi}(y,c)=2\rho(c)\mu(c)\G(y,c)$ with $\rho(c)=(c-u(0))(u(1)-c)$ and
\beno
\G(y,c)=
\left\{
\begin{aligned}
&\phi(y,c)\int_0^y\frac{1}{\phi(z,c)^2}dz\quad0\leq y<y_c,\\
&\phi(y,c)\int_1^y\frac{1}{\phi(z,c)^2}dz\quad y_c<y\leq 1,
\end{aligned}
\right.
\eeno
where $\phi(y,c)$ is the solution of the homogeneous Rayleigh equation:
\beno
\left\{
\begin{array}{l}
\phi''-\al^2\phi-\frac{u''}{u-c}\phi=0,\\
\phi(y_c,c)=0,\quad \phi'(y_c,c)=u'(y_c).
\end{array}\right.
\eeno
Thus, $\phi(y,c)$ has the behaviour near $y_c$:
\beno
\phi(y,c)\sim u'(y_c)(y-y_c)+\f {u''(y_c)} {2u'(y_c)}(y-y_c)^2+o\big((y-y_c)^2\big),
\eeno
which implies that $\G(y,c)$ has the behaviour near $y_c$:
\beno
\G(y,c)\sim a+b(y-y_c)\log|y-y_c|,
\eeno
for some constants $a,b$.

The goal of Section 4-Section 5 is to obtain various kinds of uniform estimates for the solution of the homogeneous Rayleigh equations.
The assumption that $u(y)$ is monotone plays an important role.
In Section 8, we will establish the weighted $H^2$ estimate of $\mu(c)$, where we need to
assume that $\mathcal{L}$ has no embedding eigenvalues.

Based on the solution formula (\ref{eq:psi-sket}), using the weighted $H^2$ estimate of $\mu(c)$ and $L^p$ boundedness for various kinds of singular integral operators,
we will establish the $H^2_y$ estimate of $W(t,x,y)=\om(t,x+tu(y),y)$ in Section 9.
In fact, we only prove the weighted $H^2_y$ estimate, while $H^2_y$ bound is impossible in general.

With the uniform Sobolev estimates of the vorticity, the decay estimates can be deduced by following the dual argument introduced by Lin and Zeng \cite{LZ}.

In the appendix, we will establish $L^p$ boundedness for various kinds of singular integral operators, which was used in Section 8 and Section 9.

\section{Spectrum of the linearized operator}

In terms of the stream function $\psi$, the linearized Euler equations take
\beq\label{eq:Euler-L-stream}
\partial_t\Delta{\psi}+u(y)\partial_x\Delta{\psi}-u''(y)\partial_x{\psi}=0.
\eeq
Taking the Fourier transform in $x$, we get
\beno
(\partial_y^2-\al^2)\partial_t\widehat{\psi}=i\al\big(u''(y)-u(y)(\partial_y^2-\al^2)\big)\widehat{\psi}.
\eeno
Inverting the operator $(\partial_y^2-\al^2)$, we find
\ben\label{eq:Euler-L-operator}
-\frac{1}{i\al}\partial_t\widehat{\psi}=\mathcal{R}_\al\widehat{\psi},
\een
where
\ben\label{def:ray ope}
\mathcal{R}_\al\widehat{\psi}=-(\partial_y^2-\al^2)^{-1}\big(u''(y)-u(\partial_y^2-\al^2)\big)\widehat{\psi}.
\een
It is easy to show that
\beno
\bigcup_{\al}\sigma_d(i\al\mathcal{R}_\al)=\sigma_d(\mathcal{L}).
\eeno

Let us recall some classical results for the spectrum $\sigma(\mathcal{R}_\al)$ of the operator $\mathcal{R}_\al$(see \cite{Ros, Ste}for more details).

\begin{itemize}

\item[1.] The spectrum $\sigma(\mathcal{R}_\al)$ is compact;

\item[2.] The continuous spectrum $\sigma_c(\mathcal{R}_\al)$ is contained in the range $ \textit{Ran}(u)$ of $u(y)$;

\item[3.] The eigenvalues of $\mathcal{R}_\al$ can not cluster except possibly along on $\textit{Ran}(u)$;

\item[4.] If $u(y)$ has no infection points in $[0,1]$, then $\mathcal{R}_\al$ has no embedding eigenvalues;

\item[5.] If $u(y)$ has  inflection points, then $\mathcal{R}_{\al}$ has no embedding eigenvalues for $\al^2>\al_{max}^2$,
where
\beno
\al_{max}^2\eqdef-\inf_{y_c: u''(y_c)=0}\inf_{\phi\in H^{1}_0(0,1)}\frac{\int_0^1|\phi'(y)|^2-\frac{u''(y)}{u(y)-u(y_c)}|\phi(y)|^2dy}{\int_0^1|\phi(y)|^2dy}.
\eeno
\end{itemize}

\no{\it Proof of 4}. Let $\phi\in H^2(0,1)\cap H_0^1(0,1)$ be an eigenfunction of $\mathcal{R}_\al$ with the eigenvalue $c\in \textit{Ran}(u)$, i.e.,
$\mathcal{R}_\al\phi=c\phi$, which can be reduced to the well-known Rayleigh equation:
\beq\label{eq:Rayleigh}
(u-c)(\phi''-\al^2\phi)-u''\phi=0.
\eeq
If $u''(u^{-1}(c))\neq 0$, $\phi(u^{-1}(c))=0$ by (\ref{eq:Rayleigh}).
Taking the inner product with $\frac{\phi}{u-c}$ on both sides of (\ref{eq:Rayleigh}), we obtain
\beno
\int_{0}^1\phi''{\phi}dy-\al^2\int_0^1|\phi|^2dy-\int_0^1u''\frac{|\phi|^2}{u-c}dy=0.
\eeno
Integration by parts gives
\beno
-\int_0^1\Big|\phi'-u'\frac{\phi}{u-c}\Big|^2dy-\al^2\int_0^1|\phi|^2dz=0,
\eeno
which implies that $\phi\equiv 0$. Thus, $u''(u^{-1}(c))=0$ if $c$ is an embedding eigenvalue.

\medskip

\no{\it Proof of 5.} If $c$ is an embedding eigenvalue with eigenfunction $\phi$, then we have
\beno
\int_0^1|\phi'|^2-\frac{u''}{u-c}|\phi|^2dy+\al^2\int_0^1|\phi|^2dy=0.
\eeno
However, for $\al^2>\al_{max}^2$, we have
\beno
\int_0^1|\phi'|^2-\frac{u''}{u-c}|\phi|^2dy+\al^2\int_0^1|\phi|^2dy\neq0.
\eeno

In Lemma \ref{lem:spectrum}, we will present a sufficient and necessary condition on $u(y)$ so that $\mathcal{R}_\al$ has no embedding eigenvalues. In Remark \ref{rem:spectrum}, we will present a sufficient condition so that $\mathcal{R}_\al$ has no eigenvalues, which in particular implies that $\mathcal{R}_\al$ has no eigenvalues if $u(y)$ has inflection points in $[0,1]$ and $\al^2>\al_{max}^2$.

Let $\Omega$ be a simple connected domain including the spectrum $\sigma(\mathcal{R}_\al)$ of $\mathcal{R}_\al$.
We have the following representation formula of the solution to (\ref{eq:Euler-L-operator}):
\ben\label{eq:stream formula}
\widehat{\psi}(t,\al,y)=\frac{1}{2\pi i}\int_{\partial\Omega}
e^{-i\al tc}(c-\mathcal{R}_\al)^{-1}\widehat{\psi}(0,\al,y)dc.
\een
Thus, the large time behaviour of the solution $\widehat{\psi}(t,\al,y)$
is reduced to the study of the resolvent $(c-\mathcal{R}_\al)^{-1}$.

\section{The homogeneous Rayleigh equation}

To study the resolvent $(c-\mathcal{R}_\al)^{-1}$, we first construct a smooth solution of the homogeneous Rayleigh equation on $[0,1]$:
\ben\label{eq:Rayleigh-H}
(u-c)(\phi''-\al^2\phi)-u''\phi=0,
\een
where the complex constant $c$ will be taken in four kinds of domains defined by
\beno
&&D_0\triangleq\big\{c\in [u(0),u(1)]\big\},\\
&&D_{\epsilon_0}\triangleq\big\{c=c_r+i\epsilon,~c_r\in [u(0),u(1)], 0<|\epsilon|<\epsilon_0\big\},\\
&&B_{\epsilon_0}^l\triangleq\big\{c=u(0)+\epsilon e^{i\theta},~0<\epsilon<\epsilon_0,~\frac{\pi}{2}\leq \theta\leq\frac{3\pi}{2}\big\},\\
&&B_{\epsilon_0}^r\triangleq\big\{c=u(1)-\epsilon e^{i\theta},~0<\epsilon<\epsilon_0,~\frac{\pi}{2}\leq \theta\leq\frac{3\pi}{2}\big\},
\eeno
for some $\epsilon_0\in (0,1)$. We denote
\ben\label{eq:Omega}
\Om_{\epsilon_0}\triangleq D_0\cup D_{\epsilon_0}\cup B_{\epsilon_0}^l\cup B^r_{\epsilon_0}.
\een

In the sequel, we always assume that $u(y)\in C^4([0,1])$ and satisfies $u'(y)\ge c_0$ for some $c_0>0$.

\subsection{Functional space and the integral operator}

Given $|\al|\ge 1$, let $A$ be a constant larger than $C|\al|$ with
$C\ge 1$ only depending on $c_0$ and $\|u\|_{C^4}$.

\begin{definition}
For a function $f(y,c)$ defined on $[0,1]\times \Om_{\epsilon_0}$, we define
\beno
&&\|f\|_{X_0}\eqdef\sup_{(y,c)\in [0,1]\times D_0}\bigg|\frac{f(y,c)}{\cosh(A(y-y_c))}\bigg|,\\
&&\|f\|_{X}\eqdef\sup_{(y,c)\in [0,1]\times D_{\epsilon_0}}\bigg|\frac{f(y,c)}{\cosh(A(y-y_c))}\bigg|,\\
&&\|f\|_{X_l}\eqdef\sup_{(y,c)\in [0,1]\times B_{\epsilon_0}^l}\bigg|\frac{f(y,c)}{\cosh(Ay)}\bigg|,\\
&&\|f\|_{X_r}\eqdef\sup_{(y,c)\in [0,1]\times B_{\epsilon_0}^r}\bigg|\frac{f(y,c)}{\cosh(A(y-1))}\bigg|,
\eeno
where $y_c=u^{-1}(c_r)$ with
\ben\label{def:cr}
c_r=Re\, c \quad \textrm{for }c\in D_{\epsilon_0},\quad c_r=u(0)\quad \textrm{for }c\in B_{\epsilon_0}^l,
\quad c_r=u(1)\quad \textrm{for }c\in B_{\epsilon_0}^r.
\een
\end{definition}

\begin{definition}
For a function $f(y,c)$ defined on $[0,1]\times \Om_{\epsilon_0}$, we define
\beno
&&\|f\|_{Y_C}\eqdef\sum_{k=0}^2\sum_{|\b|=k}A^{-k}\|\partial^{\b}f\|_{X_0},\\
&&\|f\|_{Y}\eqdef\|f\|_{X}+\frac{1}{A}\big(\|\partial_xf\|_{X}+\|\partial_{c}f\|_{X}\big),\\
&&\|f\|_{Y_l}\eqdef\|f\|_{X_l}+\frac{1}{A}\big(\|\partial_xf\|_{X_l}+\|\partial_{c}f\|_{X_l}\big),\\
&&\|f\|_{Y_r}\eqdef\|f\|_{X_r}+\frac{1}{A}\big(\|\partial_xf\|_{X_r}+\|\partial_{c}f\|_{X_r}\big).
\eeno

\end{definition}

\begin{definition}
Let $y_c=u^{-1}(c_r)$ with $c_r$ defined by (\ref{def:cr}).
The integral operator $T$ is defined by
\beno
T\triangleq T_0\circ T_{2,2}\eqdef\int_{y_c}^y\f 1 {(u(y')-c)^2}\int_{y_c}^{y'}f(z,c)(u(z)-c)^2dzdy',
\eeno
where
\beno
&&T_0f(y,c)\eqdef\int_{y_c}^yf(z,c)dz,\\
&&T_{k,j}f(y,c)\eqdef\frac{1}{(u(y)-c)^j}\int_{y_c}^yf(z,c)(u(z)-c)^kdz,\quad j\le k+1.
\eeno
\end{definition}

\begin{lemma}\label{lem:T-bound}
There exists a constant $C$ independent of $A$ such that
\beno
&&\|Tf\|_{Y_C}\leq \frac{C}{A^2}\|f\|_{Y_C},\quad\|Tf\|_{Y}\leq \frac{C}{A^2}\|f\|_Y,\\
&&\|Tf\|_{Y_l}\leq \frac{C}{A^2}\|f\|_{Y_l},\quad\|Tf\|_{Y_r}\leq \frac{C}{A^2}\|f\|_{Y_r}.
\eeno
\end{lemma}

\begin{proof}
We only prove the first inequality. The proof of the other inequalities is similar.
A direct calculation shows
\begin{align*}
\|T_0f\|_{X_0}
=&\sup_{(y,c)\in[0,1]\times D_0}\bigg|
\frac{1}{\cosh A(y-y_c)}\int_{y_c}^y
\frac{f(z,c)}{\cosh A(z-y_c)}\cosh A(z-y_c)dz\bigg|\\
\leq &\sup_{(y,c)\in[0,1]\times D_0}\bigg|
\frac{1}{\cosh A(y-y_c)}\int_{y_c}^y
\cosh A(z-y_c)dz\bigg|\|f\|_{X_0}\\
\leq &\frac{1}{A}\|f\|_{X_0},
\end{align*}
and
\begin{align*}
\|T_{2,2}f\|_{X_0}
\leq &\sup_{(y,c)\in[0,1]\times D_0}
\bigg|\frac{y-y_c}{\cosh A(y-y_c)}
\int_0^1
\cosh tA(y-y_c)dt
\bigg|\|f\|_{X_0}\\
\leq& \frac{C}{A}\|f\|_{X_0},
\end{align*}
which imply
\ben\label{eq:T-bound-1}
\|Tf\|_{X_0}\le \frac C {A^2}\|f\|_{X_0}.
\een

It is easy to see that
\beno
&&\partial_yTf(y,c)=T_{2,2}f(y,c),\\
&&\partial_cTf(y,c)=2T_0\circ T_{2,3}f(y,c)-2T_0\circ T_{1,2}f(y,c)+T\partial_cf(y,c),
\eeno
and
\begin{align*}
\|T_{k,k+1}f\|_{X_0}
\leq &C\sup_{(y,c)\in[0,1]\times D_0}
\bigg|\frac{1}{\cosh A(y-y_c)}\int_0^1
\cosh tA(y-y_c)dt
\bigg|\|f\|_{X_0}\\
\leq&  C\|f\|_{X_0},
\end{align*}
which along with (\ref{eq:T-bound-1}) yield
\ben\label{eq:T-bound-2}
&\|\partial_{y,c}Tf\|_{X_0}\leq \frac{C}{A}\|f\|_{X_0}
+\frac C {A^2}\|\partial_{c}f\|_{X_0}.
\een

Using the formula
\begin{align*}
T_{k,k+1}f(y,c)
=\int_0^1\frac{\big(\int_0^1u'(y_c+st(y-y_c))ds\big)^{k}}{\big(\int_0^1u'(y_c+s(y-y_c))ds\big)^{k+1}}f(y_c+t(y-y_c),c)t^{k}dt,
\end{align*}
we can deduce that
\begin{align*}
\|\partial_{y,c}^2T_{k,k+1}f\|_{X_0}
\leq & C\sum_{|\b|\le 2}\|\pa^{\b}f\|_{X_0}.
\end{align*}
Then by a similar argument leading to (\ref{eq:T-bound-2}), we obtain
\ben\label{eq:T-bound-3}
&\|\partial_{y,c}^2Tf\|_{X_0}\leq C\|f\|_{X_0}+\f C A\|\pa_{y,c}f\|_{X_0}
+\frac C {A^2}\|\partial_{y,c}^2f\|_{X_0}.
\een

Putting (\ref{eq:T-bound-1})-(\ref{eq:T-bound-3}) together, we conclude the first inequality.
\end{proof}

\subsection{Existence of the solution}

\begin{proposition}\label{prop:Rayleigh-Hom}
There exists a solution $\phi(y,c)\in C^1\big([0,1]\times \Omega_{\epsilon_0}\setminus D_0\big)\cap C\big([0,1]\times \Omega_{\epsilon_0}\big)$
of the Rayleigh equation (\ref{eq:Rayleigh-H}). Moreover, there exists $\epsilon_0>0$ such that for any $\epsilon\in[0,\epsilon_0)$ and $(y,c)\in [0,1]\times \Om_{\epsilon_0}$,
\beno
&&|\phi_{1}(y,c)|\ge \f12,\quad |\phi_{1}(y,c)-1|\leq C |u(y)-c|^2,
\eeno
where $\phi_1(y,c)=\frac {\phi(y,c)} {u(y)-c}$, and  the constant $C$ may depend on $\al$.
\end{proposition}

The proof is based the following lemmas.

\begin{lemma}\label{lem:Ray-D}
Let $c\in D_{\epsilon_0}$ and $y_c=u^{-1}(c_r)$.
Then there exists a solution $\phi(x,c)\in Y$ to the Rayleigh equation
\beno
\left\{
\begin{array}{l}
\phi''-\al^2\phi-\frac{u''}{u-c}\phi=0,\\
\frac{\phi(y_c,c)}{u(y_c)-c}=1,\quad \Big(\frac{\phi(y,c)}{u(y)-c}\Big)'\big|_{y=y_c}=0.
\end{array}\right.
\eeno
Moreover, there holds
\begin{align*}
\|\phi\|_{Y}\leq C,
\end{align*}
where the constant $C$ may depend on $\al$.
\end{lemma}

\begin{proof}
Let $\phi_{1}=\frac{\phi}{u-c}$, then
$\phi_{1}$ satisfies
\beno
\big((u-c)^2\phi_{1}'\big)'=\al^2\phi_{1}(u-c)^2,
\eeno
from which, we infer that
\beno
\phi_{1}(y,c)=1+\int_{y_c}^y\frac{\al^2}{(u(y')-c)^2}\int_{y_c}^{y'}\phi_{1}(z,c)(u(z)-c)^2dzdy'.
\eeno
This means that $\phi_{1}(y,c)$ satisfies
\beno
\phi_{1}(y,c)=1+\al^2T\phi_{1}(y,c).
\eeno
It follows from Lemma \ref{lem:T-bound} that the operator $I-\al^2T$ is invertible in the space $Y$. Thus,
\beno
\phi_{1}(y,c)=(I-\al^2T)^{-1}1,
\eeno
with the bound $\|\phi_1\|_Y\le C$, hence $\|\phi\|_Y\le C$.
\end{proof}

In a similar way as in Lemma \ref{lem:Ray-D}, we can show that

\begin{lemma}\label{Lem:Ray-Bl}
Let $c\in B^l_{\epsilon_0}$.
Then there exists a solution $\phi(x,c)\in Y_l$ to the Rayleigh equation
\beno
\left\{
\begin{array}{l}
\phi''-\al^2\phi-\frac{u''}{u-c}\phi=0,\\
\frac{\phi(0,c)}{u(0)-c}=1,\quad \Big(\frac{\phi(y,c)}{u(y)-c}\Big)'\big|_{y=0}=0.
\end{array}\right.
\eeno
Moreover, there holds
\begin{align*}
\|\phi\|_{Y_l}\leq C,
\end{align*}
where the constant $C$ may depend on $\al$.
\end{lemma}

\begin{lemma}\label{Lem:Ray-Br}
Let $c\in B^r_{\epsilon_0}$.
Then there exists a solution $\phi(x,c)\in Y_r$ to the Rayleigh equation
\beno
\left\{
\begin{array}{l}
\phi''-\al^2\phi-\frac{u''}{u-c}\phi=0,\\
\frac{\phi(1,c)}{u(1)-c}=1,\quad \Big(\frac{\phi(y,c)}{u(y)-c}\Big)'\big|_{y=1}=0.
\end{array}\right.
\eeno
Moreover, there holds
\begin{align*}
\|\phi\|_{Y_r}\leq C,
\end{align*}
where the constant $C$ may depend on $\al$.
\end{lemma}

\begin{lemma}\label{lem:Ray-D0}
Let $c\in D_0$ and $y_c=u^{-1}(c)$.
Then there exists a solution $\phi(y,c)\in Y_C$ to the Rayleigh equation
\beno
\left\{
\begin{array}{l}
\phi''-\al^2\phi-\frac{u''}{u-c}\phi=0,\\
\phi(y_c,c)=0,\quad \phi'(y_c,c)=u'(y_c).
\end{array}\right.
\eeno
Moreover, there holds
\begin{align*}
\|\phi\|_{Y_C}\leq C,
\end{align*}
where the constant $C$ may depend on $\al$.
\end{lemma}

\begin{proof}
We rewrite the Rayleigh equation (\ref{eq:Rayleigh-H}) as
\beno
(\phi'(u-c)-\phi u')'=\al^2\phi(u-c).
\eeno
Using the boundary conditions $\phi(x_0)=0$ and $\phi'(x_0)=u'(x_0)$, we get
\beno
\phi'(y)(u(y)-c)-\phi(y) u'(x)=\int_{y_c}^y\al^2\phi(z)(u(z)-c)dz,
\eeno
which implies
\beno
\bigg(\frac{\phi(y,c)}{u(y)-u(y_c)}\bigg)'
=\frac{\al^2}{(u(y)-u(y_c))^2}\int_{y_c}^y\phi(z)(u(z)-c)dz.
\eeno
Let $\phi_1(y,c)=\frac{\phi(y,c)}{u(y)-c}$. So, $\phi_1(y_c)=1$ and
\beno
\phi_1(y,c)=1+\int_{y_c}^y\frac{\al^2}{(u(y')-u(y_c))^2}\int_{y_c}^{y'}\phi_1(z,c)(u(z)-u(y_c))^2dzdy',
\eeno
that is,
\ben\label{eq:Ray-inte}
\phi_1(y,c)=1+\al^2T\phi_1(y,c).
\een
Then Lemma \ref{lem:T-bound} ensures the existence of the solution in $Y_C$ to the equation (\ref{eq:Ray-inte})
with the bound $\|\phi_1\|_{Y_C}\le C$.
\end{proof}
\begin{remark}
\label{rem:Ray-positive}
Since $T$ is a positive operator, $\phi_1(y,c)$ for $c\in D_0$ given by (\ref{eq:Ray-inte}) satisfies
\beno
\phi_1(y,c)\geq 1.
\eeno
Moreover, the inverse of  $I-\al^2T$ exists in $X_0$ and has the infinite series representation
\beno
\big(I-\al^2T\big)^{-1}=\sum_{k=0}^\infty\al^{2k}T^k.
\eeno
\end{remark}

Now we are in a position to prove Proposition \ref{prop:Rayleigh-Hom}. We define
\beno
\phi(y,c)\eqdef\left\{
\begin{array}{l}
\phi^0(y,c)\quad \textrm{for } c\in D_0,\\
\phi^\pm(y,c)\quad \textrm{for } c\in D_{\epsilon_0},\\
\phi^l(y,c)\quad \textrm{for } c\in B_{\epsilon_0}^l,\\
\phi^r(y,c)\quad \textrm{for } c\in B_{\epsilon_0}^r,
\end{array}
\right.
\eeno
where $\phi^\pm$, $\phi^l, \phi^r, \phi^0$ are given by Lemma \ref{lem:Ray-D}, Lemma \ref{Lem:Ray-Bl}, Lemma \ref{Lem:Ray-Br} and Lemma \ref{lem:Ray-D0} respectively.

By our constructions, $\phi(y,c)\in C^1\big([0,1]\times \Omega_{\epsilon_0}\setminus D_0\big)$.
Notice that
$$Tf(y,c)=(y-y_c)^2\iint_{[0,1]^2}f(y_c+st(y-y_c),c)
K_0(s,t,y,c)dtds,$$
where
$$
K_0(s,t,y,c)=s\left(\frac{u(y_c+st(y-y_c),c)-c}{u(y_c+s(y-y_c),c)-c}\right)^2.
$$
Using the fact that $K_0\in C([0,1]^3\times\Omega_{\varepsilon_0})$ and $|K_0|\leq s$, we conclude that $T$
maps $C([0,1]\times\Omega_{\varepsilon_0})$ to $C([0,1]\times\Omega_{\varepsilon_0})$. Then by using the formula
$\phi_{1}(y,c)=\dfrac{}{}\sum\limits_{k=0}^{+\infty}\al^{2k}T^{k}1$
for $c\in \Omega_{\varepsilon_0}$ and that the convergence is uniform,
we conclude that $\phi_{1}(y,c)\in C([0,1]\times\Omega_{\varepsilon_0})$, thus $\phi(y,c)\in C\big([0,1]\times \Omega_{\epsilon_0}\big)$.
Furthermore, Remark \ref{rem:Ray-positive} ensures that
there exists $\epsilon_0>0$ such that for any $\epsilon\in[0,\epsilon_0)$ and $(y,c)\in [0,1]\times \Om_{\epsilon_0}$,
\beno
|\phi_1(y,c)|\ge \f12.
\eeno


Using the integral representation of $\phi_1(y,c)$, we have
\begin{align*}
|\phi_{1}(y,c)-1|\leq &\al^2\int_{y_c}^y\int_{y_c}^{z}|\phi_{1}(y',c)|\big|\frac{u(y')-c}{u(z)-c}\big|^2dy'dz\\
\leq & C|y-y_c|^2
\leq  C|u(y)-u(y_c)|^2\\
\leq & C|u(y)-c|^2.
\end{align*}
This completes the proof of Proposition \ref{prop:Rayleigh-Hom}.\ef

\section{Uniform estimates for the homogeneous Rayleigh equation}

The goal of this section is to establish some uniform estimates in wave number $\al$ for the solution $\phi(y,c)$ for $c\in D_0$
of the Rayleigh equation given by Lemma \ref{lem:Ray-D0}. Let $\phi_1(y,c)=\frac {\phi(y,c)} {u(y)-c}$ and $y_c=u^{-1}(c)$ for $c\in D_0$.

Without loss of generality, we always assume $\al\ge 1$ in the sequel.


\subsection{Uniform estimates in wave number}

\begin{proposition}\label{prop:phi}
There exists a constant $C$ independent of $\al$ such that
for any $(y,c)\in [0,1]\times D_0$,
\beno
\frac{\sinh \al(y-y_c)}{C\al}\leq \phi(y,c)\leq \frac{C\sinh \al(y-y_c)}{\al},\\
\frac{\sinh \al(y-y_c)}{C\al(y-y_c)}\leq \phi_1(y,c)\leq \frac{C\sinh \al(y-y_c)}{\al(y-y_c)}.
\eeno
Moreover, it holds that
\beno
|\partial_y^{\b}\partial_c^{\gamma}\phi_1(y,c)|\leq C \al^{\b+\g}\phi_1(y,c)
\eeno
for  $\b+\gamma\le 2$.
\end{proposition}

The proof of the proposition needs the following two lemmas.

\begin{lemma}\label{lem:S-bound}
Let the operator $S$ be defined by
\beno
Sf(y,c)\eqdef\int_{y_c}^y\frac{dy'}{\sinh^2 \al(y'-y_c)}\int_{y_c}^{y'}\frac{\sinh^2 \al(z-y_c)u''(z)}{u(z)-u(y_c)}f(z,c)dz.
\eeno
Then there exists a constant $C$ independent of $\al$ such that
\beno
\|Sf\|_{L^{\infty}([0,1]\times D_0)}\leq C\frac{1+\ln \al}{\al}\|f\|_{L^{\infty}([0,1]\times D_0)}.
\eeno
Especially, there exists $M_0>0$ such that for $\al\ge M_0$,
\beno
\|Sf\|_{L^{\infty}([0,1]\times D_0)}\leq \f12\|f\|_{L^{\infty}([0,1]\times D_0)}.
\eeno
\end{lemma}
\begin{proof}
A direct calculation shows
\begin{align*}
\|Sf\|_{L^{\infty}([0,1]\times D_0)}\leq
&\bigg|\int_{y_c}^y\frac{dy'}{\sinh^2 \al(y'-y_c)}\int_{y_c}^{y'}\frac{\sinh^2 \al(z-y_c)u''(z)}{u(z)-u(y_c)}f(z,c)dz\bigg|\\
=&\bigg|\int_{y_c}^y\frac{dy}{\sinh^2 \al(y'-y_c)}\int_{y_c}^{y'}\frac{\sinh^2 \al(z-y_c)}{z-y_c}\frac{u''(z)}{\int_0^1u'(y_c+(z-y_c)t)dt}f(z,c)dz\bigg|\\
\leq &C\frac{1}{\al}\int_0^{\al|y-y_c|}\frac{dy}{\sinh ^2y'}\int_{0}^{|y'|}\frac{\sinh^2 z}{z}dz\|f\|_{L^{\infty}([0,1]\times D_0)}\\
\leq &C\frac{1}{\al}\int_0^{\al}\frac{\sinh^2 z}{z}dz\int_{|z|}^{\infty}\frac{dy'}{\sinh ^2y'}\|f\|_{L^{\infty}([0,1]\times D_0)}\\
=&C\frac{1}{\al}\int_0^{\al}\frac{1-e^{-2z}}{z}dz\|f\|_{L^{\infty}([0,1]\times D_0)}\\
\leq& C\frac{1+\ln \al}{\al}\|f\|_{L^{\infty}([0,1]\times D_0)}.
\end{align*}
The second statement is direct.
\end{proof}

\begin{lemma}\label{lem:T-sigma}
Given $\al\ge 1$, we denote $\Sigma_{\al}\triangleq\big\{(y,c):~|y-y_c|\leq \frac{1}{\al}, c=u(y_c)\big\}$.
Then there exists a constant $C$ independent of $\al$ such that
\beno
&&\|T_{0}f\|_{L^{\infty}(\Sigma_{\al})}\leq \frac{1}{\al}\|f\|_{L^{\infty}(\Sigma_{\al})},\\
&&\|T_{k,k}f\|_{L^{\infty}(\Sigma_{\al})}\leq \frac{C}{\al}\|f\|_{L^{\infty}(\Sigma_{\al})}.
\eeno
Moreover, for $\b+\g\leq 2$,
\beno
\|\partial_{y}^{\b}\partial_c^{\g}T_{k,k+1}f\|_{L^{\infty}(\Sigma_{\al})}\leq C\sum_{\b_1+\g_1\le \b+\g}\|\partial_x^{\b_1}\partial_c^{\g_1}f\|_{L^{\infty}(\Sigma_{\al})}.
\eeno
\end{lemma}
\begin{proof}
For $(y,c)\in \Sigma_{\al}$, we have
\begin{align*}
|T_0f(y,c)|=&\bigg|\int_{y_c}^{y}f(z,c)dz\bigg|\leq |y-y_c|\|f\|_{L^{\infty}(\Sigma_{\al})}
\leq  \frac{1}{\al}\|f\|_{L^{\infty}(\Sigma_{\al})},
\end{align*}
and
\begin{align*}
|T_{k,k}f(y,c)|
=&\bigg|\frac{1}{(u(y)-c)^k}
\int_{y_c}^y(u(z)-c)^kf(z,c)dz\bigg|\\
\leq &C|y-y_c|\|u'\|_{L^{\infty}}^k\|f\|_{L^{\infty}(\Sigma_{\al})}
\leq \frac{C}{\al}\|f\|_{L^{\infty}(\Sigma_{\al})}.
\end{align*}
Using the formula
\beno
T_{k,k+1}f(y,c)=\int_0^1f(y_c+t(y-y_c),c)\frac{\big(\int_0^1u'(y_c+st(y-y_c))ds\big)^k}{\big(\int_0^1u'(y_c+s(y-y_c))ds\big)^{k+1}}t^kdt,
\eeno
a direct calculation gives the third inequality of the lemma. We omit the details here.
\end{proof}

Now let us turn to the proof of Proposition \ref{prop:phi}.

\begin{proof}

Let $M_0$ be given by Lemma \ref{lem:S-bound}.
The case of $\al\le M_0$ is a direct consequence of Lemma \ref{lem:Ray-D0}, which implies that
\beno
&1\leq |\phi_1(y,c)|\leq C,\quad C^{-1}|y-y_c|\leq |\phi(y,c)|\leq C|y-y_c|,\quad\big|\partial_y^{\b}\partial_{c}^{\gamma}\phi_1(y,c)\big|\leq C
\eeno
for any $(y,c)\in [0,1]\times D_0$, where the constant $C$ depends on $M_0$.

Thus, it suffices to consider the case of $\al\ge M_0$. It is easy to check that the solution $\phi(y,c)$ satisfies
\beno
\Big(\sinh ^2\al(y-y_c)\big(\frac{\phi(y,c)}{\sinh \al(y-y_c)}\big)'\Big)'=\sinh \al(y-y_c)\frac{u''}{u(y)-u(y_c)}\phi(y,c).
\eeno
Let $\tphi_1(y,c)=\frac{\al\phi(y,c)}{\sinh \al(y-y_c)}$, then $\tphi_1(y,c)$ satisfies
\beno
\tphi_1(y,c)=u'(y_c)+\int_{y_c}^y\frac{dy'}{\sinh^2 \al(y'-y_c)}\int_{y_c}^{y'}\frac{\sinh^2 \al(z-y_c)u''(z)}{u(z)-u(y_c)}\tphi_1(z,c)dz.
\eeno
Lemma \ref{lem:S-bound} ensures that
\beno
\tphi_1(y,c)=(I-S)^{-1}(u'(y_c))
\eeno
with the bound
\beno
C^{-1}\le \tphi_1(y,c)\le C\quad \textrm{for any}\quad (y,c)\in [0,1]\times D_0,
\eeno
which implies the first part of the proposition.

Next we prove the derivative estimates. Thanks to $u'(y)\ge c_0$, it holds that for any $y,z$ with $|y-y_c|\leq |z-y_c|$,
\beno
(u(y)-c)^2\leq (u(z)-c)^2,\quad c=u(y_c).
\eeno
Thus, for $(y,c)\in \Sigma_{\al}$,
\begin{align*}
|Tf(y,c)|\leq&\int_{y_c}^y\frac{1}{(u(z)-c)^2}\int_{y_c}^{z}(u(y')-c)^2\|f\|_{L^{\infty}(\Sigma_{\al})}dy'dz\\
\leq & \int_{y_c}^y(z-y_c)dz\|f\|_{L^{\infty}(\Sigma_{\al})}\\
\leq & \frac{1}{2\al^2}\|f\|_{L^{\infty}(\Sigma_{\al})}.
\end{align*}
This means that $\al^2T$ is a contraction mapping from $L^{\infty}(\Sigma_{\al})$ to $L^{\infty}(\Sigma_{\al})$. Thus,
\beno
\|\phi_1\|_{L^{\infty}(\Sigma_{\al})}\leq \|(I-\al^2T)^{-1}1\|_{L^{\infty}(\Sigma_{\al})}\leq 2.
\eeno
On the other hand, we have
\beno
&&\partial_{y}\phi_1(y,c)=\al^2T_{2,2}\phi_1(y,c),\\
&&\partial_{yy}\phi_1(y,c)=-2\al^2u'(y)T_{2,3}\phi_1(y,c)+\phi(y,c),
\eeno
then it follows from Lemma \ref{lem:T-sigma} that for any $(y,c)\in \Sigma_{\al}$,
\begin{align*}
|\partial_{y}\phi_1(y,c)|\leq C\al,\quad |\partial_{yy}\phi_1(y,c)|\leq C\al^2.
\end{align*}

Now we deal with the derivative estimate in $c$ variable. Recalling that
\beno
&&\phi_1(y,c)=1+\al^2T\phi_1(y,c),\\
&&\partial_{c}Tf=2T_0T_{2,3}f-2T_0T_{1,2}f+T\partial_cf,
\eeno
we obtain
\begin{align*}
&\partial_c\phi_1(y,c)=2\al^2T_0T_{2,3}\phi_1(y,c)-2\al^2T_0T_{1,2}\phi_1(y,c)+\al^2T\partial_c\phi_1(y,c),\\
&\partial_y\partial_c\phi_1(y,c)=2\al^2T_{2,3}\phi_1(y,c)-2\al^2T_{1,2}\phi_1(y,c)+\al^2T_{2,2}\partial_c\phi_1(y,c),
\end{align*}
and
\begin{align*}
\partial_c^2\phi_1(y,c)=&2\al^2\pa_cT_0T_{2,3}\phi_1(y,c)-2\al^2\pa_cT_0T_{1,2}\phi_1(y,c)+\al^2\partial_cT\partial_c\phi_1(y,c)\\
=&2\al^2T_0\partial_cT_{2,3}\phi_1(y,c)-2\al^2T_0\partial_cT_{1,2}\phi_1(y,c)+\frac{\al^2}{3u'(y_c)}\\
&+2\al^2T_0T_{2,3}\partial_c\phi_1(y,c)-2\al^2T_0T_{1,2}\partial_c\phi_1(y,c)+\al^2T\partial_c^2\phi_1(y,c),
\end{align*}
which along with Lemma \ref{lem:T-sigma} ensure that
\begin{align*}
\|\partial_c\phi_1\|_{L^{\infty}(\Sigma_{\al})}=&
\|(I-\al^2T)^{-1}
(2\al^2T_0T_{2,3}\phi_1(y,c)-2\al^2T_0T_{1,2}\phi_1)\|_{L^{\infty}(\Sigma_{\al})}\\
\leq & C\big(\al^2\|T_0T_{2,3}\phi_1\|_{L^{\infty}(\Sigma_{\al})}
+\al^2\|T_0T_{1,2}\phi_1\|_{L^{\infty}(\Sigma_{\al})}\big)\\
\leq &C\al\|\phi_1\|_{L^{\infty}(\Sigma_{\al})}\leq C\al,
\end{align*}
and
\begin{align*}
\|\partial_y\partial_c\phi_1\|_{L^{\infty}(\Sigma_{\al})}+\|\partial_c^2\phi_1\|_{L^{\infty}(\Sigma_{\al})}\leq C\al^2.
\end{align*}

Finally, let us consider the case of $|y-y_c|\geq \frac{1}{\al}$.
Using the fact that
\beno
&&|\phi_1(y,c)|\leq C\frac{\sinh \al(y-y_c)}{\al(y-y_c)}\leq C\cosh\al(y-y_c),\\
&&C\sinh \al |y-y_c|\geq \cosh\al(y-y_c)-1\quad for\quad |y-y_c|\geq \frac{1}{\al},
\eeno
we infer that
\begin{align*}
|\partial_{y}\phi_1(y,c)|=&\bigg|\frac{\al^2}{(u(y)-u(y_c))^2}\int_{y_c}^y\phi_1(z)(u(z)-u(y_c))^2dz\bigg|\\
\leq &C\bigg|\frac{\al}{(y-y_c)^2}\int_{y_c}^y\sinh\al(z-y_c)(z-y_c)dz\bigg|\\
\leq &C\frac{|\cosh\al(y-y_c)-1|}{|y-y_c|}
\leq C\al\phi_1(y,c).
\end{align*}
Using the equation
\beno
\phi_1''+\frac{2u'}{u-c}\phi_1'=\al^2\phi_1,
\eeno
we obtain
\beno
|\partial_y^2\phi_1|\leq C\al^2 \phi_1.
\eeno

To estimate the derivative in $c$ variable, we introduce $f=\frac{\partial_y\phi_1}{\phi_1}$ and $g=\frac{\partial_c\phi_1}{\phi_1}$.
It is easy to see that $\partial_cf=\partial_yg$ and
\beno
(\phi^2\partial_cf)'+2\phi_1^2u'f=0,
\eeno
which gives
\beno
\partial_cf(y,c)=\partial_yg(y,c)=\frac{-2}{\phi(y,c)^2}\int_{y_c}^y\phi_1^2(z,c)u'(z)f(z,c)dz=\frac{-2}{\phi(y,c)^2}T_0(\phi_1^2u'f),
\eeno
For $|y-y_c|\geq \frac{1}{\al}$, $\phi_1^2-1\sim \phi_1^2$ and $\partial_y\phi_1(y,c)$ has the same sign as $y-y_c$. Then
\begin{align*}
|\partial_yg(y,c)|=|\partial_cf(y,c)|\leq& \frac{C}{\phi(y,c)^2}\int_{y_c}^y\phi_1(z,c)\partial_z\phi_1(z,c)dz\\
\leq& C\frac{\phi_1(y,c)^2-1}{\phi(y,c)^2}\leq\frac{C}{|y-y_c|^2},
\end{align*}
from which, it follows that
\beno
|g(y,c)|\leq \frac{C}{|y-y_c|}\leq C\al,
\eeno
and thus,
\beno
|\partial_c\phi_1|\leq C\al \phi_1,\quad |\partial_y\partial_c\phi_1|\leq C\al^2 \phi_1.
\eeno
We have
\beno
\partial_c^2f(y,c)=\partial_y\partial_cg(y,c)
=-2\int_{y_c}^y\partial_c\Big(\frac{\phi_1(z,c)^2}{\phi(y,c)^2}f(z,c)\Big)u'(z)dz,
\eeno
which implies
\begin{align*}
|\partial_c^2f(y,c)|\leq \frac{C\al}{|y-y_c|^2}.
\end{align*}
Thus, we obtain
\beno
|\partial_cg(y,c)|\leq \frac{C\al}{|y-y_c|}\leq C\al^2,
\quad |\partial_c^2\phi_1|\leq C\al^2\phi_1.
\eeno
The proof is finished.
\end{proof}

\begin{lemma}\label{lem:phi1-est}
There exists a constant $C$ independent of $\al$ such that
for any $(y,c)\in [0,1]\times D_0$,
\beno
&\frac{1}{C}\Big(\frac{\sinh \al(y-y_c)}{\al(y-y_c)}-1\Big)\leq \phi_1(y,c)-1\leq C\Big(\frac{\sinh \al(y-y_c)}{\al(y-y_c)}-1\Big),\\
&|\phi_1(y,c)-1|\leq C\al^2(y-y_c)^2\phi_1(y,c),
\eeno
where $c=u(y_c)$.
\end{lemma}

\begin{proof} We have
\begin{align*}
\phi_1(y,c)-1=&\int_{y_c}^y\frac{\al^2}{(u(y')-u(y_c))^2}\int_{y_c}^{y'}\phi_1(z)(u(z)-u(y_c))^2dzdy'\\
\geq&C^{-1}\int_{y_c}^y\frac{\al}{(y'-y_c)^2}\int_{y_c}^{y'}\sinh \al(z-y_c)(z-y_c)dzdy'\\
=&C^{-1}\int_{0}^{\al(y-y_c)}\frac{\cosh z}{z}-\frac{\sinh z}{z^2}dz\\
=&C^{-1}\Big(\frac{\sinh \al(y-y_c)}{\al(y-y_c)}-1\Big)\geq 0,
\end{align*}
and
\begin{align*}
\phi_1(y,c)-1=&\int_{y_c}^y\frac{\al^2}{(u(y')-u(y_c))^2}\int_{y_c}^{y'}\phi_1(z)(u(z)-u(y_c))^2dzdy'\\
\leq&C\int_{y_c}^y\frac{\al}{(y'-y_c)^2}\int_{y_c}^{y'}\sinh \al(z-y_c)(z-y_c)dzdy'\\
=&C\int_{0}^{\al(y-y_c)}\frac{\cosh z}{z}-\frac{\sinh z}{z^2}dz\\
=&C\Big(\frac{\sinh \al(y-y_c)}{\al(y-y_c)}-1\Big).
\end{align*}
This gives the first inequality.

Thanks to $\pa_y\phi_1(y,c)\geq0$ for $y\geq y_c$ and $\pa_y\phi_1(y,c)\leq0$ for $y\leq y_c$, we have
\beno
\phi_1(y,c)\geq \phi_1(z,c),\quad |u(z)-c|\leq |u(y)-c|
\eeno
for $|y-y_c|\geq |z-y_c|$.
Then we get
\begin{align*}
|\phi_1(y,c)-1|\leq& \al^2\int_{y_c}^y\int_{y_c}^{z}\phi_1(y',c)\frac{(u(y')-c)^2}{(u(z)-c)^2}dy'dz\\
\leq & C\al^2|y-y_c|^2\phi_1(y,c)\leq C\al^2|u(y)-u(y_c)|^2\phi_1(y,c).
\end{align*}
The second inequality is proved.
\end{proof}

\begin{remark}\label{rem:phi1}
We also need a more precise estimate for $\phi_1(y,c)-1$:
\begin{align*}
\phi_1(y,c)-1=&\al^2(y-y_c)^2\iint_{[0,1]^2}\phi_1(y_c+st(y-y_c),c)
K(y_c+s(y-y_c),y_c,t)sdtds\\
=&\al^2(y-y_c)^2\mathcal{T}(\phi_1)(y,c),
\end{align*}
where $y_c=u^{-1}(c)$ and
\beno
K(y,y_c,t)=\frac{\big(\int_0^1u'(y_c+st(y-y_c))ds\big)^2}{\big(\int_0^1u'(y_c+s(y-y_c))ds\big)^2}t^2.
\eeno
Using the fact that $\|K\|_{C^2_{y,y_c,t}}\leq C$ and $\phi_1(y,c)\geq \phi_1(z,c)$ for $|y-y_c|\geq |z-y_c|$, we get by Proposition \ref{prop:phi} that
\beno
&&|\mathcal{T}(\phi_1)(y,c)|\leq C\phi_1(y,c),\\
&&|\partial_y\mathcal{T}(\phi_1)(y,c)|+|\partial_{c}\mathcal{T}(\phi_1)(y,c)|\leq C\al\phi_1(y,c),\\
&&|\partial_y^2\mathcal{T}(\phi_1)(y,c)|
+|\partial_{c}^2\mathcal{T}(\phi_1)(y,c)|
+|\partial_{yc}\mathcal{T}(\phi_1)(y,c)|\leq C\al^2\phi_1(y,c).
\eeno

\end{remark}

\subsection{Uniform estimates for good derivative}

\begin{proposition}\label{prop:phi-good}
There exists a constant $C$ independent of $\al$ such that
for any $(z,c)\in D_0\times D_0$,
\begin{align*}
&\big|(\partial_z+\partial_c)\phi_1(u^{-1}(z),c)\big|\leq C\al^2(y-y_c)^2\phi_1,\\
&|(\partial_z+\partial_c)^2\phi_1(u^{-1}(z),c)|\leq
C\al^2(y-y_c)^2\cosh\al(y-y_c),
\end{align*}
where $y_c=u^{-1}(c)$ and $y=u^{-1}(z)$.
\end{proposition}

\begin{proof}
A direct calculation gives
\begin{align*}
&(\partial_z+\partial_c)\phi_1(u^{-1}(z),c)\\
=&\frac{\al^2}{(z-c)^2}\int_{y_c}^{u^{-1}(z)}(u(y')-c)^2\phi_1(y',c)dy'\frac{1}{u'(u^{-1}(z))}\\
&+2\al^2\int_{y_c}^{u^{-1}(z)}\frac{1}{(u(y')-c)^3}\int_{y_c}^{y'}(u(t)-c)^2\phi_1(t,c)dtdy'\\
&-2\al^2\int_{y_c}^{u^{-1}(z)}\frac{1}{(u(y')-c)^2}\int_{y_c}^{y'}(u(t)-c)\phi_1(t,c)dtdy'\\
&+\al^2T(\partial_c\phi_1)(u^{-1}(z),c)\\
=&\frac{\al^2}{(z-c)^2}\int_{y_c}^{u^{-1}(z)}(u(y')-c)^2\phi_1(y',c)dy'\frac{1}{u'(u^{-1}(z))}
+\al^2T(\partial_c\phi_1)(u^{-1}(z),c)\\
&-\al^2\int_{y_c}^{u^{-1}(z)}\partial_{y'}\big(\frac{1}{(u(y')-c)^2}\big)\frac{1}{u'(y')}\int_{y_c}^{y'}(u(t)-c)^2\phi_1(t,c)dtdy'\\
&-\al^2\int_{y_c}^{u^{-1}(z)}\frac{1}{(u(y')-c)^2}\int_{y_c}^{y'}\partial_t\big((u(t)-c)^2\big)\frac{\phi_1(t,c)}{u'(t)}dtdy',
\end{align*}
Integration by parts and taking into consideration the boundary condition, we obtain
\begin{align*}
&(\partial_z+\partial_c)\phi_1(u^{-1}(z),c)\\
=&\frac{\al^2}{(z-c)^2}\int_{y_c}^{u^{-1}(z)}(u(y')-c)^2\phi_1(y',c)dy'\frac{1}{u'(u^{-1}(z))}
+\al^2T(\partial_c\phi_1)(u^{-1}(z),c)\\
&-\al^2\big(\frac{1}{(u(y')-c)^2}\big)\frac{1}{u'(y')}\int_{y_c}^{y'}(u(z')-c)^2\phi_1(z',c)dz'\bigg|_{y_c}^{u^{-1}(z)}\\
&+\al^2\int_{y_c}^{u^{-1}(z)}\big(\frac{1}{(u(y')-c)^2}\big)\partial_{y'}\Big(\frac{1}{u'(y')}\int_{y_c}^{y'}(u(z')-c)^2\phi_1(z',c)dz'\Big)dy'\\
&-\al^2\int_{y_c}^{u^{-1}(z)}\frac{1}{(u(y')-c)^2}
(u(z')-c)^2\frac{\phi_1(z',c)}{u'(z')}\bigg|_{y_c}^{y'}dy'\\
&+\al^2\int_{y_c}^{u^{-1}(z)}\frac{1}{(u(y')-c)^2}\int_{y_c}^{y'}(u(z')-c)^2\partial_y\Big(\frac{\phi_1(z',c)}{u'(z')}\Big)dz'dy'\\
=&\al^2T(\partial_c\phi_1)(u^{-1}(z),c)
+\al^2\int_{y_c}^{u^{-1}(z)}\frac{1}{(u(y')-c)^2}
\partial_{y'}\Big(\frac{1}{u'(y')}\Big)\int_{y_c}^{y'}(u(z')-c)^2\phi_1(z',c)dz'dy'\\
&+\al^2\int_{y_c}^{u^{-1}(z)}\frac{1}{(u(y')-c)^2}\int_{y_c}^{y'}(u(z')-c)^2\partial_y\Big(\frac{\phi_1(z',c)}{u'(z')}\Big)dz'dy'\\
=&\al^2T(\partial_c\phi_1)(u^{-1}(z),c)
+\al^2T_0\Big(\partial_{y}\Big(\frac{1}{u'}\Big)T_{2,2}(\phi_1)\Big)(u^{-1}(z),c)\\
&+\al^2T\Big(\partial_y\Big(\frac{\phi_1}{u'}\Big))(u^{-1}(z),c)\\
=&\al^2T\Big(\partial_c\phi_1+\frac{1}{u'}\partial_y\phi_1\Big)(u^{-1}(z),c)
+\al^2T_0\Big(\partial_y\Big(\frac{1}{u'}\Big)T_{2,2}(\phi_1)\Big)(u^{-1}(z),c)\\
&+\al^2T\Big(\phi_1\partial_y\Big(\frac{1}{u'}\Big)\Big)(u^{-1}(z),c).
\end{align*}
This shows that
\begin{align*}
&(I-\al^2 T)\big(\partial_z\phi_1+\partial_c\phi_1\big)(u^{-1}(z),c)\\
&=\al^2T_0\Big(\partial_y\Big(\frac{1}{u'}\Big)T_{2,2}(\phi_1)\Big)(u^{-1}(z),c)
+\al^2T\Big(\phi_1\partial_y\Big(\frac{1}{u'}\Big)\Big)(u^{-1}(z),c),
\end{align*}
where the right hand side is bounded by $\al^2T\phi_1$.
Recall that $T$ is a positive operator. Thus, we have
\begin{align}
&\big|(\partial_z+\partial_c)\phi_1(u^{-1}(z),c)\big|\nonumber\\
&\leq \bigg|(I-\al^2 T)^{-1}
\bigg[\al^2T_0\Big(\partial_y\Big(\frac{1}{u'}\Big)T_{2,2}(\phi_1)\Big)(u^{-1}(z),c)\nonumber\\
&\quad+\al^2T\Big(\phi_1\partial_y\Big(\frac{1}{u'}\Big)\Big)(u^{-1}(z),c)
\bigg]\bigg|\nonumber\\
&\leq C\Big|(I-\al^2 T)^{-1} \al^2T\phi_1\Big|\nonumber\\
&\leq C\Big|\al^2T(I-\al^2 T)^{-2} 1\Big|
\leq C\Big|\sum_{k=1}^{\infty}k\al^{2k}T^{k}1\Big|.\label{eq:phi1-gd}
\end{align}

Using the fact that  $(u(z)-c)^2\leq (u(y')-c)^2$ for $|z-y_c|\leq |y'-y_c|$,  we obtain
\begin{align*}
|T1|=&
\bigg|\int_{y_c}^y\frac{1}{(u(y')-c)^2}\int_{y_c}^{y'}(u(z)-c)^2dzdy'\bigg|\\
\leq &\frac{1}{2}(y-y_c)^2.
\end{align*}
A simple iteration gives
\beno
|T^k1|\leq \frac{1}{(2k)!}(y-y_c)^{2k},
\eeno
which implies
\begin{align*}
\Big|\sum_{k=1}^{\infty}k\al^{2k}T^{k}1\Big|\leq&
C\al(y-y_c)\sum_{k=1}^{\infty} \frac{1}{(2k-1)!}(\al(y-y_c))^{2k-1}\\
\leq& C\al(y-y_c)\sinh\al(y-y_c)\\
\le& C\al^2(y-y_c)^2\phi_1,
\end{align*}
which together with (\ref{eq:phi1-gd}) gives the first inequality.

Now we prove the second inequality.  We have
\begin{align*}
&(\partial_z+\partial_c)^2\phi_1(u^{-1}(z),c)\\
&=\al^2(\partial_z+\partial_c)T(\partial_c\phi_1)(u^{-1}(z),c)
+\al^2(\partial_z+\partial_c)T_0\bigg(\partial_y\Big(\frac{1}{u'}\Big)T_{2,2}(\phi_1)\bigg)(u^{-1}(z),c)\\
&\quad+\al^2(\partial_z+\partial_c)T\Big(\partial_y\Big(\frac{\phi_1}{u'}\Big)\Big)(u^{-1}(z),c).
\end{align*}
Let $F(y,c)$ be a function with $\frac{F(y,c)}{(u(y)-c)^2}\to 0$ as $y\to y_c$.
Then we have
\begin{align*}
&(\partial_z+\partial_c)\int_{y_c}^{u^{-1}(z)}\frac{1}{(u(y')-c)^2}F(y',c)dy'\\
&=\frac{F(u^{-1}(z),c)}{(z-c)^2}(u^{-1})'(z)
-\frac{F(y',c)}{(u(y')-c)^2u'(y')}\bigg|_{y_c}^{u^{-1}(z)}\\
&\quad+\int_{y_c}^{u^{-1}(z)}\frac{1}{(u(y')-c)^2}\bigg(\partial_y\Big(\frac{1}{u'(y')}\Big)F(y',c)\bigg)dy'
+\int_{y_c}^{u^{-1}(z)}\frac{1}{(u(y')-c)^2}\bigg(\frac{1}{u'(y')}\partial_yF(y',c)\bigg)dy'\\
&\quad+\int_{y_c}^{u^{-1}(z)}\frac{1}{(u(y')-c)^2}\partial_cF(y',c)dy'\\
&=\int_{y_c}^{u^{-1}(z)}\frac{F(y',c)}{(u(y')-c)^2}\partial_y\Big(\frac{1}{u'(y')}\Big)dy'
+\int_{y_c}^{u^{-1}(z)}\frac{1}{(u(y')-c)^2}\bigg(\frac{1}{u'(y')}\partial_yF(y',c)\bigg)dy'\\
&\quad+\int_{y_c}^{u^{-1}(z)}\frac{1}{(u(y')-c)^2}\partial_cF(y',c)dy'.
\end{align*}
Similarly, we have
\begin{align*}
&(\partial_z+\partial_c)\int_{y_c}^{u^{-1}(z)}f(y',c)(u(y')-c)^2dy'\\
&=f(u^{-1}(z),c)(z-c)^2(u^{-1})'(z)
+\int_{y_c}^{u^{-1}(z)}\partial_cf(y',c)(u(y')-c)^2dy'\\
&\quad-\int_{y_c}^{u^{-1}(z)}\frac{f(y',c)}{u'(y')}\partial_y\big((u(y')-c)^2\big)dy'\\
&=\int_{y_c}^{u^{-1}(z)}\partial_cf(y',c)(u(y')-c)^2dy'
+\int_{y_c}^{u^{-1}(z)}\partial_y\big(\frac{f(y',c)}{u'(y')}\big)(u(y')-c)^2dy'.
\end{align*}
With the above equalities, we can deduce that
\begin{align*}
&(\partial_z+\partial_c)^2\phi_1(u^{-1}(z),c)\\
&=2\al^2\int_{y_c}^{u^{-1}(z)}\partial_y\Big(\frac{1}{u'(y')}\Big)T_{2,2}(\partial_c\phi_1)(y',c)dy'
+\al^2\int_{y_c}^{u^{-1}(z)}T_{2,2}(\partial_c^2\phi_1)(y',c)dy'\\
&\quad+\al^2\int_{y_c}^{u^{-1}(z)}T_{2,2}\phi_1(y',c)\Big(\partial_y(\frac{1}{u'(y')})\Big)^2dy'
+\al^2\int_{y_c}^{u^{-1}(z)}T_{2,2}\phi_1(y',c)\frac{1}{u'(y')}\partial_y^2(\frac{1}{u'(y')})dy'\\
&\quad+2\al^2\int_{y_c}^{u^{-1}(z)}\partial_y\Big(\frac{1}{u'(y')}\Big)T_{2,2}\Big(\partial_y\big(\frac{\phi_1}{u'}\big)\Big)(y',c)dy'
+2\al^2\int_{y_c}^{u^{-1}(z)}T_{2,2}\Big(\partial_c\partial_y\big(\frac{\phi_1}{u'}\big)\Big)(y',c)dy'\\
&\quad+\al^2\int_{y_c}^{u^{-1}(z)}T_{2,2}\Big(\partial_y\Big(\frac{1}{u'}\partial_y\big(\frac{\phi_1}{u'}\big)\Big)\Big)(y',c)dy'.
\end{align*}
Note that
\begin{align*}
&\partial_c^2\phi_1
+\partial_y\Big(\frac{1}{u'}\partial_y\big(\frac{\phi_1}{u'}\big)\Big)
+2\partial_c\partial_y\big(\frac{\phi_1}{u'}\big)\\
&=(\partial_z+\partial_c)^2\phi_1(u^{-1}(z),c)
+2\partial_y\Big(\frac{1}{u'(y)}\Big)(\partial_z+\partial_c)\phi_1(u^{-1}(z),c)\\&\quad+\partial_y^2\Big(\frac{1}{u'(y)}\Big)\frac{\phi_1}{u'(y)}
+\bigg(\partial_y\Big(\frac{1}{u'(y)}\Big)\bigg)^2\phi_1.
\end{align*}
Thus, we obtain
\begin{align*}
&(\partial_z+\partial_c)^2\phi_1(u^{-1}(z),c)\\
&=\al^2T\Big((\partial_z+\partial_c)^2\phi_1(u^{-1}(z),c)\Big)(u^{-1}(z),c)
+2\al^2T_0\Big(\partial_y\Big(\frac{1}{u'(y')}\Big)T_{2,2}((\partial_z+\partial_c)\phi_1(u^{-1}(z),c))\Big)\\
&\quad+\al^2\int_{y_c}^{u^{-1}(z)}T_{2,2}\phi_1(y',c)\Big(\partial_y(\frac{1}{u'(y')})\Big)^2dy'
+\al^2\int_{y_c}^{u^{-1}(z)}T_{2,2}\phi_1(y,c)\frac{1}{u'(y)}\partial_y^2(\frac{1}{u'(y)})dy'\\
&\quad+2\al^2\int_{y_c}^{u^{-1}(z)}\partial_y\Big(\frac{1}{u'(y')}\Big)T_{2,2}\Big(\phi_1\partial_y\big(\frac{1}{u'}\big)\Big)dy'
+2\al^2T\Big(\partial_y\Big(\frac{1}{u'}\Big)(\partial_{z}+\partial_c)\phi_1(u^{-1}(z),c)\Big)\\
&\quad+\al^2T\bigg[\partial_y^2\Big(\frac{1}{u'}\Big)\frac{\phi_1}{u'}
+\bigg(\partial_y\Big(\frac{1}{u'}\Big)\bigg)^2\phi_1\bigg].
\end{align*}
Then by using the estimate
\beno
|(\partial_y+\partial_c)\phi_1(u^{-1}(y),c)|\leq C\big|(I-\al^2T)^{-1}\al^2T\phi_1\big|,
\eeno
and the fact that $T$ is a positive operator, we deduce that
\begin{align*}
&|(\partial_y+\partial_c)^2\phi_1(u^{-1}(y),c)|\\
&\leq C\big|(I-\al^2T)^{-1}\al^2T\phi_1\big|
+C\big|(I-\al^2T)^{-1}\al^2T(I-\al^2T)^{-1}\al^2T\phi_1\big|\\
&\leq C\big|\al^2T(I-\al^2T)^{-2} 1\big|
+C\big|(\al^2T)^2(I-\al^2T)^{-3}1\big|\\
&\leq C\al(y-y_c)\sinh\al(y-y_c)+
C\big|\sum_{k=0}^{\infty}(k+1)(k+2)(\al^2T)^{k+2}1\big|\\
&\leq C\big(\al(y-y_c)\sinh\al(y-y_c)
+\al^2(y-y_c)^2\cosh\al(y-y_c)\big).
\end{align*}
This deduces the second inequality.
\end{proof}

\section{The inhomogeneous Rayleigh equation}
In this section, we consider the inhomogeneous Rayleigh equation
\beq\label{eq:Rayleigh-Ihom}
\left\{
\begin{aligned}
&\Phi''-\al^2\Phi-\frac{u''}{u-c}\Phi=f,\\
&\Phi(0)=\Phi(1)=0.
\end{aligned}
\right.
\eeq

Let $\phi(y,c)$ be the solution of the homogeneous Rayleigh equation given by Proposition \ref{prop:Rayleigh-Hom}, and
$\phi_1(y,c)\triangleq\frac {\phi(y,c)} {u(y)-c}$. We denote
\ben
&&\mathrm{II}_2\triangleq\textrm{p.v.}\int_0^1
\frac{u'(y)-u'(y_c)}{(u(y)-c)^2}dy,\label{eq:Pi-2}\\
&&\mathrm{II}_3\triangleq\int_0^1\frac{1}{(u(y)-c)^2}\Big(\frac{1}{\phi_{1}(y,c)^2}-1\Big)dy,\label{eq:Pi-3}
\een
for $c\in D_0$ and $y_c=u^{-1}(c)$.

The following lemma gives a sufficient and necessary condition so that $\mathcal{R}_\al$ has no embedding eigenvalues.

\begin{lemma}\label{lem:spectrum}
Let $\mathrm{A}(c)=u(0)-u(1)-\rho(c)\mathrm{II}_2+u'(y_c)\rho(c)\mathrm{II}_3$ and $\mathrm{B}(c)=\pi\rho(c)\frac{u''(y_c)}{u'(y_c)^2}$,
where $\rho(c)=(c-u(0))(u(1)-c)$.
Then $c\in D_0$ is an embedding eigenvalue of $\mathcal{R}_\al$ if and only if
\beno
\mathrm{A}(c)^2+\mathrm{B}(c)^2=0.
\eeno
\end{lemma}

\begin{proof}
Let $c=u(y_c)$ be an embedding eigenvalue of  $\mathcal{R}_\al$. We know that $u''(u^{-1}(c))=0$, thus $\mathrm{B}(c)=0$.
Next we show $\mathrm{A}(c)=0$. In such case, $\mathrm{A}(c)=0$ is equivalent to the Wronskian $W[\varphi_1,\varphi_2;c]=0$,
where $\varphi_1(y,c)$ and $\varphi_2(y,c)$ are two linearly independent solutions of the Rayleigh equation
\ben\label{eq:Ray-2}
\varphi''-\al^2\varphi-\frac{u''}{u-c}\varphi=0.
\een

Indeed, thanks to $u''(y_c)=0$,  we can construct a smooth solution $\varphi(y,c)$
of (\ref{eq:Ray-2}) with the boundary conditions $\varphi(0,c)=0$ and $\varphi'(0,c)=1$. Moreover, we have
\beno
W[\varphi_1,\varphi_2;c]=\varphi(1, c).
\eeno
Let $\phi(y,c)$ be a solution of (\ref{eq:Ray-2}) given by Lemma \ref{lem:Ray-D0} and
$\phi_1(y,c)=\frac {\phi(y,c)} {u(y)-c}$. Then $\varphi(y,c)$ has the following representation formula
\beno
\varphi(y,c)=\phi_1(0,c)\overline{\varphi}(y,c),
\eeno
where $\overline{\varphi}(y,c)$ is given by
\begin{align*}
\overline{\varphi}(y,c)=&\frac{\phi_1(y,c)}{u'(y_c)}(u(y)-u(0))\\
&+\frac{\phi_1(y,c)}{u'(y_c)}(u(y)-c)(u(0)-c)\int_0^y\frac{u'(y_c)-u'(z)}{(u(z)-c)^2}dz\\
&+\phi_1(y,c)(u(y)-c)(u(0)-c)\int_0^y\frac{1}{(u(z)-c)^2}\Big(\frac{1}{\phi_1(z,c)^2}-1\Big)dz.
\end{align*}
Then we find that $\varphi(1,c)=0$ is equivalent to
\begin{align*}
0=&\varphi(1,c)=\phi_1(0,c)\overline{\varphi}(y,c)\\
=&\frac{\phi_1(0,c)\phi_1(1,c)}{u'(y_c)}(u(1)-u(0))
-\frac{\phi_1(0,c)\phi_1(1,c)}{u'(y_c)}\rho(c)\int_0^1\frac{u'(y_c)-u'(y)}{(u(y)-c)^2}dy\\
&-\phi_1(0,c)\phi_1(1,c)\rho(c)\int_0^1\frac{1}{(u(y)-c)^2}\Big(\frac{1}{\phi_1(y,c)^2}-1\Big)dy\\
=&-\frac{\phi_1(0,c)\phi_1(1,c)}{u'(y_c)}\mathrm{A}(c).
\end{align*}

It remains to show that if $\mathrm{A}(c)^2+\mathrm{B}(c)^2=0$, then $c$ must be an embedding eigenvalue of
$\mathcal{R}_\al$.
The equality $\mathrm{A}(c)^2+\mathrm{B}(c)^2=0$ implies that
$u''(y_c)=0$ and   $\varphi(1,c)=0$.
 Then $\varphi(y,c)$ is a non zero solution of (\ref{eq:Ray-2}) with boundary conditions $\varphi(0,c)=\varphi(1,c)=0$.
This means that $c$ is an embedding eigenvalue of
$\mathcal{R}_\al$ with eigenfunction $\varphi(y,c)$.
\end{proof}

\begin{remark}
For $c\in \mathbb{C}\setminus D_0$, let $\varphi(y,c)$ be the solution of (\ref{eq:Ray-2}) with $\varphi(0,c)=0$ and $\varphi'(0,c)=1$, then $\varphi(y,c)$ is analytic in $c$, and $ \lim\limits_{c\rightarrow\infty}\varphi(y,c)=\dfrac{\sinh(\al y)}{\al}.$
For $c\in \Omega_{\epsilon_0}\setminus D_0$, we have
\beno
\varphi(y,c)=\phi(0,c)\phi(y,c)\int_0^y\frac{1}{\phi(y',c)^2}dy',
\eeno  in particular,
\beno
\varphi(1,c)=\phi(0,c)\phi(1,c)\int_0^1\frac{1}{\phi(y,c)^2}dy.
\eeno
\end{remark}

\begin{lemma}\label{lem:Wronskian}
If $\mathcal{R}_\al$ has no embedding eigenvalues, then there exists $\epsilon_1>0$ such that
for any $c\in \Omega_{\epsilon_1}\setminus D_0$, $\varphi(1,c)\neq 0$.
\end{lemma}

\begin{proof}
Let $c_r$ be defined by (\ref{def:cr}). We claim the following uniform convergence:

\begin{itemize}

\item[1.] for $c_{\epsilon}\in \{Im\, c>0\}\cap D_{\epsilon_0}$,
\beno
\rho(c_{\epsilon})\int_0^1\frac{1}{\phi(y,c_\epsilon)^2}dy\longrightarrow
\frac{1}{u'(y_c)}\big(\mathrm{A}(c_r)-i\mathrm{B}(c_r)\big)\quad
\textrm{as }c_{\epsilon}\to c_r;
\eeno

\item[2.] for $c_{\epsilon}\in \{Im\, c<0\}\cap D_{\epsilon_0}$,
\beno
\rho(c_{\epsilon})\int_0^1\frac{1}{\phi(y,c_\epsilon)^2}dy\longrightarrow
\frac{1}{u'(y_c)}\big(\mathrm{A}(c_r)+i\mathrm{B}(c_r)\big)\quad
\textrm{as }c_{\epsilon}\to c_r;
\eeno

\item[3.] for $c_\epsilon \in B_{\epsilon_0}^l\cup B_{\epsilon_0}^r$,
\beno
\rho(c_{\epsilon})\int_0^1\frac{1}{\phi(y,c_\epsilon)^2}dy\longrightarrow
\frac{u(0)-u(1)}{u'(y_c)}\quad
\textrm{as }c_{\epsilon}\to c_r.
\eeno
\end{itemize}

By definition of $\rho(c)$ and $\phi_1(y,c)$, we have
\begin{align*}
\varphi(1,c)
=&-\rho(c)\phi_1(0,c)\phi_1(1,c)\int_0^1\frac{1}{\phi(y,c)^2}dy\quad \textrm{for }c\in \Omega_{\epsilon_0}\setminus D_0.
\end{align*}
It is easy to show that $A(c)$ and $B(c)$ are continuous in $D_0$. This along with the above claims implies that
$\varphi(1,c)$ is continuous in $\Om_{\epsilon_0}$.
Thus, the conclusion follows from Lemma \ref{lem:spectrum}.\medskip

Now let us first prove claim 1. Let $c_\epsilon =c+i\epsilon$ with $\epsilon\neq 0$ and $c\in D_0$ and $y_c=u^{-1}(c)$.
We write
\begin{align*}
&\rho(c_{\epsilon})\int_0^1\frac{1}{\phi(y,c_\epsilon)^2}dy\\
&=\rho(c_{\epsilon})\int_0^1\frac{dy}{(u(y)-c_\epsilon)^2\phi_{1}(y,c_\epsilon)^2}\\
&=\rho(c_{\epsilon})\int_0^1\frac{dy}{(u(y)-c_\epsilon)^2}
+\rho(c_{\epsilon})\int_0^1\frac{1}{(u(y)-c_\epsilon)^2}\Big(\frac{1}{\phi_{1}(y,c_\epsilon)^2}-1\Big)dy.
\end{align*}
By Lebesgue dominated convergence theorem, we get
\begin{align*}
\lim_{\epsilon\to 0}\int_0^1\frac{1}{(u(y)-c_\epsilon)^2}\Big(\frac{1}{\phi_{1}(y,c_\epsilon)^2}-1\Big)dy
=\int_0^1\frac{1}{(u(y)-c)^2}\Big(\frac{1}{\phi_{1}(y,c)^2}-1\Big)dy
=\mathrm{II}_3.
\end{align*}
Thanks to $u'(y_c)>c>0$, we have
\begin{align*}
&\rho(c_\epsilon)\int_0^1\frac{dy}{(u(y)-c_\epsilon)^2}\\
&=\frac{1}{u'(y_c)}\rho(c_\epsilon)\bigg(\int_0^1\frac{u'(y)dy}{(u(y)-c_\epsilon)^2}
-\int_0^1\frac{u'(y)-u'(y_c)}{(u(y)-c_\epsilon)^2}dy\bigg)\\
&\triangleq\frac{1}{u'(y_c)}(\mathrm{I}_1-\mathrm{I}_2).
\end{align*}
Thanks to $|\frac{1}{(u(y)-c_\epsilon)^2}|>c\epsilon^2>0$, we get
\begin{align*}
\mathrm{I}_1=-\frac{(c_\epsilon-u(0))(u(1)-c_\epsilon)}{u(y)-c_{\epsilon}}\bigg|_{0}^1
&=\frac{(c_\epsilon-u(0))(u(1)-c_\epsilon)}{u(0)-c_{\epsilon}}-\frac{(c_\epsilon-u(0))(u(1)-c_\epsilon)}{u(1)-c_\epsilon}\\
&=u(0)-u(1).
\end{align*}
Set $g(y)=u'(y)-u'(y_c)-\frac{u''(y_c)}{u'(y_c)^2}u'(y)(u(y)-c)$. Then
$g(y_c)=0$ and $g'(y_c)=0$.
\begin{align*}
\mathrm{I}_2
=&\rho(c_\epsilon)\int_0^1\frac{u'(y)-u'(y_c)}{(u(y)-c_{\epsilon})^2}dy\\
=&\rho(c_\epsilon)\int_0^1\frac{g(y)}{(u(y)-c_{\epsilon})^2}dy
+\rho(c_\epsilon)\frac{u''(y_c)}{u'(y_c)^2}\int_0^1\frac{u'(y)(u(y)-c)}{(u(y)-c_{\epsilon})^2}dy.
\end{align*}
Because of  $|g'(y)|=|g'(y)-g'(y_c)|\leq C|y-y_c|$, we get
\beno
\int_0^1\frac{g(y)}{(u(y)-c_{\epsilon})^2}dy\to p.v. \int_0^1\frac{g(y)}{(u(y)-c)^2}dy.
\eeno
Thus, we deduce that as $\epsilon\to 0+$,
\begin{align*}
&\rho(c_\epsilon)\int_0^1\frac{u'(y)(u(y)-c)}{(u(y)-c_{\epsilon})^2}dy
= \rho(c_\epsilon)\int_{u(0)-c_{\epsilon}}^{u(1)-c_{\epsilon}}\frac{x+i\epsilon}{x^2}dx\\
&=\rho(c_\epsilon)\ln\frac{u(1)-c_{\epsilon}}{u(0)-c_{\epsilon}}+i\epsilon(u(0)-u(1))\\
&\longrightarrow(c-u(0))(u(1)-c)\ln\frac{u(1)-c}{c-u(0)}+i\pi(c-u(0))(u(1)-c)\\
&=\textrm{p.v.}\rho(c)\int_0^1\frac{u'(y)(u(y)-c)}{(u(y)-c)^2}dy+i\pi\rho(c).
\end{align*}
This shows that
\begin{align*}
\mathrm{I}_2\longrightarrow \rho(c)\mathrm{II}_2+i\pi\frac{u''(y_c)}{u'(y_c)^2}\rho(c)\quad \textrm{as }\epsilon\to 0+.
\end{align*}
So, we have as $\epsilon\to 0+$,
\begin{align*}
\rho(c_\epsilon)\int_0^1\frac{1}{\phi(y,c_{\epsilon})^2}dy
\longrightarrow \frac{1}{u'(y_c)}\Big(u(0)-u(1)-\rho(c)\mathrm{II}_2-i\pi\frac{u''(y_c)}{u'(y_c)^2}\rho(c)\Big)+\rho(c)\mathrm{II}_3.
\end{align*}
This shows claim 1, and  the proof of claim 2 is similar.

Let us prove claim 3. Let $c_\epsilon = u(0)+\epsilon e^{i\th}$ with $\epsilon>0$. Similarly, we write
\begin{align*}
&\rho(c_{\epsilon})\int_0^1\frac{1}{\phi(y,c_\epsilon)^2}dy\\
&=\rho(c_{\epsilon})\int_0^1\frac{dy}{(u(y)-c_\epsilon)^2}
+\rho(c_{\epsilon})\int_0^1\frac{1}{(u(y)-c_\epsilon)^2}\Big(\frac{1}{\phi_{1}(y,c_\epsilon)^2}-1\Big)dy.
\end{align*}
Because $\frac{1}{(u(y)-c_\epsilon)^2}\big(\frac{1}{\phi_{1}(y,c_\epsilon)^2}-1\big)$ is uniformly bounded, the second term tends to zero.
For the first term, we write as above
\begin{align*}
&\rho(c_\epsilon)\int_0^1\frac{dy}{(u(y)-c_\epsilon)^2}\triangleq\frac{1}{u'(0)}(\mathrm{I}_1-\mathrm{I}_2).
\end{align*}
It's easy to see that
\begin{align*}
\mathrm{I}_1=u(0)-u(1).
\end{align*}
For $\mathrm{I}_2$, we have
\begin{align*}
|\mathrm{I}_2|=&\Big|\epsilon e^{i\th}(u(1)-u(0)-\epsilon e^{i\th})
\int_{u(0)}^{u(1)}\frac{u'(u^{-1}(z))-u'(0)}{(z-u(0)-\epsilon e^{i\th})^2}(u^{-1})'(z)dz\Big|\\
\leq &C\epsilon\int_0^{u(1)-u(0)}\frac{z}{\sqrt{(z+\epsilon\cos\th)^2+\epsilon^2\sin^2\th}}dz\\
\leq &C\epsilon|\ln \epsilon|\to 0 \quad \textrm{as }\epsilon\to 0.
\end{align*}
For the case of $c_\epsilon\in B_{\epsilon_0}^l$, the proof is similar.
\end{proof}

\begin{remark}\label{rem:spectrum}
If $A(c)<0$ and $B(c)=0$ for $c\in D_0$, then $\mathcal{R}_{\al}$ has no eigenvalue.

Indeed, under this assumption we can find $0<\epsilon_1<1$ such that
for any $c\in \overline{\Omega_{\epsilon_1}}\setminus D_0$, $\varphi(1, c)\not\in (-\infty,0]$.
Since $\varphi(1, c)$ is analytic for $c\in\mathbb{C}\setminus D_0$, and $\varphi(1, c)\rightarrow\dfrac{\sinh\al}{\al}$
for $c\rightarrow\infty$, we can find $R>2$ such that $\varphi(1, c)\in \mathbb{C}\setminus (-\infty,0]$ for $|c|\geq R$. Then by residue theorem,
the number of roots of $\varphi(1, c)$ for $c\in\mathbb{C}\setminus D_0$ equals to the number of roots of $\varphi(1, c)$ for $c\in B_R
\setminus \overline{\Omega_{\epsilon_1}}$, which equals  to(count multipilcity)
$$\frac{1}{2\pi i}\oint_{\partial B_R}-\oint_{\partial\Omega_{\epsilon_1}}\frac{\partial_c \varphi(1, c)}{\varphi(1, c)}dc=\frac{1}{2\pi i}\oint_{\partial B_R}-\oint_{\partial\Omega_{\epsilon_1}}\partial_c\ln \varphi(1, c)dc.$$

Let $z=\varphi(1, c)$.
For $c\in \partial B_R $, $\ln z$ is the analytic function defined in $\mathbb{C}\setminus (-\infty,0]$ such that $\ln z=\ln |z|+ i\arg z$ for
$|z|>0,\ -\pi<\arg z<\pi$.
For $c\in \partial\Omega_{\epsilon_1}$, by our assumption, if we take $\epsilon_1$ small enough, such that $\ln(z)=\ln |z|+ i\arg z$ for $|z|>0,\ -\pi<\arg z<\pi$.
Thus, we deduce that
\beno
\frac{1}{2\pi i}\oint_{\partial B_R}-\oint_{\partial\Omega_{\epsilon_1}}\partial_c\ln \varphi(1, c)dc=0.
\eeno
\end{remark}

\begin{proposition}\label{prop:Rayleigh-IH}
Let $c\in \Omega_{\epsilon_0}$ with $\epsilon_1$ as in Lemma \ref{lem:Wronskian}.
Suppose that $\mathcal{R}_\al$ has no embedding eigenvalues.
Then we have the following representation formula for the solution of (\ref{eq:Rayleigh-Ihom}):
\begin{align*}
\Phi(y,c)=\phi(y,c)\int_0^y\frac{1}{\phi(z,c)^2}\int_{y_c}^z\phi(y',c)f(y',c)dy'dz
+\mu(c)\phi(y,c)\int_0^y\frac{1}{\phi(y',c)^2}dy',
\end{align*}
where $y_c=u^{-1}(c_r)$ with $c_r$ defined by (\ref{def:cr}) and
$$
\mu(c)=-\frac{\int_0^1\frac{1}{\phi(z,c)^2}\int_{y_c}^z\phi f(y',c) dy'dz}{\int_0^1\frac{1}{\phi(y,c)^2}dy}.
$$
\end{proposition}

\begin{proof}
Let $\varphi(y)$ be a solution of the homogenous Rayleigh equation
\beno
\varphi''-\al^2\varphi-\frac{u''}{u-c}\varphi=0,
\eeno
then the inhomogeneous Rayleigh equation
\beno
\left\{
\begin{aligned}
&\Phi''-\al^2\Phi-\frac{u''}{u-c}\Phi=f,\\
&\Phi(0)=\Phi(1)=0,
\end{aligned}
\right.
\eeno
is equivalent to
\beno
\left\{
\begin{aligned}
&\Big(\varphi^2\big(\frac{\Phi}{\varphi}\big)'\Big)'=f\varphi,\\
&\psi(0)=\psi(1)=0.
\end{aligned}
\right.
\eeno
This gives our result by integration and noting that $\mu(c)$ is well-defined by Lemma \ref{lem:Wronskian}.
\end{proof}

Next we study the convergence of the solution $\psi(y,c)$ with $f=\f {\om_0(y)} {i\al(u-c)}$.
Let $c\in D_0$ and $y_c=u^{-1}(c)$.
We introduce
\begin{align*}
\Phi_\pm(y,c)
\triangleq\left\{
\begin{array}{l}
\phi\int_0^y\frac{1}{\phi(z,c)^2}\int_{y_c}^z\phi f(y',c)dy'dz
+\mu_{\pm}(c)\phi\int_0^y\frac{1}{\phi(y',c)^2}dy'\quad 0\leq y\leq y_c,\\
\phi\int_1^y\frac{1}{\phi(z,c)^2}\int_{y_c}^z\phi f(y',c)dy'dz
+\mu_{\pm}(c)\phi\int_1^y\frac{1}{\phi(y',c)^2}dy'\quad y_c\leq y\leq 1,
\end{array}
\right.
\end{align*}
and
\begin{align*}
&\Phi_l(y)\triangleq\phi(y,u(0))\int_1^y\frac{1}{\phi(z,u(0))^2}\int_{0}^z\phi(y',u(0))f(y',u(0))dy'dz,\\
&\Phi_r(y)\triangleq\phi(y,u(1))\int_0^y\frac{1}{\phi(z,u(1))^2}\int_{1}^z\phi(y',u(1))f(y',u(1))dy'dz,
\end{align*}
where
\beno
&&\mu_{+}(c)=\frac{1}{\al}\frac{iu'(y_c)\rho(c)\mathrm{II}_1
-\rho(c)\frac{\omega_0(y_c)}{u'(y_c)}\pi}
{u(0)-u(1)-\rho(c)\mathrm{II}_2-i\pi\rho(c)\frac{u''(y_c)}{u'(y_c)^2}+u'(y_c)\rho(c)\mathrm{II}_3},\\
&&\mu_{-}(c)=\frac{1}{\al}\frac{iu'(y_c)\rho(c)\mathrm{II}_1
+\rho(c)\frac{\omega_0(y_c)}{u'(y_c)}\pi}
{u(0)-u(1)-\rho(c)\mathrm{II}_2+i\pi\rho(c)\frac{u''(y_c)}{u'(y_c)^2}+u'(y_c)\rho(c)\mathrm{II}_3},
\eeno
with
\ben
\mathrm{II}_1(\om_0)=\textrm{p.v.}\int_0^1\frac{\int_{y_c}^z\omega_0(y')\phi_1(y',c)
dy'}{\phi(z,c)^2}dz.\label{eq:Pi-1}
\een

\begin{lemma}\label{lem:conver}
Suppose that $f(y,c)$ is a Lipschitz function in $y\in [0,1]$ with
\beno
|f(y,c)-f(z,c)|\leq C|y-z|\quad \textrm{for }(y,z,c)\in [0,1]^2\times D_{\epsilon_0},
\eeno
and $f(y,c_{\epsilon})$ uniformly converges to $f(y,c)$ for $c_\epsilon=c+i\epsilon, c\in D_0$ as $\epsilon\to 0$.
Then it holds that for $0\leq y\leq y_c$,
\beno
(y-y_c+i\epsilon)\int_0^y\frac{f(z,c_{\epsilon})}{(z-y_c+i\epsilon)^2}dz\longrightarrow (y-y_c)\int_0^y\frac{f(z,c)}{(z-y_c)^2}dz
\eeno
as $\epsilon$ tends to zero.
\end{lemma}
\begin{proof}
Thanks to $0\leq z\leq y\leq y_c$, then $|z-y_c|\geq |y-y_c|$ and
\begin{align*}
\bigg|\frac{(f(z,c_{\epsilon})-f(y_c,c_{\epsilon}))(y-y_c+i\epsilon)}{(z-y_c+i\epsilon)^2}\bigg|
\leq C\frac{|z-y_c|(|y-z_0|^2+\epsilon^2)^{1/2}}{|z-y_c|^2+\epsilon^2}\leq C,
\end{align*}
which together with Lebesgue dominated convergence theorem gives
\begin{align*}
(y-y_c+i\epsilon)\int_0^y\frac{f(z,c_{\epsilon})-f(y_c,c_{\epsilon})}{(z-y_c+i\epsilon)^2}dz\longrightarrow
(y-y_c)\int_0^y\frac{f(z,c)-f(y_c,c)}{(z-y_c)^2}dz.
\end{align*}
On the other hand, we have
\begin{align*}
(y-y_c+i\epsilon)\int_0^y\frac{1}{(y-y_c+i\epsilon)^2}dz
&=-\frac{y-y_c+i\epsilon}{-y_c+i\epsilon}+1\\
&\longrightarrow (y-y_c)\int_0^y\frac{1}{(z-y_c)^2}dz
\end{align*}
as $\epsilon$ tends to zero. Then the lemma follows easily.
\end{proof}

\begin{proposition}\label{prop:psi-conv}
Let $\Phi(y,c)$ be a solution of (\ref{eq:Rayleigh-Ihom}) given by Proposition \ref{prop:Rayleigh-IH}
with $f=\f {\om_0(y)} {i\al(u-c)}$ for $\om_0\in L^2(0,1)$.
If $\mathcal{R}_\al$ has no embedding eigenvalues,
then it holds that
\begin{itemize}

\item[1.] for any $(y,y_c)\in[0,1]\times [0,1]$,
\beno
\lim_{\epsilon\to 0+}\Phi(y,c_\epsilon)=\Phi_+(y,u(y_c)),\quad
\lim_{\epsilon\to 0-}\Phi(y,c_\epsilon)=\Phi_-(x,u(y_c)),
\eeno
where $c_\epsilon=c+i\epsilon$ and $c=u(y_c)$.

\item[2.] for any $y\in[0,1]$,
\beno
\lim_{\epsilon\to 0+}\Phi(y,c_\epsilon^l)=\Phi_l(y),\quad
\lim_{\epsilon\to 0+}\Phi(y,c_\epsilon^r)=\Phi_r(y),
\eeno
where $c_\epsilon^l=u(0)+\epsilon e^{i\theta}$ and $c_\epsilon^r=u(1)-\epsilon e^{i\theta}$ for $\th\in [\frac \pi 2, \frac {3\pi} 2]$.
\end{itemize}
\end{proposition}

\begin{proof}

Let us first prove the first statement. To show the convergence of $\mu(c_\epsilon)$, we consider
\begin{align*}
\mu(c_\epsilon)=&-\frac{\int_0^1\frac{1}{\phi(z,c_\epsilon)^2}\int_{y_c}^z\phi f(y,c_\epsilon) dydz}{\int_0^1\frac{1}{\phi (y,c_\epsilon)^2}dy}\\
=&-\frac{\rho(c_\epsilon)\int_0^1\frac{1}{\phi(z,c_\epsilon)^2}\int_{y_c}^z\phi f(y,c_\epsilon)dydz}{\rho(c_\epsilon)\int_0^1\frac{1}{\phi(y,c_\epsilon)^2}dy}.
\end{align*}
By the claims in the proof of Lemma \ref{lem:Wronskian}, it suffices to deal with the numerator of $\mu(c)$, which is decomposed as
\begin{align*}
\rho(c_\epsilon)\int_0^1\frac{1}{\phi(z,c_\epsilon)^2}\int_{y_c}^z\phi f(y,c_\epsilon)dydz
=&\rho(c_\epsilon)\int_0^1\frac{1}{\phi(z,c_{\epsilon})^2}\int_{y_c}^z\phi_{1}(y,c_\epsilon)\frac{\omega_0(y)}{i\al}dydz\\
=&\rho(c_\epsilon)\int_0^1\frac{1}{\phi(z,c_\epsilon)^2}\int_{y_c}^z(\phi_{1}(y,c_{\epsilon})-1)\frac{\omega_0(y)}{i\al}dydz\\
&+\rho(c_\epsilon)\int_0^1\frac{\int_{y_c}^y\frac{\omega_0(y)}{i\al}dy}{(u(z)-c_\epsilon)^2}\big(\frac{1}{\phi_{1}(z,c_{\epsilon})^2}-1\big)dz\\
&+\rho(c_\epsilon)\int_0^1\frac{\int_{y_c}^z\frac{\omega_0(y)}{i\al}dy}{(u(z)-c_\epsilon)^2}dz.
\end{align*}
Using the inequality $|\phi_{1}(y,c_{\epsilon})-1|\leq C(|y-y_c|^2+\epsilon^2)$ and Lebesgue dominated convergence theorem,
we get
\beno
&&\rho(c_\epsilon)\int_0^1\frac{1}{\phi(z,c_\epsilon)^2}\int_{y_c}^z(\phi_{1}(y,c_{\epsilon})-1)\frac{\omega_0(y)}{i\al}dydz\\
&&\longrightarrow\rho(c)\int_0^1\frac{1}{\phi(z,c)^2}\int_{y_c}^z(\phi_{1}(y,c)-1)\frac{\omega_0(y)}{i\al}dydz,
\eeno
and
\beno
&&\rho(c_\epsilon)\int_0^1\frac{\int_{y_c}^y\frac{\omega_0(y)}{i\al}dy}{(u(z)-c_\epsilon)^2}\big(\frac{1}{\phi_{1}(z,c_{\epsilon})^2}-1\big)dz
\longrightarrow\rho(c)\int_0^1\frac{\int_{y_c}^z\frac{\omega_0(y)}{i\al}dy}{(u(z)-c)^2}\big(\frac{1}{\phi_{1}(z,c)^2}-1\big)dz
\eeno
as $\epsilon$ tends to zero.
Similar to the proof of $\mathrm{I}_2$ in Lemma \ref{lem:Wronskian}, we have
\beno
\rho(c_\epsilon)\int_0^1\frac{\int_{x_0}^z\frac{\omega_0(y)}{i\al}dy}{(u(z)-c_\epsilon)^2}dz\to \rho(c)\int_0^1\frac{\int_{x_0}^z\frac{\omega_0(y)}{i\al}dy}{(u(z)-c)^2}dz+\rho(c)\frac{\omega_0(y_c)}{i\al u'(y_c)^2}\pi i.
\eeno
as $\varepsilon\rightarrow 0+$. Thus, we obtain as $\epsilon\to 0+$,
\begin{align*}
&\rho(c_\epsilon)\int_0^1\frac{1}{\phi(z,c_\epsilon)^2}\int_{y_c}^z\phi f(y,c_\epsilon)dydz\\
&\longrightarrow \textrm{p.v.}\rho(c)\int_0^1\frac{\int_{y_c}^z\frac{\omega_0(y)\phi_1(y,c)}
{i\al}dy}{\phi(z,c)^2}dz
+\rho(c)\frac{\omega_0(y_c)}{i\al u'(y_c)^2}\pi i\\
&=-\frac{1}{\al}\Big(i\rho(c)\mathrm{II}_1(\om_0)-\rho(c)\frac{\omega_0(y_c)}{u'(y_c)^2}\pi\Big).
\end{align*}
Similarly, we have that $\epsilon\to 0-$,
\begin{align*}
\rho(c_\epsilon)\int_0^1\frac{1}{\phi(z,c_\epsilon)^2}\int_{y_c}^z\phi f(y,c_\epsilon)dydz
\longrightarrow-\frac{1}{\al}\Big(i\rho(c)\mathrm{II}_1+\rho(c)\frac{\omega_0(y_c)}{u'(y_c)^2}\pi\Big).
\end{align*}
This shows by Lemma \ref{lem:spectrum} that
\ben\label{eq:mu-con}
\mu(c_\epsilon)\longrightarrow \mu_\pm(c)\quad as \quad \epsilon\to 0\pm.
\een

Next we show the convergence of $\psi(y,c_\epsilon)$.
Using the formula
\begin{align*}
\phi(y,c_{\epsilon})
=&(u(y)-u(y_c)-i\epsilon)\phi_{1}(y,c_{\epsilon})\\
=&\big((y-y_c)\int_0^1u'(y_c+t(y-y_c))dt-i\epsilon\big)\phi_{1}(y,c_{\epsilon}),
\end{align*}
we deduce from Lemma \ref{lem:conver} that as $\epsilon\to 0$,
\beno
&&\phi(y,c_{\epsilon})\int_0^y\frac{1}{\phi(z,c_{\epsilon})^2}dy
\longrightarrow\phi(y,c)\int_0^y\frac{1}{\phi(y,c)^2}dy,\\
&&\phi(y,c_\epsilon)\int_0^y\frac{1}{\phi(z,c_\epsilon)^2}\int_{y_c}^z\phi f(y',c_{\epsilon})dy'dz
\longrightarrow \phi(y,c)\int_0^y\frac{1}{\phi(z,c)^2}\int_{y_c}^z\phi f(y',c)dy'dz,
\eeno
for $y_c\ge y\ge 0$, and
\beno
&&\phi(y,c_{\epsilon})\int_1^y\frac{1}{\phi(y,c_{\epsilon})^2}dy
\longrightarrow\phi(y,c)\int_1^y\frac{1}{\phi(y',c)^2}dy',\\
&&\phi(y,c_\epsilon)\int_1^y\frac{1}{\phi(z,c_\epsilon)^2}\int_{y_c}^z\phi f(y',c_{\epsilon})dydz
\longrightarrow \phi\int_1^x\frac{1}{\phi(z,c)^2}\int_{y_c}^z\phi f(y',c)dy'dz,
\eeno
for $y_c\leq y\leq 1$. This finishes the proof of the first statement.

The second statement is similar just by noting that in this case, we have
\beno
\mu(c_\epsilon^l)\longrightarrow 0,\quad \mu(c_\epsilon^r)\longrightarrow 0
\eeno
as $\epsilon\rightarrow 0+$.
\end{proof}

\section{Representation formula of the solution}
Let $\widehat{\psi}(t,\al,y)$ be the solution of the linearized Euler equations
\beno
\frac 1 {i\al}\partial_t\widehat{\psi}=(\partial_y^2-\al^2)^{-1}\big(u''(y)-u(y)(\partial_y^2-\al^2)\big)\widehat{\psi}=-\mathcal{R}_\al\widehat{\psi}\\
\eeno
with the initial data
\beno
\psi(0,\al, y)=(\al^2-\pa_y^2)^{-1}\widehat{\om}_0(\al,y),
\eeno
where $\widehat{\om}_0(\al,y)=\mathcal{F}_x\om_0$.
We know from (\ref{eq:stream formula}) that
\beno
\widehat{\psi}(t,\al,y)=\frac{1}{2\pi i}\int_{\partial\Omega}
e^{-i\al tc}(c-\mathcal{R}_\al)^{-1}\widehat{\psi}(0,\al,y)dc,
\eeno
where $\Omega$ is a simple connected domain including the spectrum $\sigma(\mathcal{R}_\al)$ of $\mathcal{R}_\al$.

Note that $P_{i\al\mathcal{R}_\al}\widehat{\psi}(0,\al,y)=0$ if $P_{\mathcal{L}}\om_0=0$.
Thus, we have
\beno
\widehat{\psi}(t,\al,y)=\frac{1}{2\pi}\int_{\partial\Omega_{\epsilon}}e^{-i\al tc}(c-\mathcal{R}_\al)^{-1}\widehat{\psi}(0,\al,y)dc,
\eeno
where $\Om_\epsilon$ defined by (\ref{eq:Omega}) with $\epsilon$ sufficiently small.

Let $\Phi(\al,y,c)$ be the solution of  (\ref{eq:Rayleigh-Ihom})
with $f(\al,y,c)=\f {\widehat{\om}_0(\al,y)} {i\al(u-c)}$. It is easy to see that
\beno
(c-\mathcal{R}_\al)^{-1}\widehat{\psi}(0,\al,y)=i\al\Phi(\al,y,c).
\eeno
This gives
\beq\label{eq:solution formu}
\widehat{\psi}(t,\al,y)=\frac{1}{2\pi}\int_{\partial\Omega_\epsilon}\al\Phi(\al,y,c)e^{-i\al ct}dc.
\eeq
Since $\mathcal{L}$(thus, $\mathcal{R}_\al$) has no embedding eigenvalues, Lemma \ref{lem:Wronskian} ensures that there are no eigenvalues of $\mathcal{R}_\al$ in $\Omega_{\epsilon}$ for $\epsilon$ small enough. Then it follows from Proposition \ref{prop:psi-conv} that
\begin{align}
\widehat{\psi}(t,\al,y)=&\frac{1}{2\pi}\int_{\partial\Omega_{\epsilon}}\al\Phi(\al,y,c)e^{-i\al ct}dc\nonumber\\
=&\lim_{\epsilon\to 0^+}\frac{1}{2\pi}\int_{\partial\Omega_{\epsilon}}\al\Phi(\al,y,c)e^{-i\al ct}dc\nonumber\\
=&\frac{1}{2\pi}\int_{u(0)}^{u(1)}\al\widetilde{\Phi}(y,c)e^{-i\al ct}dc,\label{eq:psi-e}
\end{align}
where
\beno
\widetilde{\Phi}(y,c)=
\left\{
\begin{aligned}
&(\mu_-(c)-\mu_+(c))\phi(y,c)\int_0^y\frac{1}{\phi(z,c)^2}dz\quad0\leq y<y_c,\\
&(\mu_-(c)-\mu_+(c))\phi(y,c)\int_1^y\frac{1}{\phi(z,c)^2}dz\quad y_c<y\leq 1,
\end{aligned}
\right.
\eeno
with $y_c=u(c)$, and $\phi(y,c)$ is the solution of (\ref{eq:Rayleigh-H}) given by Proposition \ref{prop:Rayleigh-Hom}.

We denote
\begin{align*}
\mathrm{A}&\triangleq u(0)-u(1)-\rho(c)\textrm{II}_2+u'(y_c)\rho(c)\textrm{II}_3,\\
\mathrm{B}&\triangleq \pi\rho(c)\frac{u''(y_c)}{u'(y_c)^2},\\
\mathrm{C}&\triangleq \rho(c)\frac{\widehat{\omega}_0(\al, y_c)}{u'(y_c)}\pi,\quad \mathrm{D}\triangleq u'(y_c)\rho(c)\textrm{II}_1(\widehat{\om}_0),
\end{align*}
where $\rho(c)=(c-u(0))(u(1)-c)$, and $\textrm{II}_1, \textrm{II}_2$ and $\textrm{II}_3$ are given by (\ref{eq:Pi-1}), (\ref{eq:Pi-2}) and (\ref{eq:Pi-3}) respectively.
Then we have
\begin{align}\label{def:mu}
\mu_-(c)-\mu_+(c)=\frac{2}{\al}\frac{\mathrm{AC}+\mathrm{BD}}{\mathrm{A}^2+\mathrm{B}^2}\triangleq\frac{2}{\al}\rho(c)\mu(c).
\end{align}

\section{Sobolev regularity of $\mu(c)$}

In this section, we study the regularity of $\mu(c)$ defined by (\ref{def:mu}). The result
is stated as follows.

\begin{proposition}\label{prop:mu}
With the same assumptions as in Theorem \ref{Thm:main}, there exists a constant $C$ independent of $\al$ such that
\begin{itemize}

\item[1.] $L^2$ estimates
\beno
&&\|\rho\mu\|_{L^2}\leq \frac{C}{\al}\|\widehat{\omega}_0(\al,\cdot)\|_{L^2},\\
&&\|\mu\|_{L^2}\le C\|\widehat{\omega}_0(\al,\cdot)\|_{L^2}.
\eeno

\item[2.] $W^{1,2}$ estimates
\beno
&&\|\partial_c(\rho\mu)\|_{L^2}
\leq C\|\widehat{\omega}_0(\al,\cdot)\|_{H^1},\\
&&\|\partial_c\mu\|_{L^2}
\leq C(1+\al)\|\widehat{\omega}_0(\al,\cdot)\|_{H^1}.
\eeno

\item[3.] $W^{2,2}$ estimates
\beno
&&\|\partial_c^2(\rho\mu)\|_{L^2}\leq C(1+\al)\|\widehat{\omega}_0(\al,\cdot)\|_{H^2},\\
&&\|\rho\partial_c^2(\rho\mu)\|_{L^2}\leq C\|\widehat{\omega}_0(\al,\cdot)\|_{H^2}.
\eeno
\end{itemize}
\end{proposition}

\begin{proof}
Let us first estimate $\mathrm{II}_1, \mathrm{II}_2$ and $\mathrm{II}_3$. Recall that
\beno
&&\mathrm{II}_1(\widehat{\om}_0)=\textrm{p.v.}\int_0^1\frac{\int_{y_c}^z\widehat{\omega}_0(\al,y')\phi_1(y',c)
dy'}{\phi(z,c)^2}dz,\\
&&\mathrm{II}_2=\textrm{p.v.}\int_0^1
\frac{u'(y)-u'(y_c)}{(u(y)-c)^2}dy,\\
&&\mathrm{II}_3=\int_0^1\frac{1}{(u(y)-c)^2}\Big(\frac{1}{\phi_{1}(y,c)^2}-1\Big)dy.
\eeno

By a change of variable, $\mathrm{II}_1$ can be written as
\beno
\mathrm{II}_1(\widehat{\om}_0)(c)=\mathcal{L}_1\big(\widehat{\omega}_0(\al,u^{-1})(u^{-1})'\big).
\eeno
Then Lemma \ref{lem:iden-L}, Lemma\ref{lem:B-H1} and Lemma\ref{lem:L-H1} ensure that
\ben
&&\big\|\mathrm{II}_1(\widehat{\om}_0)\big\|_{L^p}\le C\|\widehat{\omega}_0(\al,\cdot)\|_{L^p},\label{eq:pi1-L2}\\
&&\big\|\pa_c(\rho\mathrm{II}_1(\widehat{\om}_0))\big\|_{L^p}\le C\|\widehat{\omega}_0(\al,\cdot)\|_{W^{1,p}},\label{eq:pi1-H1}\\
&&\big\|\pa_c^2(\rho^2\mathrm{II}_1(\widehat{\om}_0))\big\|_{L^p}\leq C\|\widehat{\omega}_0(\al,\cdot)\|_{W^{2,p}}.\label{eq:pi1-H2}
\een

Let $y_c=u^{-1}(c)$ in the sequel.
We rewrite $\textrm{II}_2$ as
\begin{align*}
\mathrm{II}_2(c)=&\textrm{p.v.}\int_0^1\frac{\int_{y_c}^zu''(y)dy}{(u(z)-c)^2}dz\\
=&-\frac{1}{u'(y_c)}\int_0^1
\frac{(\int_0^1u''(z+t(z-y_c))dt)^2}
{(\int_0^1u'(z+t(z-y_c))dt)^2}dz
\\&+\frac{1}{u'(y_c)}\textrm{p.v.}\int_0^1
\frac{-u'(z)\int_{y_c}^zu''(y)dy}{(u(z)-c)^2}dz\\
\triangleq&\mathrm{II}_{2,1}+\frac{1}{u'(y_c)}\mathrm{II}_{2,2}.
\end{align*}
It's easy to see that
\beno
\|\mathrm{II}_{2,1}\|_{L^{\infty}}+
\|\partial_c\mathrm{II}_{2,1}\|_{L^{\infty}}+
\|\partial_{c}^2\mathrm{II}_{2,1}\|_{L^{\infty}}\leq C.
\eeno
We rewrite $\mathrm{II}_{2,2}$ as
\begin{align*}
\mathrm{II}_{2,2}(c)=&
\int_{u(0)}^{u(1)}\frac{\int_c^{c'}u''(u^{-1}(z))(u^{-1})'(z)dz}
{(c'-c)^2}dc'\\
=&-\chi_{D_0}(c)H\big(u''(u^{-1})(u^{-1})'\chi_{D_0}\big)(c)\\
&+\frac{1}{c-u(0)}\int_{u(0)}^cu''(u^{-1}(z))(u^{-1})'(z)dz
-\frac{1}{u(1)-c}\int_c^{u(0)}u''(u^{-1}(z))(u^{-1})'(z)dz,
\end{align*}
from which and Lemma \ref{lem:H-H2}, it follows that
\beno
\|\mathrm{II}_{2,2}\|_{L^{p}}+
\|\partial_c(\rho\mathrm{II}_{2,2})\|_{L^{p}}+
\|\partial_{c}^2(\rho^2\mathrm{II}_{2,2})\|_{L^{p}}\leq C.
\eeno
Thus, we deduce that for any $p\in(1,\infty)$,
\ben
&&\|\mathrm{II}_{2}\|_{L^p}+\|\partial_c(\rho\mathrm{II}_{2})\|_{L^p}
+\|\partial_c^2(\rho^2\mathrm{II}_{2})\|_{L^p}\leq C,\label{eq:pi2-Lp}\\
&&\|\rho\mathrm{II}_{2}\|_{L^{\infty}}+\|\partial_c(\rho^2\mathrm{II}_{2})\|_{L^{\infty}}\leq C.\label{eq:pi2-Linfty}
\een

We get by Lemma \ref{lem:phi1-est} that
\begin{align*}
|\mathrm{II}_3(c)|=&\int_0^1\frac{1}{(u(y)-c)^2}\Big(1-\frac{1}{\phi_{1}(y,c)^2}\Big)dy\\
\geq&C^{-1}\int_0^1\frac{1}{(y-y_c)^2}\Big(\frac{\phi_1^2(y,c)-1}{\phi_{1}(y,c)^2}\Big)dy\\
\geq&C^{-1}\al\int_{-\al y_c}^{\al(1-y_c)}\frac{1}{\sinh^2 y}\Big(\frac{\sinh y}{y}-1\Big)dy\\
\geq&C^{-1}\al\int_0^{1/2}\frac{1}{\sinh^2 y}\Big(\frac{\sinh y}{y}-1\Big)dy
\geq C^{-1}\al,
\end{align*}
and
\begin{align*}
|\mathrm{II}_3(c)|=&\int_0^1\frac{1}{(u(y)-c)^2}\Big(1-\frac{1}{\phi_{1}(y,c)^2}\Big)dy\\
\leq&C\al\int_{-\al y_c}^{\al(1-y_c)}\frac{1}{y\sinh y}\Big(\frac{\sinh y}{y}-1\Big)dy\\
\leq &C\al\Big(1+\int_{|y|>1}\frac{1}{y^2}+\frac{1}{y\sinh y}dy\Big)
\leq C\al.
\end{align*}
This gives
\ben
C^{-1}\al\leq |\mathrm{II}_3(c)|\leq C\al.\label{eq:pi3-Lower}
\een

By Remark \ref{rem:phi1}, we have
\begin{align*}
\pa_c\mathrm{II}_3(c)=&\pa_c\int_0^1\frac{1}{(u(y)-c)^2}\Big(\frac{1}{\phi_{1}(y,c)^2}-1\Big)dy\\
=&\int_0^1-\pa_y\Big(\frac{1}{(u(y)-c)^2}\Big)\frac{1}{u'(y)}\Big(\frac{1}{\phi_{1}(y,c)^2}-1\Big)dy\\
&+\int_0^1\frac{1}{(u(y)-c)^2}\pa_c\Big(\frac{1}{\phi_{1}(y,c)^2}-1\Big)dy\\
=&\frac{\al^2\mathcal{T}(\phi_1)(y,c)(\phi_1(y,c)+1)}{\phi_1(y,c)^2u'(y_c)\int_{0}^1u'(y_c+t(y-y_c))dt}\Big|_{y=0}^1\\
&+\int_0^1\frac{1}{(u(y)-c)^2}\big(\pa_c+\frac{1}{u'(y)}\pa_y\big)\Big(\frac{1}{\phi_{1}(y,c)^2}\Big)dy\\
&+\int_0^1\frac{1}{(u(y)-c)^2}\Big(\frac{1}{u'(y)}\Big)'\Big(\frac{1}{\phi_{1}(y,c)^2}-1\Big)dy,
\end{align*}
and
\begin{align*}
\pa_c^2\mathrm{II}_3(c)
=&\pa_c\bigg(\frac{\al^2\mathcal{T}(\phi_1)(y,c)(\phi_1(y,c)+1)}{\phi_1(y,c)^2u'(y_c)\int_{0}^1u'(y_c+t(y-y_c))dt}\Big|_{y=0}^1\bigg)\\
&+\int_0^1-\pa_y\Big(\frac{1}{(u(y)-c)^2}\Big)\frac{1}{u'(y)}\big(\pa_c+\frac{1}{u'(y)}\pa_y\big)\Big(\frac{1}{\phi_{1}(y,c)^2}\Big)dy\\
&+\int_0^1-\pa_y\Big(\frac{1}{(u(y)-c)^2}\Big)\frac{1}{u'(y)}\Big(\frac{1}{u'(y)}\Big)'\Big(\frac{1}{\phi_{1}(y,c)^2}-1\Big)dy\\
&+\int_0^1\Big(\frac{1}{(u(y)-c)^2}\Big)\pa_c\big(\pa_c+\frac{1}{u'(y)}\pa_y\big)\Big(\frac{1}{\phi_{1}(y,c)^2}\Big)dy\\
&+\int_0^1\frac{1}{(u(y)-c)^2}\Big(\frac{1}{u'(y)}\Big)'\pa_c\Big(\frac{1}{\phi_{1}(y,c)^2}-1\Big)dy\\
=&\pa_c\bigg(\frac{\al^2\mathcal{T}(\phi_1)(y,c)(\phi_1(y,c)+1)}{\phi_1(y,c)^2u'(y_c)\int_{0}^1u'(y_c+t(y-y_c))dt}\Big|_{y=0}^1\bigg)\\
&-\frac{1}{(u(y)-c)^2}\frac{1}{u'(y)}\big(\pa_c+\frac{1}{u'(y)}\pa_y\big)\Big(\frac{1}{\phi_{1}(y,c)^2}\Big)\Big|_{y=0}^1\\
&-\frac{1}{(u(y)-c)^2}\frac{1}{u'(y)}\Big(\frac{1}{u'(y)}\Big)'\Big(\frac{1}{\phi_{1}(y,c)^2}-1\Big)\Big|_{y=0}^1\\
&+\int_0^1\frac{1}{(u(y)-c)^2}\Big(\frac{1}{u'(y)}\Big)'\big(\pa_c+\frac{1}{u'(y)}\pa_y\big)\Big(\frac{1}{\phi_{1}(y,c)^2}\Big)dy\\
&+\int_0^1\frac{1}{(u(y)-c)^2}\big(\pa_c+\frac{1}{u'(y)}\pa_y\big)^2\Big(\frac{1}{\phi_{1}(y,c)^2}\Big)dy\\
&+\int_0^1\frac{1}{(u(y)-c)^2}\Big(\frac{1}{u'(y)}\Big(\frac{1}{u'(y)}\Big)'\Big)'\Big(\frac{1}{\phi_{1}(y,c)^2}-1\Big)dy\\
&+\int_0^1\frac{1}{(u(y)-c)^2}\Big(\frac{1}{u'(y)}\Big)'\big(\frac{1}{u'(y)}\pa_y+\pa_c\big)\Big(\frac{1}{\phi_{1}(y,c)^2}\Big)dy.
\end{align*}

Then by Remark \ref{rem:phi1} again, Proposition \ref{prop:phi} and Proposition \ref{prop:phi-good}, we obtain
\ben
\|\partial_{c}\mathrm{II}_3\|_{L^{\infty}}\leq C\al^2,\quad
\|\partial_{c}^2\mathrm{II}_3\|_{L^{\infty}}\leq C\al^3,\quad
\|\partial_{c}(\rho^2\pa_c\mathrm{II}_3)\|_{L^{\infty}}\leq C\al.\label{eq:pi3-H2}
\een

Now let us turn to the estimate of $\mu(c)$.
\begin{align*}
\mu(c)=&\frac{\pi\rho(c)\textrm{II}_3(c)\widehat{\omega}_0(\al,y_c)}{\mathrm{A}^2+\mathrm{B}^2}
-\frac{\pi\rho(c)\textrm{II}_2(c)\widehat{\omega}_0(\al,y_c)}{u'(y_c)(\mathrm{A}^2+\mathrm{B}^2)}\\
&+\frac{\pi(u(0)-u(1))\widehat{\omega}_0(\al,y_c)}{u'(y_c)(\mathrm{A}^2+\mathrm{B}^2)}
+\frac{\pi u''(y_c)\rho\textrm{II}_1(\widehat{\omega}_0)(c)}{u'(y_c)(\mathrm{A}^2+\mathrm{B}^2)}.
\end{align*}
By Lemma \ref{lem:spectrum}, $\mathrm{A}^2+\mathrm{B}^2\geq C^{-1}$. Thus, by (\ref{eq:pi3-Lower}), we get
\beno
\mathrm{A}^2+\mathrm{B}^2\geq C^{-1}((1+\al\rho)^2+\rho^2).
\eeno
Thus, by (\ref{eq:pi1-L2}), (\ref{eq:pi3-Lower}) and (\ref{eq:pi2-Linfty}), we get
\begin{align*}
\|\mu\|_{L^2}\leq& C\big(1+\frac{1}{\al}\|\textrm{II}_3\|_{L^{\infty}}+\|\rho\textrm{II}_2\|_{L^{\infty}}\big)\|\widehat{\om}_0(\al,\cdot)\|_{L^2}
+C\|\textrm{II}_1\|_{L^2}\leq C\|\widehat{\om}_0(\al,\cdot)\|_{L^2}.
\end{align*}
Using the inequality $(1+\al\rho)^2+\rho^2\geq C^{-1}(\rho\al)^k$ for $k=1, 2$, we can deduce that
\begin{align*}
\|\rho\mu\|_{L^2}\leq \frac{C}{\al}\|\widehat{\om}_0(\al,\cdot)\|_{L^2}.
\end{align*}

The $H^1$ estimate of $\mu(c)$ is similar. Here we just show the estimate of one term.
\begin{align*}
\pa_c\Big(\frac{\rho\textrm{II}_1(\widehat{\omega}_0)(c)}{\mathrm{A}^2+\mathrm{B}^2}\Big)
&=\frac{\pa_c(\rho\textrm{II}_1(\widehat{\omega}_0)(c))}{\mathrm{A}^2+\mathrm{B}^2}
-\frac{\rho\textrm{II}_1(\widehat{\omega}_0)(c)\pa_c(\mathrm{A}^2+\mathrm{B}^2)}{(\mathrm{A}^2+\mathrm{B}^2)^2}.
\end{align*}
We infer from (\ref{eq:pi1-H1}) and (\ref{eq:pi3-Lower}) that
\beno
\Big\|\frac{\pa_c(\rho\textrm{II}_1(\widehat{\omega}_0)(c))}{\mathrm{A}^2+\mathrm{B}^2}
\Big\|_{L^2}\leq C\|\widehat{\om}_0(\al,\cdot)\|_{L^2}.
\eeno
Notice that
\begin{align*}
|\pa_c(\mathrm{A}^2+\mathrm{B}^2)|\leq &\big|2\mathrm{A}\big(\pa_c(u'(y_c)\rho(c)\mathrm{II}_3)-\pa_c(\rho(c)\mathrm{II}_2)\big)
+2\mathrm{B}\pa_c\mathrm{B}\big|\\
\leq &C(\mathrm{A}^2+\mathrm{B}^2)^{1/2}\big(1+|\pa_c(\rho(c)\mathrm{II}_3)|+|\rho\mathrm{II}_3|+|\pa_c(\rho(c)\mathrm{II}_2)|\big)
\end{align*}
Thus, by (\ref{eq:pi1-L2}), (\ref{eq:pi3-H2}), (\ref{eq:pi2-Lp}) and Sobolev embedding, we get
\begin{align*}
&\Big\|\frac{\rho\textrm{II}_1(\widehat{\omega}_0)(c)\pa_c(\mathrm{A}^2+\mathrm{B}^2)}{(\mathrm{A}^2+\mathrm{B}^2)^2}\Big\|_{L^2}\\
&\leq C\Big\|\frac{\pa_c(\rho(c)\mathrm{II}_2)}{(\mathrm{A}^2+\mathrm{B}^2)^{3/2}}\Big\|_{L^4}\|\rho\textrm{II}_1(\widehat{\omega}_0)\|_{L^{4}}
+C\Big(\frac{1}{\al}\|\pa_c\mathrm{II}_3\|_{L^{\infty}}+\|\mathrm{II}_3\|_{L^{\infty}}\Big)\|\rho\textrm{II}_1(\widehat{\omega}_0)\|_{L^{2}}\\
&\leq C\al \|\widehat{\om}_0(\al,\cdot)\|_{H^1},
\end{align*}
which gives
\begin{align*}
\Big\|\pa_c\Big(\frac{\rho\textrm{II}_1(\widehat{\omega}_0)(c)}{\mathrm{A}^2+\mathrm{B}^2}\Big)\Big\|_{L^2}\leq C\al \|\widehat{\om}_0(\al,\cdot)\|_{H^1}.
\end{align*}
The $H^2$ estimate of $\mu(c)$ is similar. We left it to the interested readers.
\end{proof}

\section{Uniform Sobolev estimates of the vorticity}

Recall that $\widehat{\om}(t,\al,y)$ satisfies
\beno
\left\{
\begin{aligned}
&\partial_t\widehat{\omega}+i\al u\widehat{\omega}+i\al u''\widehat{\psi}=0,\\
&\widehat{\omega}|_{t=0}=\mathcal{F}\om_0(\al,y).
\end{aligned}
\right.
\eeno
This is equivalent to
\beno
\Big(e^{i\al t u(y)}\widehat{\omega}(t,\al,y)\Big)_t=-i\al e^{i\al t u(y)}u''(y)\widehat{\psi}(t,\al,y).
\eeno
Integration in $t$ gives
\beno
e^{i\al t u(y)}\widehat{\omega}(t,\al,y)=\widehat{\om}_0(\al,y)-i\al u''(y)\int_0^te^{i\tau u(y)}\widehat{\psi}(\tau,\al,y)d\tau.
\eeno
We denote $W(t,x,y)\triangleq\omega(t,x+u(y)t,y).$ We find that
\beno
\widehat{W}(t,\al,y)=e^{i\al t u(y)}\widehat{\omega}(t,\al,y).
\eeno

Then we get by (\ref{eq:psi-e}) that
\begin{align*}
\widehat{W}(t,\al,y)&=\widehat{\om}_0(\al,y)-\frac{u''(y)}{\pi}\int_{u(0)}^{u(1)}\frac{(e^{it(u(y)-c)\al}-1)\rho(c)\mu(c)\G(y,c)}{u(y)-c}dc\\
&\triangleq \widehat{\om}_0(\al,y)-\frac{u''(y)}{\pi}\mathbf{T}(\mu)(t,y),
\end{align*}
where
\beno
\G(y,c)=
\left\{
\begin{aligned}
&\phi(y,c)\int_0^y\frac{1}{\phi(z,c)^2}dz\triangleq \G_0(y,c)\quad0\leq y<y_c,\\
&\phi(y,c)\int_1^y\frac{1}{\phi(z,c)^2}dz\triangleq \G_1(y,c)\quad y_c<y\leq 1.
\end{aligned}
\right.
\eeno

We have the following uniform estimates in Sobolev space for $W(t,x,y)$.

\begin{proposition}\label{prop:vorticity}
With the same assumptions as in Theorem \ref{Thm:main}, it holds that
\beno
&&\|W(t)\|_{H^{-1}_xL^2_{y}}\le C\|w_0\|_{H^{-1}_xL^2_{y}},\\
&&\|W(t)\|_{H^{-1}_xH^1_y}\le C\|w_0\|_{H^{-1}_xH^1_y},\\
&&\|\rho(u(y))\pa_y^2W(t)\|_{H^{-1}_xL^2_{y}}\le C\|w_0\|_{H^{-1}_xH^2_y}.
\eeno
\end{proposition}

\begin{proof}
The proof is split into three steps.\medskip

{\bf Step 1.} $L^2$ estimate\medskip

Let $\varphi(y)\in C_0^{\infty}(0,1)$ with $\|\varphi\|_{L^2}=1$. A direct calculation gives
\begin{align*}
&\int_0^1\mathbf{T}(\mu)(t,y)\varphi(y)dy\\
&=\int_0^1\int_{u(0)}^{u(1)}\frac{(e^{it(u(y)-c)\al}-1)\rho(c)\mu(c)\G(y,c)}{u(y)-c}dc\varphi(y)dy\\
&=\int_0^1\int_{u(0)}^{u(y)}\frac{(e^{it(u(y)-c)\al}-1)\rho(c)\mu(c)\G(y,c)}{u(y)-c}dc\varphi(y)dy\\
&\quad+\int_0^1\int_{u(y)}^{u(1)}\frac{(e^{it(u(y)-c)\al}-1)\rho(c)\mu(c)\G(y,c)}{u(y)-c}dc\varphi(y)dy\\
&=\int_{u(0)}^{u(1)}\rho(c)\mu(c)\int_{y_c}^1\frac{(e^{it(u(y)-c)\al}-1)\G_1(y,c)\varphi(y)}{u(y)-c}dydc\\
&\quad+\int_{u(0)}^{u(1)}\rho(c)\mu(c)\int_{0}^{y_c}\frac{(e^{it(u(y)-c)\al}-1)\G_0(y,c)\varphi(y)}{u(y)-c}dydc\\
&=\int_{u(0)}^{u(1)}\rho(c)\mu(c)\int_{y_c}^1(e^{it(u(y)-c)\al}-1)\phi_1(y,c)\int_1^y\frac{1}{\phi(z,c)^2}dz\varphi(y)dydc\\
&\quad+\int_{u(0)}^{u(1)}\rho(c)\mu(c)\int_{0}^{y_c}(e^{it(u(y)-c)\al}-1)\phi_1(y,c)\int_0^y\frac{1}{\phi(z,c)^2}dz\varphi(y)dydc\\
&=\int_{u(0)}^{u(1)}\rho(c)\mu(c)\int_0^1\frac{-1}{\phi(z,c)^2}
\int_{y_c}^z(e^{it(u(y)-c)\al}-1)\varphi(y)\phi_1(y,c)dydzdc\\
&=-\int_{u(0)}^{u(1)}\rho(c)\mu(c)e^{-i\al ct}\mathbb{T}\Big((u^{-1})',\, \varphi\circ u^{-1}e^{it\al z}(u^{-1})',\,\phi_1(u^{-1},c),\,\frac{1}{\phi_1(u^{-1},c)^2}\Big)(c)dc\\
&\quad+\int_{u(0)}^{u(1)}\rho(c)\mu(c)\mathbb{T}\Big((u^{-1})', \varphi\circ u^{-1}(u^{-1})',\,\phi_1(u^{-1},c),\,\frac{1}{\phi_1(u^{-1},c)^2}\Big)(c)dc.
\end{align*}
Then by Lemma \ref{lem:T-L2} and Proposition \ref{prop:mu}, we get
\begin{align*}
&\int_0^1\mathbf{T}(\mu)(t,y)\varphi(y)dy\leq C\|\rho(c)\mu(c)\|_{L^2}\|\varphi\|_{L^2}\leq C\|\widehat{\omega}_0(\al,\cdot)\|_{L^2}.
\end{align*}
This gives
\beno
\|{W}(t)\|_{H^{-1}_xL^2_{y}}\leq C\|{\omega}_0\|_{H^{-1}_xL^2_{y}}.
\eeno

{\bf Step 2.} $H^1$ estimate\medskip

Let $\varphi(y)\in C_0^{\infty}(0,1)$ with $\|\varphi\|_{L^2}=1$. Then we have
\begin{align*}
&\int_0^1\partial_y\mathbf{T}(\rho\mu)(t,y)\varphi(y)dy
=-\int_0^1\mathbf{T}(\rho\mu)(t,y)\varphi'(y)dy\\
&=\int_{u(0)}^{u(1)}\rho(c)\mu(c)\int_0^1\frac{1}{\phi(y',c)^2}
\int_{y_c}^{y'}(e^{it(u(y)-c)\al}-1)\varphi'(y)\phi_1(y,c)dydy'dc\\
&=\int_{u(0)}^{u(1)}\rho(c)\mu(c)\mathbb{T}\Big((u^{-1})',\,(\varphi\circ u^{-1})',\,e^{it\al (z-c)}\phi_1(u^{-1},c),\,\frac{1}{\phi_1(u^{-1},c)^2}\Big)(c)dc\\
&\quad-\int_{u(0)}^{u(1)}\rho(c)\mu(c)\mathbb{T}\Big((u^{-1})',\,(\varphi\circ u^{-1})',\,\phi_1(u^{-1},c),\,\frac{1}{\phi_1(u^{-1},c)^2}\Big)(c)dc.
\end{align*}
We get by Lemma \ref{lem:iden-T2} that
\begin{align*}
&\int_0^1\partial_y\mathbf{T}(\rho\mu(\cdot))(t,y)\varphi(y)dy\\
&=-\int_{u(0)}^{u(1)}\rho(c)\mu(c)\partial_c\mathcal{L}_1(\varphi\circ u^{-1})(c)dc
+\int_{u(0)}^{u(1)}\rho(c)\mu(c)\mathcal{L}_2(\varphi\circ u^{-1})(c)dc\\
&\quad+\int_{u(0)}^{u(1)}\rho(c)\mu(c)\partial_c\big(e^{-it\al c}\mathcal{L}_1\big(\varphi\circ u^{-1}e^{it\al z}\big)(c)\big)dc
\\&\quad-\int_{u(0)}^{u(1)}\rho(c)\mu(c)e^{-it\al c}\mathcal{L}_2\big(\varphi\circ u^{-1}e^{it\al z}\big)(c)dc\\
&\quad+\int_{u(0)}^{u(1)}\rho(c)\mu(c)e^{-it\al c}\mathcal{B}\big(\varphi\circ u^{-1}e^{it\al z}\big)(c)dc
-\int_{u(0)}^{u(1)}\rho(c)\mu(c)\mathcal{B}(\varphi\circ u^{-1})(c)dc\\
&=\int_{u(0)}^{u(1)}\partial_c\big(\rho(c)\mu(c)\big)
\mathcal{L}_1(\varphi\circ u^{-1})(c)dc
+\int_{u(0)}^{u(1)}\rho(c)\mu(c)\mathcal{L}_2(\varphi\circ u^{-1})(c)dc\\
&\quad-\int_{u(0)}^{u(1)}\partial_c\big(\rho(c)\mu(c)\big)\big(e^{-it\al c}\mathcal{L}_1\big(\varphi\circ u^{-1}(z)e^{it\al z}\big)(c)\big)dc\\
&\quad-\int_{u(0)}^{u(1)}\rho(c)\mu(c)e^{-it\al c}\mathcal{L}_2\big(\varphi\circ u^{-1}e^{it\al z}\big)(c)dc\\
&\quad+\int_{u(0)}^{u(1)}\rho(c)\mu(c)e^{-it\al c}\mathcal{B}\big(\varphi\circ u^{-1}(z)e^{it\al z}\big)(c)dc
-\int_{u(0)}^{u(1)}\rho(c)\mu(c)\mathcal{B}(\varphi\circ u^{-1})(c)dc.
\end{align*}
Thus, by Lemma \ref{lem:L-H1}, Lemma \ref{lem:B-H1} and Proposition \ref{prop:mu}, we obtain
\begin{align*}
&\int_0^1\partial_y\mathbf{T}(\rho\mu(\cdot))(t,y)\varphi(y)dy\\
&\leq \|\partial_c\big(\rho(c)\mu(c)\big)\|_{L^2}\Big(\|\mathcal{L}_1(\varphi\circ u^{-1})\|_{L^2}
+\|\mathcal{L}_1(e^{it\al z}\varphi\circ u^{-1})\|_{L^2}\Big)\\
&\quad+C\|\rho\mu\|_{L^2}\Big(\|\mathcal{L}_2(\varphi\circ u^{-1})\|_{L^2}
+\|\mathcal{L}_2(e^{it\al z}\varphi\circ u^{-1})\|_{L^2}\\
&\qquad+\|\mathcal{B}(\varphi\circ u^{-1})\|_{L^2}
+\|\mathcal{B}(e^{it\al z}\varphi\circ u^{-1})\|_{L^2}\Big)\\
&\leq C\|\partial_c\big(\rho(c)\mu(c)\big)\|_{L^2}+C\|\rho\mu\|_{L^2}\le C\|\widehat{\omega_0}(\al,\cdot)\|_{H^1}.
\end{align*}
This gives
\beno
\|W(t)\|_{H^{-1}_xH^1_y}\leq C\|\om_0\|_{H^{-1}_xH^1_y}.
\eeno

{\bf Step 3.} $H^2$ estimate\medskip

Let $\varphi(y)\in C_0^{\infty}(0,1)$ with $\|\varphi\|_{L^2}=1$. Then we have
\begin{align*}
&\int_0^1\rho(u(y))\partial_y^2\mathbf{T}(\rho\mu(\cdot))(t,y)\varphi(y)dy\\
&=\int_0^1\mathbf{T}(\rho\mu(\cdot))(t,y)\partial_y^2\big(\varphi(y)\rho(u(y))\big)dy\\
&=-\int_{u(0)}^{u(1)}\rho(c)\mu(c)\int_0^1\frac{1}{\phi(y',c)^2}
\int_{y_c}^{y'}(e^{it(u(y)-c)\al}-1)\partial_y^2\big(\varphi(y)\rho(u(y))\big)\phi_1(y,c)dydy'dc,
\end{align*}
where
\begin{align*}
&\int_0^1\frac{1}{\phi^2(z,c)}
\int_{x_0}^z(e^{it(u(y)-c)\al}-1)\partial_y^2\big(\varphi(y)\rho(u(y))\big)\phi_1(y,c)dydz\\
&=\bbT\bigg((u^{-1})',\,\Big(\partial_y^2\big(\varphi(y)\rho(u(y))\big)\circ u^{-1}\Big)(u^{-1})',\,(e^{it(z-c)\al}-1)\phi_1(u^{-1}(z),c),\frac{1}{\phi_1(u^{-1}(z),c)^2}\bigg).
\end{align*}
We denote $F(z,c)=\phi_1(u^{-1}(z),c)$, $G(z,c)=\frac{1}{\phi_1(u^{-1}(z),c)^2}$ and $h(z)=\varphi(u^{-1}(z))\rho(z)$.
Using the fact that
\beno
(f\circ g)'=(f'\circ g)g',\quad
(f\circ g)''=(f''\circ g)(g')^2+\frac{(f\circ g)'g''}{g'},
\eeno
we deduce from Lemma \ref{Lem:iden-T1} and Lemma \ref{lem:iden-T2} that
\begin{align*}
&\rho(c)\int_0^1\frac{1}{\phi^2(z,c)}
\int_{x_0}^z(e^{it(u(y)-c)\al}-1)\partial_y^2\big(\varphi(y)\rho(u(y))\big)\phi_1(y,c)dydz\\
&=\rho(c)\bbT\Big((u^{-1})',\,h'',\,(e^{it(z-c)\al}-1)\frac{1}{(u^{-1})'}F,\,G\Big)\\
&\quad-\rho(c)\bbT\Big((u^{-1})',\,h',\,(e^{it(z-c)\al}-1)\frac{(u^{-1})''}{((u^{-1})')^2}F,\,G\Big)\\
&=\rho(c)\partial_c\bbT\Big((u^{-1})',\,h',\,(e^{it(z-c)\al}-1)\frac{1}{(u^{-1})'}F,\,G\Big)\\
&\quad-\rho(c)\bbT\Big((u^{-1})',\,h',\,(e^{it(z-c)\al}-1)(\partial_y+\partial_c)\big(\frac{1}{(u^{-1})'}F\big),\,G\Big)\\
&\quad-\rho(c)\bbT\Big((u^{-1})',\,h',\,(e^{it(z-c)\al}-1)\frac{1}{(u^{-1})'}F,\,(\partial_z+\partial_c)G\Big)\\
&\quad-\rho(c)\bbT\Big((u^{-1})'',\,h',\,(e^{it(z-c)\al}-1)\frac{1}{(u^{-1})'}F,\,G\Big)\\
&\quad+\mathbb{B}_0\Big((u^{-1})',\,h',\,(e^{it(z-c)\al}-1)\frac{1}{(u^{-1})'}F,\,G\Big)\\
&\quad+\mathbb{B}_1\Big((u^{-1})',\,h',\,(e^{it(z-c)\al}-1)\frac{1}{(u^{-1})'}F,\,G\Big)\\
&\quad-\rho(c)\partial_c\bbT\Big((u^{-1})',\,h,\,(e^{it(z-c)\al}-1)\frac{(u^{-1})''}{((u^{-1})')^2}F,\,G\Big)\\
&\quad+\rho(c)\bbT\Big((u^{-1})',\,h,\,(e^{it(z-c)\al}-1)(\partial_z+\partial_c)\big(\frac{(u^{-1})''}{((u^{-1})')^2}F\big),\,G\Big)\\
&\quad+\rho(c)\bbT\Big((u^{-1})',\,h,\,(e^{it(z-c)\al}-1)\frac{(u^{-1})''}{((u^{-1})')^2}F,\,(\partial_z+\partial_c)G\Big)\\
&\quad+\rho(c)\bbT\Big((u^{-1})'',\,h,\,(e^{it(z-c)\al}-1)\frac{(u^{-1})''}{((u^{-1})')^2}F,\,G\Big)\\
&\quad-\mathbb{B}_0\Big((u^{-1})',\,h,\,(e^{it(z-c)\al}-1)\frac{(u^{-1})''}{((u^{-1})')^2}F,\,G\Big)\\
&\quad-\mathbb{B}_1\Big((u^{-1})',\,h,\,(e^{it(z-c)\al}-1)\frac{(u^{-1})''}{((u^{-1})')^2}F,\,G\Big).
\end{align*}
This gives
\begin{align*}
&\rho(c)\int_0^1\frac{1}{\phi^2(z,c)}
\int_{x_0}^z(e^{it(u(y)-c)\al}-1)\partial_y^2\big(\varphi(y)\rho(u(y))\big)\phi_1(y,c)dydz\\
&=\rho(c)\partial_c\bbT\Big((u^{-1})',\,h',\,(e^{it(z-c)\al}-1)\frac{1}{(u^{-1})'}F,\,G\Big)\\
&\quad+\mathbb{B}_0\Big((u^{-1})',\,h',\,(e^{it(z-c)\al}-1)\frac{1}{(u^{-1})'}F,\,G\Big)
+\mathbb{B}_1\Big((u^{-1})',\,h',\,(e^{it(z-c)\al}-1)\frac{1}{(u^{-1})'}F,\,G\Big)\\
&\quad+\Xi_1(h)+\rho(c)\partial_c\Xi_2(h),
\end{align*}
where
\begin{align*}
\Xi_1(h)&=-e^{-it\al c}\mathcal{B}\Big(he^{itz}\frac{(u^{-1})''}{((u^{-1})')^2}\Big)
+\mathcal{B}\Big(h\frac{(u^{-1})''}{((u^{-1})')^2}\Big)\\
&\quad+\rho(c)e^{-it\al c}\mathcal{L}_2\Big(he^{itz}\frac{(u^{-1})''}
{((u^{-1})')^2}\Big)
-\rho(c)\mathcal{L}_2\Big(h\frac{(u^{-1})''}{((u^{-1})')^2}\Big)\\
&\quad+\rho(c)e^{-it\al c}\mathcal{L}_1\Big(he^{itz}\big(\frac{(u^{-1})''}
{((u^{-1})')^2}\big)'\Big)
-\rho(c)\mathcal{L}_1\Big(h\big(\frac{(u^{-1})''}
{((u^{-1})')^2}\big)'\Big)\\
&\quad+\rho(c)e^{-it\al c}\mathcal{L}_3\Big(he^{itz}\frac{1}{(u^{-1})'}\Big)
-\rho(c)\mathcal{L}_3\Big(h\frac{1}{(u^{-1})'}\Big)\\
&\quad-e^{-it\al c}\mathcal{B}_1\Big(h(z)e^{itz}\frac{1}{(u^{-1})'}\Big)
+\mathcal{B}_1\Big(h(z)\frac{1}{(u^{-1})'}\Big)\\
&\quad+\rho(c)e^{-it\al c}\mathcal{L}_1\Big(h(z)e^{itz}\big(\frac{1}{(u^{-1})'}\big)''\Big)
-\rho(c)\mathcal{L}_1\Big(h\big(\frac{1}{(u^{-1})'}\big)''\Big),
\end{align*}
and
\begin{align*}
\Xi_2(h)=&
e^{-it\al c}\mathcal{L}_1\Big(h(z)e^{itz}\frac{(u^{-1})''}
{((u^{-1})')^2}\Big)
-\mathcal{L}_1\Big(h\frac{(u^{-1})''}
{((u^{-1})')^2}\Big)\\
&\quad-e^{-it\al c}\mathcal{L}_2\Big(he^{itz}\frac{1}{(u^{-1})'}\Big)
-\mathcal{L}_2\Big(h\frac{1}{(u^{-1})'}\Big)\\
&\quad-e^{-it\al c}\mathcal{L}_2\Big(he^{itz}\big(\frac{1}{(u^{-1})'}\big)'\Big)
-\mathcal{L}_2\Big(h\big(\frac{1}{(u^{-1})'}\big)'\Big).
\end{align*}

By Lemma \ref{lem:L-H1} and Lemma \ref{lem:B-H1}, we have
\ben\label{eq:The1}
\|\Xi_1(h)\|_{L^2}+\|\Xi_2(h)\|_{L^2}\leq C\|h\|_{L^2}\le C\|\varphi\|_{L^2}.
\een

Using Lemma \ref{lem:iden-B01}, we write
\begin{align*}
&\mathbb{B}_0\Big((u^{-1})',\,h',\,(e^{it(z-c)\al}-1)\frac{1}{(u^{-1})'}F,\,G\Big)\\
&\quad+\mathbb{B}_1\Big((u^{-1})',\,h',\,(e^{it(z-c)\al}-1)\frac{1}{(u^{-1})'}F,\,G\Big)\\
&=\partial_c\mathbb{B}_0\Big((u^{-1})',\,h(e^{it(z-c)\al}-1)\frac{1}{(u^{-1})'},\,F,\,G\Big)\\
&\quad+\partial_c\mathbb{B}_1\Big((u^{-1})',\,h(e^{it(z-c)\al}-1)\frac{1}{(u^{-1})'},\,F,\,G\Big)\\
&\quad-\mathbb{B}_0\Big((u^{-1})',\,h,\,(e^{it(z-c)\al}-1)\frac{1}{(u^{-1})'}F,\,\partial_cG\Big)\\
&\quad-\mathbb{B}_1\Big((u^{-1})',\,h,\,(e^{it(z-c)\al}-1)\frac{1}{(u^{-1})'}F,\,\partial_cG\Big)\\
&\quad-\mathbb{B}_0\Big((u^{-1})',\,h(e^{it(z-c)\al}-1)\frac{1}{(u^{-1})'},\,(\partial_z+\partial_c)F,\,\partial_cG\Big)\\
&\quad-\mathbb{B}_1\Big((u^{-1})',\,h(e^{it(z-c)\al}-1)\frac{1}{(u^{-1})'},\,(\partial_z+\partial_c)F,\,\partial_cG\Big)\\
&\quad-\mathbb{B}_0\Big((u^{-1})',\,h(e^{it(z-c)\al}-1)\big(\frac{1}{(u^{-1})'}\big)',\,F,\,\partial_cG\Big)\\
&\quad-\mathbb{B}_1\Big((u^{-1})',\,h(e^{it(z-c)\al}-1)\big(\frac{1}{(u^{-1})'}\big)',\,F,\,\partial_cG\Big)\\
&\quad+\frac{(u^{-1})'(u(0))G(u(0),c)}{c-u(0)}\int_{u(0)}^ch(z)(e^{it(z-c)\al}-1)\frac{1}{(u^{-1})'}F(z,c)dz\\
&\quad+\frac{(u^{-1})'(u(1))G(u(1),c)}{u(1)-c}\int_c^{u(1)}h(z)(e^{it(z-c)\al}-1)\frac{1}{(u^{-1})'}F(z,c)dz\\
&\triangleq \partial_c\Xi_4(h)+\Xi_3(h).
\end{align*}
It follows from Lemma \ref{lem:B01-L2} that
\ben\label{eq:The34}
\|\Xi_3(h)\|_{L^2}\leq C\|\varphi\|_{L^2},\quad
\|\Xi_4(h)\|_{L^2}\leq \frac{C}{\al}\|\varphi\|_{L^2}.
\een

Using Lemma \ref{Lem:iden-T1}, we write
\begin{align*}
&\rho(c)\partial_c\bbT\Big((u^{-1})',\,h',\,(e^{it(z-c)\al}-1)\frac{1}{(u^{-1})'}F,\,G\Big)\\
&=\rho(c)\partial_c^2\bbT\Big((u^{-1})',\,h(e^{it(z-c)\al}-1)\frac{1}{(u^{-1})'},\,F,\,G\Big)\\
&\quad-\rho(c)\partial_c\bbT\Big((u^{-1})',\,h,\,(e^{it(z-c)\al}-1)(\partial_c+\partial_y)\big(\frac{1}{(u^{-1})'}F\big),\,G\Big)\\
&\quad-\rho(c)\partial_c\bbT\Big((u^{-1})',\,h,\,(e^{it(z-c)\al}-1)\frac{1}{(u^{-1})'}F,\,(\partial_c+\partial_z)G\Big)\\
&\quad-\rho(c)\partial_c\bbT\Big((u^{-1})'',\,h,\,(e^{it(z-c)\al}-1)\frac{1}{(u^{-1})'}F,\,G\Big)\\
&\quad+\rho(c)\partial_c\frac{1}{\rho}\mathbb{B}_0\Big((u^{-1})',\,h,\,(e^{it(z-c)\al}-1)\frac{1}{(u^{-1})'}F,\,G\Big)\\
&\quad+\rho(c)\partial_c\frac{1}{\rho}\mathbb{B}_1\Big((u^{-1})',\,h,\,(e^{it(z-c)\al}-1)\frac{1}{(u^{-1})'}F,\,G\Big)\\
&=\rho(c)\partial_c^2[\mathcal{L}_1,\rho]\Big(\varphi(e^{it(z-c)\al}-1)\frac{1}{(u^{-1})'}\Big)
+\rho(c)\partial_c^2\bigg(\rho\mathcal{L}_1\Big(\varphi(e^{it(z-c)\al}-1)\frac{1}{(u^{-1})'}\Big)\bigg)\\
&\quad-\rho(c)\partial_c\mathcal{L}_2\Big(h(e^{it(z-c)\al}-1)\frac{1}{(u^{-1})'}\Big)
-\rho(c)\partial_c\mathcal{L}_1\Big(h(e^{it(z-c)\al}-1)\big(\frac{1}{(u^{-1})'}\big)'\Big)\\
&\quad+\rho(c)\partial_c\frac{1}{\rho}\mathcal{B}\Big(h(e^{it(z-c)\al}-1)\frac{1}{(u^{-1})'}\Big)\\
&\triangleq\rho(c)\partial_c\Xi_5(\varphi)
+\rho\partial_c^2(\rho\Xi_6(\varphi))
+\rho\partial_c\Xi_7(h).
\end{align*}
This together with Lemma \ref{lem:L-com}, Lemma \ref{lem:L-H1} and Lemma \ref{lem:B-H1} shows that
\ben\label{eq:The567}
\|\Xi_5(\varphi)\|_{L^2}+\|\Xi_6(\varphi)\|_{L^2}
+\|\Xi_7(h)\|_{L^2}\leq C\|\varphi\|_{L^2}.
\een
Thus, we obtain
\begin{align*}
&\int_0^1\rho(u(y))\partial_y^2\mathbf{T}(\rho\mu(\cdot))(t,y)\varphi(y)dy\\
&=-\int_{u(0)}^{u(1)}\mu(c)\Big(\Xi_1(h)
+\rho(c)\partial_c\Xi_2(h)
+\Xi_3(h)
+\partial_c\Xi_4(h)\\
&\quad+\rho(c)\partial_c\Xi_5(\varphi)
+\rho\partial_c^2(\rho\Xi_6(\varphi))
+\rho\partial_c\Xi_7(h)\Big)dc\\
&=-\int_{u(0)}^{u(1)}\mu(c)\Xi_1(h)
-\partial_c(\rho\mu)\Xi_2(h)
+\mu(c)\Xi_3(h)
-\Xi_4(h)\partial_c\mu\\
&\quad-\partial_c(\rho\mu)\Xi_5(\varphi)
+\rho(c)\partial_c^2(\rho\mu)\Xi_6(\varphi)
-\partial_c(\rho\mu)\Xi_7(h)dc,
\end{align*}
which along with (\ref{eq:The1})-(\ref{eq:The567}) and Proposition \ref{prop:mu} yields
\begin{align*}
\|\rho(u(y))\partial_y^2\mathbf{T}(\rho\mu)(t,y)\|_{L^2}
\leq C\|\widehat{\om}_0(\al,\cdot)\|_{H^2}.
\end{align*}
This implies
\beno
\|\rho(u(\cdot))\partial_y^2W(t)\|_{H^{-1}_xH^2_y}\leq C\|\om_0\|_{H^{-1}_xH^2_y}.
\eeno

The proof of the proposition is completed.
\end{proof}

\section{Proof of Theorem \ref{Thm:main}}

This section is devoted to the proof of Theorem \ref{Thm:main}.

\subsection{Decay estimates}

With the uniform Sobolev estimates of the vorticity, the decay estimates of the velocity are similar to \cite{LZ}.
For the completeness, we present a proof. By the elliptical estimate, we get
\beno
\|V\|_{L^2}\leq C\|\omega\|_{H^{-1}}.
\eeno
We get by duality that
\begin{align}\label{equ:dual}\nonumber
\|V(t)\|_{L^2}\leq&
\sup_{\varphi\in C^{\infty}_0,~\|\varphi\|_{H^1}\leq 1}
\bigg|\int_0^1\int_{-\pi}^{\pi}\varphi\omega(t)dxdy\bigg|\\
\leq &C\sup_{\varphi\in C^{\infty}_0,~\|\varphi\|_{H^1}\leq 1}
\bigg|\sum_{|\al|\neq 0}\int_0^1\overline{\widehat{\varphi}(\al,y)}\widehat{W}(t,\al,y)e^{-i\al u(y)t}dy\bigg|.
\end{align}
Integration by parts gives
\begin{align*}
\int_0^1\overline{\widehat{\varphi}(\al,y)}\widehat{W}(t,\al,y)e^{-i\al u(y)t}dy
=&\frac{e^{-i\al u(y)t}}{-i\al tu'(y)}\overline{\widehat{\varphi}(\al,y)}\widehat{W}(t,\al,y)\bigg|_{y=0}^1\\
&-\int_0^1\frac{e^{-i\al u(y)t}}{-i\al t}\partial_y\Big(\frac{\overline{\widehat{\varphi}(\al,y)}\widehat{W}(t,\al,y)}{u'(y)}\Big)dy\\
=&-\int_0^1\frac{e^{-i\al u(y)t}}{-i\al t}\partial_y\Big(\frac{\overline{\widehat{\varphi}(\al,y)}\widehat{W}(t,\al,y)}{u'(y)}\Big)dy,
\end{align*}
from which and Proposition \ref{prop:vorticity}, we infer that
\begin{align*}
\|V(t)\|_{L^2}\leq& Ct^{-1}\sup_{\varphi\in C^{\infty}_0,~\|\varphi\|_{H^1}\leq 1}
\|W(t)\|_{H_x^{-1}H_y^1}\|\varphi\|_{L^2_xH^1_y}\\
\leq& Ct^{-1}\|\omega_0\|_{H_x^{-1}H_y^1}.
\end{align*}
Also we get from (\ref{equ:dual}) that
\begin{align*}
\|V(t)\|_{L^2}\leq& C\sup_{\varphi\in C^{\infty}_0,~\|\varphi\|_{H^1}\leq 1}\|W(t)\|_{H_x^{-1}L_y^2}\|\varphi\|_{H^1_xL^2_y}\\
\leq& C\|\omega_0\|_{H_x^{-1}L_y^2}.
\end{align*}

Recall that $-\triangle V^2=\partial_x\omega$ and $V^2(x,0)=V^2(1,y)=0$. We define
\beno
-\Delta \vartheta=V^2,\quad \vartheta(t,x,0)=\vartheta(t,x,1)=0.
\eeno
We get by the elliptic estimate that
\beno
\|\vartheta\|_{H^2}\le C\|V^2\|_{L^2}.
\eeno
Integration by parts gives
\begin{align*}
\|V^2\|_{L^2}^2=&\iint\partial_x\omega \vartheta dxdy
=\sum_{|\al|\neq 0}\int_0^1i\al e^{-i\al u(y)t}\widehat{W}(t,\al,y)\overline{\widehat{\vartheta}}(t,\al,y)dy\\
=&\sum_{|\al|\neq 0}\int_0^1\frac{e^{-i\al u(y)t}}{t}\partial_y\Big(\frac{\widehat{W}(t,\al,y)\overline{\widehat{\vartheta}}(t,\al,y)}{u'(y)}\Big)dy\\
=&\sum_{|\al|\neq 0}\frac{e^{-i\al u(y)t}}{i\al t^2 u'}
\partial_y\Big(\frac{\widehat{W}(t,\al,y)\overline{\widehat{\vartheta}}(t,\al,y)}{u'(y)}\Big)\bigg|_{y=0}^1\\
&-\sum_{|\al|\neq 0}\int_0^1\frac{e^{-i\al u(y)t}}{i\al t^2}\partial_y\bigg(\frac{1}{u'}\partial_y\Big(\frac{\widehat{W}(t,\al,y)\overline{\widehat{\vartheta}}(t,\al,y)}{u'(y)}\Big)\bigg)dy.
\end{align*}
By Sobolev embedding and Proposition \ref{prop:vorticity}, the first term on the right hand side is bounded by
\begin{align*}
&Ct^{-2}\sum_{|\al|\neq 0}\frac{1}{\al}\|\widehat{W}(t,\al,\cdot)\|_{H^1_y}\big\|\vartheta(t,\al,\cdot)\big\|_{H^2_y}\le Ct^{-2}\|\om_0\|_{H^{-1}_xH^2_y}\|V^2\|_{L^2}.
\end{align*}
Using Proposition \ref{prop:vorticity} again and Hardy inequality, we get
\begin{align*}
\bigg\|\partial_y\bigg(\frac{1}{u'}\partial_y\Big(\frac{\widehat{W}(t,\al,y)\overline{\widehat{\vartheta}}(t,\al,y)}{u'(y)}\Big)\bigg)\bigg\|_{L^1_y}
\leq& C\|\rho(u(y))\partial_y^2\widehat{W}(t,\al,y)\|_{L^2_y}\Big\|\frac{\vartheta(t,\al,y)}{\rho(u(y))}\Big\|_{L^{2}_y}\\
&+C\|\partial_y\widehat{W}(t,\al,y)\|_{L^2_y}\big\|\vartheta(t,\al,y)\big\|_{H^1_y}\\
&+C\|\widehat{W}(t,\al,y)\|_{L^2_y}\big\|\vartheta(t,\al,y)\big\|_{H^2_y}\\
\leq & C\|\omega_0(\al,\cdot)\|_{H^2_y}\big\|\vartheta(t,\al,y)\big\|_{H^2_y}.
\end{align*}
This implies that the second term is bounded by
\beno
Ct^{-2}\|\om_0\|_{H^{-1}_xH^2_y}\|V^2\|_{L^2}.
\eeno

Thus, we show that
\begin{align*}
\|V^2\|_{L^2} \leq Ct^{-2}\|\omega_0\|_{H^{-1}_xH_y^2}.
\end{align*}

\subsection{Scattering}
We only prove $H^1$ scattering. The proof of $L^2$ scattering is similar.
Suppose that $\om_0\in H^{-1}_xH^1_y$.
Let $\{\om^{(n)}_0\}_{n\geq 1}\subseteq H^{-1}_xH^2_y$ be a sequence
such that $\om^{(n)}_0\longrightarrow \om_0$ in $H^{-1}_xH^1_y$. Let $\om^{(n)}(t,x,y)$ be the solution to linearized Euler equation with initial data $\om^{(n)}_0(x,y)$.
Then $\om^{(n)}\longrightarrow \om$ in $L^{\infty}_t(H^{-1}_xH^1_y)$.
Thanks to $W_t^{(n)}=u''(y)\pa_x\psi^{(n)}$, we get
\beno
W^{(n)}(t,x,y)=\om^{(n)}(0,x,y)-\int_0^tu''(y)\pa_x\psi^{(n)}(s,x+su(y),y)ds.
\eeno
By the decay estimate, we have
\beno
\|\pa_x\psi^{(n)}(s,x,y)\|_{L^2_{x,y}}\leq \frac{C}{\langle t\rangle^2}\|\omega_0^{(n)}\|_{H^{-1}_xH^2_y}.
\eeno
Letting $t\to \infty$, we get
\beno
W^{(n)}(t,x,y)\longrightarrow \om^{(n)}(0,x,y)-\int_0^{\infty}u''(y)\pa_x\psi^{(n)}(s,x+su(y),y)ds=\om_{\infty}^{(n)}\quad \textrm{in }H^{-1}_xL^2_{y}.
\eeno
By Proposition \ref{prop:vorticity}, $\{W^{(n)}\}_{n\geq 1}$ is a Cauchy sequence in $L^{\infty}_t(H^{-1}_xH_y^1)$.
Thus, $\om_{\infty}^{(n)}$ is a Cauchy sequence in $H^{-1}_xH_y^1$, then $\om^{\infty}=\lim_{n\to \infty}\om_{\infty}^{(n)}\in H^{-1}_xH_y^1$.

\section{Appendix}

\subsection{Multilinear singular integral operators}

The multilinear operator $\mathbb{T}$ is defined by
\beno
\bbT(f,g,F,G)(c)\eqdef\textrm{p.v.}\int_{u(0)}^{u(1)}\frac{\int_{c}^{c'}g(z)F(z,c)dz}{(c-c')^2}G(c',c)f(c')dc',
\eeno
where $f,g$ are functions defined on $D_0$ and $F,G$ are functions defined on $D_0\times D_0$.

The multilinear operator $\mathbb{B}_0$ and $\mathbb{B}_1$ are defined by
\beno
&&\mathbb{B}_0(f,g,F,G)(c)\eqdef\frac{(u(1)-c)f(u(0))}{c-u(0)}\int_{u(0)}^cg(z)F(z,c)G(u(0),c)dz,\\
&&\mathbb{B}_1(f,g,F,G)(c)\eqdef\frac{(c-u(0))f(u(1))}{u(1)-c}\int_{c}^{u(1)}g(z)F(z,c)G(u(1),c)dz.
\eeno

\begin{lemma}\label{Lem:iden-T1}
It holds that
\begin{align*}
\partial_c\big(\rho(c)\mathbb{T}(f,g,F,G)(c)\big)
=&\rho(c)\mathbb{T}(f,g',F,G)(c)
+\rho(c)\mathbb{T}(f,g,(\partial_z+\partial_c)F,G)(c)
\\
&+\rho(c)\mathbb{T}(f,g,F,(\partial_z+\partial_c)G)(c)\\
&+\rho(c)\mathbb{T}(f',g,F,G)(c)
+\rho'(c)\mathbb{T}(f,g,F,G)(c)\\
&-\mathbb{B}_0(f,g,F,G)(c)
-\mathbb{B}_1(f,g,F,G)(c).
\end{align*}
\end{lemma}

\begin{proof}
Thanks to
\beno
\rho(c)-\rho(c')=(c'-c)(c'+c-u(0)-u(1)),
\eeno
we get
\begin{align*}
&\rho(c)\mathbb{T}(f,g,F,G)(c)\\
&=\textrm{p.v.}\int_{u(0)}^{u(1)}\frac{(c'+c-u(0)-u(1))\int_{c}^{c'}g(z)F(z,c)dz}{c'-c}G(c',c)f(c')dc'\\
&\quad+\textrm{p.v.}\int_{u(0)}^{u(1)}\frac{\int_{c}^{c'}g(z)F(z,c)dz}{(c'-c)^2}G(c',c)f(c')\rho(c')dc'.
\end{align*}
A direct calculation shows
\begin{align*}
&\partial_c\bigg(\frac{\int_{c}^{c'}g(z)F(z,c)dz}
{c'-c}\bigg)\\
&=-\partial_{c'}\bigg(\frac{\int_{c}^{c'}g(z)F(z,c)dz}{c'-c}\bigg)
+\frac{\int_{c}^{c'}g'(z)F(z,c)dz}{c'-c}
+\frac{\int_{c}^{c'}g(z)(\partial_c+\partial_z)F(z,c)dz}{c'-c},
\end{align*}
and
\begin{align*}
&\partial_c\bigg(\frac{\int_{c}^{c'}g(z)F(z,c)dz}
{(z-c)^2}\bigg)\\
&=-\partial_{c'}\bigg(\frac{\int_{c}^{c'}g(z)F(z,c)dz}{(c'-c)^2}\bigg)
+\frac{\int_{c}^{c'}g'(z)F(z,c)dz}{(c'-c)^2}
+\frac{\int_{c}^{c'}g(z)(\partial_c+\partial_z)F(z,c)dy}{(c'-c)^2}.
\end{align*}
Thus, we obtain
\begin{align*}
&\partial_c\big(\rho(c)\mathbb{T}(f,g,F,G)(c)\big)\\
=&\textrm{p.v.}\int_{u(0)}^{u(1)}\frac{\int_{c}^{c'}g(z)F(z,c)dz}{(c'-c)}(\partial_c+\partial_{c'})\Big(G(c',c)f(c')(c'+c-u(0)-u(1))\Big)dc'\\
&+\textrm{p.v.}\int_{u(0)}^{u(1)}\frac{\int_{c}^{c'}g(z)F(z,c)dz}{(c'-c)^2}(\partial_c+\partial_{c'})\Big(G(c',c)f(c')\rho(c')\Big)dc'\\
&+\textrm{p.v.}\int_{u(0)}^{u(1)}\frac{\int_{c}^{c'}g(z)(\partial_c+\partial_z)F(z,c)dz}{(c'-c)}G(c',c)f(c')(c'+c-u(0)-u(1))dc'\\
&+\textrm{p.v.}\int_{u(0)}^{u(1)}\frac{\int_{c}^{c'}g(z)(\partial_c+\partial_z)F(z,c)dz}{(c'-c)^2}G(c',c)f(c')\rho(c')dc'\\
&+\textrm{p.v.}\int_{u(0)}^{u(1)}\frac{\int_{c}^{c'}g'(z)F(z,c)dz}{(c'-c)}G(c',c)f(c')(c'+c-u(0)-u(1))dc'\\
&+\textrm{p.v.}\int_{u(0)}^{u(1)}\frac{\int_{c}^{c'}g'(z)F(z,c)dz}{(c'-c)^2}G(c',c)f(c')\rho(c')dc'\\
&-\frac{(c-u(0))f(u(1))}{u(1)-c}\int_{c}^{u(1)}g(z)F(z,c)G(u(1),c)dz\\
&+\frac{(c-u(1))f(u(0))}{c-u(0)}\int_{u(0)}^cg(z)F(z,c)G(u(0),c)dz.
\end{align*}
Then the lemma follows by using
\beno
&&(c'-c)(\partial_c+\partial_{c'})\big((c'+c-u(0)-u(1))f(c')\big)
+\partial_{c'}\big(f(c')\rho(c')\big)\\
&&=\rho'(c)f(c')+\rho(c)f'(c'),
\eeno
and $\rho(c)-\rho(c')=(c'-c)(c'+c-u(0)-u(1))$.
\end{proof}

By Lemma \ref{Lem:iden-T1} and $(\partial_z+\partial_c)e^{it\al(z-c)}=0$, we can deduce that

\begin{lemma}\label{lem:iden-T2}
It holds that
\begin{align*}
\rho(c)\mathbb{T}(f, g',e^{it\al(z-c)}F,G)(c)
=&\rho(c)\partial_c\big(e^{-it\al c}\mathbb{T}(f, ge^{it\al z},F,G)\big)(c)\\
&-\rho(c)e^{-it\al c}\mathbb{T}(f, ge^{it\al z},(\partial_z+\partial_c)F,G)(c)\\
&-\rho(c)e^{-it\al c}\mathbb{T}(f, g e^{it\al z},F,(\partial_z+\partial_c)G)(c)\\\nonumber
&-\rho(c)e^{-it\al c}\mathbb{T}(f', ge^{it\al z},F,G)(c)\\
&+e^{-it\al c}\mathbb{B}_0(f, ge^{it\al z},F,G)(c)
+e^{-it\al c}\mathbb{B}_1(f, ge^{it\al z},F,G)(c).
\end{align*}
\end{lemma}

\begin{lemma}\label{lem:iden-B01}
It holds that
\begin{align*}
\partial_c\mathbb{B}_0(f,g,F,G)(c)
=&\mathbb{B}_0(f,g,F,\partial_cG)(c)
-\frac{f(u(0))G(u(0),c)}{c-u(0)}\int_{u(0)}^cg(z)F(z,c)dz\\
&+\mathbb{B}_0(f,g',F,G)(c)
+\mathbb{B}_0(f,g, (\partial_z+\partial_c)F,G)(c),\\
\partial_c\mathbb{B}_1(f,g,F,G)(c)
=&\mathbb{B}_1(f,g,F,\partial_cG)(c)
-\frac{f(u(1))G(u(0),c)}{u(1)-c}\int_c^{u(1)}g(z)F(z,c)dz\\
&+\mathbb{B}_1(f,g',F,G)(c)
+\mathbb{B}_1(f,g,(\partial_z+\partial_c)F,G)(c).
\end{align*}
\end{lemma}
\begin{proof}
The lemma follows by a direct calculation.
\end{proof}

Next let us introduce some linear operators. Let $\phi(y,c)$ be the solution of (\ref{eq:Rayleigh-H}) given by Proposition \ref{prop:Rayleigh-Hom}
and $\phi_1(y,c)=\frac {\phi(y,c)} {u(y)-c}$. We define
\begin{align*}
\mathcal{L}_1(g)(c)\eqdef&\mathbb{T}\Big((u^{-1})', g, \phi_1(u^{-1},c),\,\frac{1}{\phi_1(u^{-1},c)^2}\Big)(c),\\
\mathcal{L}_2(g)(c)\eqdef&\mathbb{T}\Big((u^{-1})', g,\,\,(\partial_z+\partial_c)\phi_1(u^{-1},c),\,\frac{1}{\phi_1(u^{-1},c)^2}\Big)(c)\\
&+\mathbb{T}\Big((u^{-1})', g,\phi_1(u^{-1},c),\,(\partial_z+\partial_c)\big(\frac{1}{\phi_1(u^{-1},c)^2}\big)\Big)(c)\\
&+\mathbb{T}\Big((u^{-1})'', g,\phi_1(u^{-1},c),\,\frac{1}{\phi_1(u^{-1},c)^2}\Big)(c),\\
\mathcal{L}_3(g)(c)\eqdef&
\mathbb{T}\Big((u^{-1})', g,(\partial_z+\partial_c)^2\phi_1(u^{-1},c),\frac{1}{\phi_1(u^{-1},c)^2}\Big)(c)\\
&+\mathbb{T}\Big((u^{-1})', g,\phi_1(u^{-1},c),(\partial_z+\partial_c)^2\big(\frac{1}{\phi_1(u^{-1},c)^2}\big)\Big)(c)\\
&+\mathbb{T}\Big((u^{-1})''', g,\phi_1(u^{-1},c),\frac{1}{\phi_1(u^{-1},c)^2}\Big)(c)\\
&+2\mathbb{T}\Big((u^{-1})'', g, (\partial_z+\partial_c)\phi_1(u^{-1},c),\frac{1}{\phi_1(u^{-1},c)^2}\Big)(c)\\
&+2\mathbb{T}\Big((u^{-1})', g, (\partial_z+\partial_c)\phi_1(u^{-1},c),(\partial_z+\partial_c)\big(\frac{1}{\phi_1(u^{-1},c)^2}\big)\Big)(c)\\
&+2\mathbb{T}\Big((u^{-1})'', g, \phi_1(u^{-1},c),(\partial_z+\partial_c)\big(\frac{1}{\phi_1(u^{-1},c)^2}\big)\Big)(c).
\end{align*}
We also introduce the linear operators
\begin{align*}
\mathcal{B}(g)(c)\eqdef&\mathbb{B}_0\Big((u^{-1})', g,\phi_1(u^{-1},c),\,\frac{1}{\phi_1(u^{-1},c)^2}\Big)(c)\\
&+\mathbb{B}_1\Big((u^{-1})', g,\phi_1(u^{-1},c),\,\frac{1}{\phi_1(u^{-1},c)^2}\Big)(c),\\
\mathcal{B}_1(g)(c)\eqdef&\mathbb{B}_0\Big((u^{-1})', g,(\partial_z+\partial_c)\phi_1(u^{-1},c),\,\frac{1}{\phi_1(u^{-1},c)^2}\Big)(c)\\
&+\mathbb{B}_1\Big((u^{-1})', g,(\partial_z+\partial_c)\phi_1(u^{-1},c),\,\frac{1}{\phi_1(u^{-1},c)^2}\Big)(c)\\
&+\mathbb{B}_0\Big((u^{-1})', g,\phi_1(u^{-1},c),\,(\partial_z+\partial_c)\big(\frac{1}{\phi_1(u^{-1},c)^2}\big)\Big)(c)\\
&+\mathbb{B}_1\Big((u^{-1})', g,\phi_1(u^{-1},c),\,(\partial_z+\partial_c)\big(\frac{1}{\phi_1(u^{-1},c)^2}\big)\Big)(c)\\
&+\mathbb{B}_0\Big((u^{-1})'', g,\phi_1(u^{-1},c),\,\frac{1}{\phi_1(u^{-1},c)^2}\Big)(c)\\
&+\mathbb{B}_1\Big((u^{-1})'', g,\phi_1(u^{-1},c),\,\frac{1}{\phi_1(u^{-1},c)^2}\Big)(c).
\end{align*}

\begin{lemma}\label{lem:iden-L}
It holds that
\begin{align*}
&\partial_c\big(\rho\mathcal{L}_1(g)\big)=\rho(c)\mathcal{L}_1(g')
+\rho'(c)\mathcal{L}_1(g)+\rho(c)\mathcal{L}_2(g)+\mathcal{B}(g),\\
&\rho\partial_c\mathcal{L}_2(g)=\rho\mathcal{L}_2(g')+\rho\mathcal{L}_3(g)-\mathcal{B}_1(g),
\end{align*}
and
\begin{align*}
\partial_c^2(\rho^2\mathcal{L}_1(g))=&2\rho''\rho\mathcal{L}_1(g)
+4\rho'\rho\mathcal{L}_1(g')
+2\rho'\rho'\mathcal{L}_1(g)
+4\rho'\rho\mathcal{L}_2(g)
-2\rho'\mathcal{B}(g)\\\nonumber
&+\rho^2\mathcal{L}_1(g'')
+2\rho^2\mathcal{L}_2(g')
-\rho\mathcal{B}(g')
+\rho^2\mathcal{L}_3(g)
-\rho\mathcal{B}_1(g)
-\partial_c\big(\rho\mathcal{B}(g)\big).
\end{align*}

\begin{proof}
The lemma can be proved by using Lemma \ref{Lem:iden-T1}. We omit the details.
\end{proof}

\end{lemma}

\subsection{Boundedness of singular integral operators}
Let $H$ be the Hilbert operator which is defined by
\beno
Hf(c)=\textrm{p.v.}\int\frac {f(c')} {c-c'}dc'.
\eeno
The maximal Hilbert transform $H^*$ is defined by
\beno
H^{\ast}f(c)=\sup_{\epsilon>0}|H^{\epsilon}f(c)|,\quad H^{\epsilon}f(c)=\int_{|c-c'|\geq \epsilon}\frac{f(c')}{c-c'}dc'.
\eeno
We denote
\beno
\mathrm{H}f(c)\triangleq\chi_{D_0}(c)H(f\chi_{D_0})(c).
\eeno

\begin{lemma}\label{lem:H-H2}
Suppose that $f\in H^2(D_0)$, then it holds that for any $p\in (1,\infty)$,
\begin{align*}
&\|\mathrm{H}f\|_{L^{p}}\leq C\|f\|_{L^p},\\
&\|\rho\mathrm{H}f\|_{W^{1,p}}\leq C\|f\|_{W^{1,p}},\\
&\|\rho^2\mathrm{H}f\|_{W^{2,p}}\leq C\|f\|_{W^{2,p}}.
\end{align*}
\end{lemma}
\begin{proof}
Thanks to
\begin{align*}
[\rho,\mathrm{H}]f(c)
=&\int_{u(0)}^{u(1)}\frac{\rho(c)-\rho(c')}{c-c'}f(c')dc'\\
=&\int_{u(0)}^{u(1)}(u(0)+u(1)-c-c')f(c')dc',
\end{align*}
we deduce that
\begin{align}\label{eq:app-com}
\|\partial_c[\rho,\mathrm{H}]f\|_{L^p}
+\|\partial_c^2[\rho^2,\mathrm{H}]f\|_{L^p}\leq C\|f\|_{L^{p}}.
\end{align}
We have
\begin{align*}
\partial_c\Big(\textrm{p.v.}\int_{u(0)}^{u(1)}\frac{\rho(c')f(c')}{c-c'}dc'\Big)
=&\partial_c\Big(\textrm{p.v.}\int_{c-u(1)}^{c-u(0)}\frac{\rho(c-z)f(c-z)}{z}dz\Big)\\
=&\textrm{p.v.}\int_{c-u(1)}^{c-u(0)}\frac{(\rho f)'(c-z)}{z}dz
=\textrm{p.v.}\int_{u(0)}^{u(1)}\frac{(\rho f)'(c')}{c-c'}dc',
\end{align*}
from which and $L^p$ boundedness of Hilbert transform, we infer that
\begin{align*}
&\Big\|\partial_c\Big(\textrm{p.v.}\int_{u(0)}^{u(1)}\frac{\rho(c')f(c')}{c-c'}dc'\Big)\Big\|_{L^p}\leq C\|(\rho f)'\|_{L^p},\\
&\Big\|\partial_c^2\Big(\textrm{p.v.}\int_{u(0)}^{u(1)}\frac{\rho(c')^2f(c')}{c-c'}dc'\Big)\Big\|_{L^p}\leq C_0\|(\rho^2 f)''\|_{L^p},
\end{align*}
which along with (\ref{eq:app-com}) gives the lemma.
\end{proof}

\begin{lemma}\label{lem:sing2}
There exists a constant $C$ independent of $\epsilon\in (0,(u(1)-u(0))/10)$ such that for any $p\in (1,\infty)$
\beno
&&\bigg\|\sup_{\epsilon}\bigg|\int_{\{|c-c'|>\epsilon\}\cap[u(0),u(1)]}\frac{\int_{c}^{c'}f(z)dz}{(c-c')^2}dc'\bigg|
\bigg\|_{L^p}\leq C\|f\|_{L^p},\\
&&\bigg\|\textrm{p.v.}\int_{u(0)}^{u(1)}\frac{\int_{c}^{c'}f(z)dz}{(c-c')^2}dc'\bigg\|_{L^p}
\leq C\|f\|_{L^p}.
\eeno
\end{lemma}
\begin{proof}
For $u(0)\leq c<u(0)+\epsilon$, we have
\begin{align*}
\int_{c+\epsilon}^{u(1)}\frac{\int_{c}^{c'}f(z)dz}{(c-c')^2}dc'
=&\int_{c+\epsilon}^{u(1)}-\partial_{c'}\big(\frac{1}{c'-c}\big)\int_{c}^{c'}f(z)dzdc'\\
=&-\frac{1}{u(1)-c}\int_{c}^{u(1)}f(z)dz
+\frac{1}{\epsilon}\int_c^{c+\epsilon}f(z)dz
+\int_{c+\epsilon}^{u(1)}\frac{f(z)}{z-c}dz.
\end{align*}
For $u(0)+\epsilon\leq c\leq u(1)-\epsilon$, we have
\begin{align*}
&\int_{c+\epsilon}^{u(1)}+\int_{u(0)}^{c-\epsilon}\frac{\int_{c}^{c'}f(z)dz}{(c'-c)^2}dc'\\
&=-\int_{c+\epsilon}^{u(1)}+\int_{u(0)}^{c-\epsilon}\partial_{c'}\big(\frac{1}{c'-c}\big)\int_{c}^{c'}f(z)dzdc'\\
&=-\frac{1}{u(1)-c}\int_{c}^{u(1)}f(z)dz
+\frac{1}{\epsilon}\int_c^{c+\epsilon}f(z)dz
+\int_{c+\epsilon}^{u(1)}+\int_{u(0)}^{c-\epsilon}\frac{f(z)}{z-c}dz\\
&\quad-\frac{1}{\epsilon}\int_{c-\epsilon}^cf(z)dz
+\frac{1}{c-u(0)}\int_{u(0)}^cf(z)dz.
\end{align*}
For $u(1)-\epsilon<c\leq u(1)$,
\begin{align*}
\int_{u(0)}^{c-\epsilon}\frac{\int_{c}^{c'}f(z)dz}{(c'-c)^2}dc'
=&\int_{u(0)}^{c-\epsilon}-\partial_{c'}\big(\frac{1}{c'-c}\big)\int_{c}^{c'}f(z)dzdc'\\
=&\frac{1}{c-u(0)}\int_{u(0)}^{c}f(z)dz
-\frac{1}{\epsilon}\int_{c-\epsilon}^cf(z)dz
+\int_{u(0)}^{c-\epsilon}\frac{f(z)}{z-c}dz.
\end{align*}
Thus, we obtain
\begin{align*}
\int_{\{|c-c'|>\epsilon\}\cap[u(0),u(1)]}&\frac{\int_{c}^{c'}f(z)dz}{(c-c')^2}dc'=-\chi_{D_0}(c)H^{\epsilon}(f\chi_{D_0})(c)\\
&+\Big(\frac{1}{c-u(0)}\int_{u(0)}^{c}f(z)dz-\frac{1}{\epsilon}\int_{c-\epsilon}^cf(z)dz\Big)\chi_{[\epsilon,u(1)]}(c)\\
&-\Big(\frac{1}{u(1)-c}\int_{c}^{u(1)}f(z)dz-\frac{1}{\epsilon}\int_c^{c+\epsilon}f(z)dz\Big)\chi_{[u(0),u(1)-\epsilon]}(c).
\end{align*}
Taking $\epsilon\to 0+$ and using the fact that $\frac{1}{\epsilon}\int_{x-\epsilon}^xf(y)dy-\frac{1}{\epsilon}\int_x^{x+\epsilon}f(y)dy\to 0$,
we deduce
\begin{align*}
\textrm{p.v.}\int_{u(0)}^{u(1)}\frac{\int_{c}^{c'}f(z)dz}{(c-c')^2}dc'
=&-\chi_{D_0}(c)H(f\chi_{D_0})(c)\\
&+\frac{1}{c-u(0)}\int_{u(0)}^cf(z)dz-\frac{1}{u(1)-c}\int_c^{u(1)}f(z)dz.
\end{align*}
With the above identities, the boundedness in $L^p$ follows from $L^p$ boundedness of Hilbert transform, maximal Hilbert transform
and Hardy-Littlewood maximal function.
\end{proof}

\subsection{Boundedness of multilinear singular integral operators}

\begin{lemma}\label{lem:T-L2}
There exists a constant $C$ independent of $\al$ such that for any $p\in (1,\infty)$,
\beq
\label{equ: estimate bbT Lp}
\Big\|\bbT\Big(f, g, \phi_1(u^{-1},c),\,\frac{1}{\phi_1(u^{-1},c)^2}\Big)\Big\|_{L^p}
\leq C(\|f\|_{L^{\infty}}+\|f'\|_{L^{\infty}})\|g\|_{L^p}.
\eeq
\end{lemma}

\begin{proof}
We write
\begin{align*}
&\bbT\Big(f, g, \phi_1(u^{-1},c),\,\frac{1}{\phi_1(u^{-1},c)^2}\Big)(c)\\
=&\int_{u(0)}^{u(1)}\frac{\int_{c}^{c'}g(z)(\phi_1(u^{-1}(z),c)-1)dz}{(c'-c)^2\phi_1(u^{-1}(c'),c)^2}f(c')dc'\\
&+\int_{u(0)}^{u(1)}\frac{\int_{c}^{c'}g(z)dz}{(c'-c)^2}\frac{f(c')-f(c)}{\phi_1(u^{-1}(c'),c)^2}dc'\\
&+f(c)\int_{\{|c'-c|\leq \frac{1}{\al}\}\cap D_0}\frac{\int_{c}^{c'}g(z)dz}{(c'-c)^2}\frac{(\phi_1(u^{-1}(c'),c)+1)(\phi_1(u^{-1}(c'),c)-1)}{\phi_1(u^{-1}(c'),c)^2}dc'\\
&+f(c)\int_{\{|c'-c|>\frac{1}{\al}\}\cap D_0}\frac{\int_{c}^{c'}g(z)dz}{(c'-c)^2\phi_1(u^{-1}(c'),c)^2}dc'\\
&+f(c)\textrm{p.v.}\int_{u(0)}^{u(1)}\frac{\int_{c}^{c'}g(z)dz}{(c'-c)^2}dc'
-f(c)\textrm{p.v.}\int_{\{|c'-c|>\frac{1}{\al}\}\cap D_0}\frac{\int_{c}^{c'}g(z)dz}{(c'-c)^2}dc'\\
\triangleq&I_1+\cdots I_6.
\end{align*}

{\bf Step 1.} Estimate of $I_1$\medskip

We infer from  Remark \ref{rem:phi1} that
\begin{align*}
I_1=\al^2\int_{u(0)}^{u(1)}\frac{\int_{c}^{c'}g(z)(u^{-1}(z)-u^{-1}(c))^2\mathcal{T}(\phi_1)(u^{-1}(z),c)dz}{(c'-c)^2\phi_1(u^{-1}(c'),c)^2}f(c')dc'.
\end{align*}
For $|z-c|\leq|c'-c|$, we have $(u^{-1}(z)-u^{-1}(c))^2\leq C(c'-c)^2$. And by Remark \ref{rem:phi1},
\beno
\mathcal{T}(\phi_1)(u^{-1}(z),c)\leq C\phi_1(u^{-1}(z),c)\leq C\phi_1(u^{-1}(c'),c).
\eeno
Then we infer that
\begin{align*}
|I_1|\leq C\al^2\|f\|_{L^\infty}\int_{u(0)}^{u(1)}\int_{c}^{c'}|g(z)|dz\frac{1}{\phi_1(u^{-1}(c'),c)}dc',
\end{align*}
from which and Proposition \ref{prop:phi}, it follows that
\begin{align*}
|I_1|\leq &C\al^2\|f\|_{L^\infty}\|g\|_{L^{\infty}}\int_0^1\frac{|y-y_c|}{\phi_1(y,c)}dy\\
\leq &C\bigg(\int_{0}^{\al(1-y_c)}\frac{y^2}{\sinh y}dy
+\int_{0}^{\al y_c}\frac{y^2}{\sinh y}dy\bigg)\|f\|_{L^\infty}\|g\|_{L^{\infty}}\\
\leq &C\|f\|_{L^\infty}\|g\|_{L^{\infty}}.
\end{align*}
And by Proposition \ref{prop:phi} again, we get
\begin{align*}
\|I_{1}\|_{L^1}
\leq &C\|f\|_{L^\infty}\bigg(\int_0^1\int_0^{y_c}\al^2\int_z^{y_c}|g(u(y''))|\frac{\al(z-y_c)}{\sinh \al(z-y_c)}dy''dzdy_c
\\\nonumber
&+\int_0^1\int_{y_c}^1\al^2\int_{y_c}^z|g(u(y''))|\frac{\al(z-y_c)}{\sinh \al(z-y_c)}dy''dzdy_c\bigg)\\\nonumber
\leq &C\al^2\|f\|_{L^\infty}\int_0^1|g(u(y''))|\int_0^{y''}\int_{y''}^1\frac{\al(z-y_c)}{\sinh \al(z-y_c)}dy_cdzdy''\\\nonumber
&+C\al^2\|f\|_{L^\infty}\int_0^1|g(u(y''))|\int_0^{y''}\int_{y''}^1\frac{\al(z-y_c)}{\sinh \al(z-y_c)}dzdy_cdy''.
\end{align*}
Using the fact that
\beno
\frac{z}{\sinh z}\leq 2\frac{\sinh\frac{z}{2}}{\sinh z}\leq \frac{1}{\cosh \frac{z}{2}}\leq \frac{1}{e^\frac{z}{2}},~z>0.
\eeno
we deduce that
\begin{align*}
\bigg|\al^2\int_0^y\int_y^1\frac{\al(z-y_c)}{\sinh \al(z-y_c)}dy_cdz\bigg|
\leq& \int_{\al y}^{+\infty}\int_0^{\al y}e^{\frac{1}{2}(z-y_c)}dzdy_c\\
\leq& 2\int_{\al y}^{+\infty}e^{\frac{1}{2}\al (y-y_c)}-e^{-\frac12y_c}dy_c\leq 8.
\end{align*}
This shows that
\beno
\|I_1\|_{L^1}\leq C\|f\|_{L^\infty}\|g\|_{L^1}.
\eeno
Then by the interpolation, we get
\beno
\|I_1\|_{L^p}\leq C\|f\|_{L^\infty}\|g\|_{L^p}.
\eeno

{\bf Step 2.} Estimate of $I_2$\medskip

Thanks to $\phi_1(u^{-1}(z),c)^2\geq1$, we have
\begin{align*}
|I_2|\leq&
C\int_{u(0)}^{u(1)}\frac{\int_{c}^{c'}|g(z)|dz}{c'-c}\frac{\|f'\|_{L^{\infty}}}{\phi_1(u^{-1}(c'),c)^2}dc'\\
\leq&C\|f'\|_{L^{\infty}}\int_{u(0)}^{u(1)}\frac{\int_{c}^{c'}|g(z)|dz}{c'-c}dc'\le C\|f'\|_{L^{\infty}}\|g\|_{L^\infty}.
\end{align*}
For $L^1$ estimate, we write
\begin{align*}
\int_{u(0)}^{u(1)}\frac{\int_{c}^{c'}|g(z)|dz}{c'-c}dc'
=&\int_{u(0)}^{c}\frac{\int_{c}^{c'}|g(z)|dz}{c'-c}dc'
+\int_{c}^{u(1)}\frac{\int_{c}^{c'}|g(z)|dz}{c'-c}dc'\\
=&\int_{u(0)}^c|g(z)|\int_{u(0)}^{z}\frac{1}{c-c'}dc'dz
+\int_{c}^{u(1)}|g(z)|\int_{z}^{u(1)}\frac{1}{c'-c}dc'dz\\
=&\int_{u(0)}^c|g(z)|\ln\frac{c-u(0)}{c-z}dz
+\int_{c}^{u(1)}|g(z)|\ln\frac{u(1)-c}{z-c}dz,
\end{align*}
from which, it follows that
\begin{align*}
\bigg\|\int_{u(0)}^{u(1)}\frac{\int_{c}^{c'}|g(z)|dz}{c'-c}dc'\bigg\|_{L^1}\leq &
\int_{u(0)}^{u(1)}\int_{u(0)}^c|g(z)|\ln\frac{c-u(0)}{c-z}dz
+\int_{c}^{u(1)}|g(z)|\ln\frac{u(1)-c}{z-c}dzdc\\\nonumber
\leq &\int_{u(0)}^{u(1)}|g(z)|\int_{z}^{u(1)}\ln\frac{c-u(0)}{c-z}dcdz
+\int_{u(0)}^{u(1)}|g(z)|\int_{u(0)}^{z}\ln\frac{u(1)-c}{z-c}dcdz\\\nonumber
\leq &C\|g\|_{L^{1}}.
\end{align*}
This gives
\beno
\|I_2\|_{L^1}\le C\|f'\|_{L^{\infty}}\|g\|_{L^{1}}.
\eeno
By the interpolation, we get
\begin{align*}
\|I_2\|_{L^{p}}\leq C\|f'\|_{L^{\infty}}\|g\|_{L^{p}}.
\end{align*}

{\bf Step 3.} Estimate of $I_3$\medskip

By Lemma \ref{lem:phi1-est}, $|\phi_1(u^{-1}(c'),c)-1|\leq \al^2(c'-c)^2\phi_1(u^{-1}(c'),c)$. Then we have
\begin{align*}
|I_3|\leq& C\al^2\|f\|_{L^\infty}\int_{\{0\leq c'-c\leq \frac{1}{\al}\}\cap D_0}\int_{c}^{c'}|g(z)|dzdc'\\
&+C\al^2\|f\|_{L^\infty}\int_{\{0\leq c-c'\leq \frac{1}{\al}\}\cap D_0}\int_{c'}^{c}|g(z)|dzdc'\\
\leq & C\|f\|_{L^\infty}\al^2\int_{c}^{\min\big\{c+\frac{1}{\al},u(1)\big\}}\int_{c}^{c'}|g(z)|dzdc'\\
&+C\|f\|_{L^\infty}\al^2\int_{\max\big\{c-\frac{1}{\al},u(0)\big\}}^c\int_{c'}^{c}|g(z)|dzdc',
\end{align*}
from which, it follows that
\begin{align*}
\|I_3\|_{L^{\infty}}\leq& C\|f\|_{L^\infty}\bigg\|
\al^2\int_{\{|c'-c|\leq \frac{1}{\al}\}\cap D_0}
\|g\|_{L^{\infty}}|c'-c|dc'\bigg\|_{L^{\infty}}\\
\leq &C\|f\|_{L^\infty}\|g\|_{L^{\infty}},
\end{align*}
and
\begin{align*}
\|I_3\|_{L^{1}}\leq& C\al^2\|f\|_{L^\infty}\bigg(\int_{u(0)}^{u(1)}\int_{c}^{\min\big\{c+\frac{1}{\al},u(1)\big\}}\int_{c}^{c'}|g(z)|dzdc'dc\\
&\qquad\qquad\quad+\int_{u(0)}^{u(1)}\int_{\max\big\{c-\frac{1}{\al},u(0)\big\}}^c\int_{c'}^{c}|g(z)|dzdc'dc\bigg)\\
\leq& C\al^2\|f\|_{L^\infty}\bigg(\int_{u(0)}^{u(1)}|g(z)|\int_{z-\frac{1}{\al}}^{z}\big(c+\frac{1}{\al}-z\big)dcdz\\
&\qquad\qquad\quad+\int_{u(0)}^{u(1)}|g(z)|\int_{z}^{z+\frac{1}{\al}}\big(\frac{1}{\al}+z-c\big)dcdz\bigg)\\
\leq& C\|f\|_{L^\infty}\|g\|_{L^{1}}.
\end{align*}
Thus by the interpolation, we get
\begin{align*}
\|I_{3}\|_{L^{p}}\leq& C\|f\|_{L^\infty}\|g\|_{L^{p}}.
\end{align*}

{\bf Step 4.} Estimate of $I_4$\medskip

We infer from Proposition \ref{prop:phi} that
\beno
\f 1 {\phi_1(u^{-1}(c'),c)}\geq C^{-1}\frac{\al(u^{-1}(c')-u^{-1}(c))}{\sinh\al(u^{-1}(c')-u^{-1}(c))},
\eeno
which gives
\begin{align*}
|I_4|\leq& C\|f\|_{L^\infty}\int_{|c'-c|\geq \frac{1}{\al}}
\|g\|_{L^{\infty}}\frac{\al^2|u^{-1}(c')-u^{-1}(c)|}{\sinh^2\al(u^{-1}(c')-u^{-1}(c))}dc'\\
\leq& C\|f\|_{L^\infty}\|g\|_{L^\infty}\int_1^{\infty}\frac{z}{\sinh^2 z}dz\\
\leq& C\|f\|_{L^\infty}\|g\|_{L^\infty},
\end{align*}
and
\begin{align*}
\|I_4\|_{L^1}\leq &C\al^2\|f\|_{L^\infty}\int_{u(0)}^{u(1)}\int_{\{|z-c|\geq \frac{1}{\al}\}\cap\{z\geq c\}\cap [u(0),u(1)]}\int_{c}^z|g(c')|\frac{1}{\phi_1(u^{-1}(z),c)^2}dc'dzdc\\
&+C\al^2\|f\|_{L^\infty}\int_{u(0)}^{u(1)}\int_{\{|z-c|\geq \frac{1}{\al}\}\cap\{z< c\}\cap [u(0),u(1)]}\int_{z}^c|g(c')|\frac{1}{\phi_1(u^{-1}(z),c)^2}dc'dzdc.
\end{align*}
By a change of variable, the first term is bounded by
\begin{align*}
&C\|f\|_{L^\infty}\int_0^1\int_{\{|u(y'')-c|\geq \frac{1}{\al}\}\cap\{u(y'')\geq c\}\cap [0,1]}\int_{y_c}^{y''}|g(u(y'))|\frac{\al^2}{\phi_1(y'',c)^2}dy'dy''dy_c\\
&\leq C\|f\|_{L^\infty}\int_0^1|g(u(y'))|\int_{0}^{y'}\int_{y'}^1\frac{\al^2}{\phi_1(y'',c)^2}dy''dy_cdy'\\
&\leq C\|f\|_{L^\infty}\|g\|_{L^{1}},
\end{align*}
where we used the fact that $\phi_1(y'',c)\geq C^{-1}\frac{\sinh \al(y''-y_c)}{\al(y''-y_c)}$ so that
\begin{align*}
\Big|\int_{0}^{y'}\int_{y'}^1\frac{\al^2}{\phi_1^2(y'',c)}dy''dy_c\Big|
\leq &C\int_{0}^{y'}\int_{y'}^1\frac{\al^4(y''-y_c)^2}{\sinh \al(y''-y_c)^2}dy''dy_c\\
\leq &C\int_{0}^{\al y'}\int_{\al y'}^{+\infty}e^{-1/2(y-y_c)}dydy_c\le C.
\end{align*}
The estimate of the second term is the same. This gives
\beno
\|I_4\|_{L^1}\le C\|f\|_{L^\infty}\|g\|_{L^{1}}.
\eeno
Thus, the interpolation gives
\beno
\|I_4\|_{L^p}\le C\|f\|_{L^\infty}\|g\|_{L^{p}}.
\eeno

{\bf Step 5.} Estimates of $I_5$ and $I_6$\medskip

Lemma \ref{lem:sing2} ensures that
\beno
\|I_5\|_{L^p}+\|I_6\|_{L^p}\le C\|f\|_{L^\infty}\|g\|_{L^{p}}.
\eeno

Putting the estimates in Step 1--Step 5 together, we conclude the lemma.
\end{proof}

\begin{lemma}\label{lem:B01-L2}
There exists a constant $C$ independent of $\al$ such that for any $p\in (1,\infty)$,
\begin{align*}
&\|\mathbb{B}_0(f,\rho g, F,\partial_cG)\|_{L^p}
+\|\mathbb{B}_1(f,\rho g,F,\partial_cG)\|_{L^p}\\\nonumber
&+\|\mathbb{B}_0(f,g,(\partial_z+\partial_c)F,G)\|_{L^p}
+\|\mathbb{B}_1(f,g,(\partial_z+\partial_c)F,G)\|_{L^p}
\leq C\|f\|_{L^\infty}\|g\|_{L^p},
\end{align*}
where $F(z,c)=\phi_1(u^{-1}(z),c)$ and $G(z,c)=\frac{1}{\phi_1(u^{-1}(z),c)^2}$. We
also have
\beno
\|\mathcal{B}(\rho g)\|_{L^p}\le \f 1 {\al}\|g\|_{L^p}.
\eeno
\end{lemma}

\begin{proof}
We only prove the estimate of $\mathbb{B}_0$, the estimate of $\mathbb{B}_1$ is similar. We get by Proposition \ref{prop:phi} that
\begin{align*}
\big|\mathbb{B}_0(f,\rho g,F,\pa_cG)(c)\big|
\leq C\|f\|_{L^\infty}\frac{\al}{\phi_1(0,c)}\int_{u(0)}^c|g(z)|dz,
\end{align*}
from which, it follows that
\begin{align*}
\big\|\mathbb{B}_0(f,\rho g,F,\pa_cG)\big\|_{L^{\infty}}
\leq &C\sup_{y\in [0,1]}\frac{\al^2 y^2}{\sinh\al y}\|f\|_{L^\infty}\|g\|_{L^{\infty}}\\
\leq &C\|f\|_{L^\infty}\|g\|_{L^{\infty}},\\
\big\|\mathbb{B}_0(f,\rho g,F,\pa_c G)\big\|_{L^1}
\leq &C\|f\|_{L^\infty}\int_0^1|g(z)|dz\int_0^1\frac{\al^2 y}{\sinh\al y}dy\\
\leq &\|f\|_{L^\infty}\|g\|_{L^{1}}.
\end{align*}

We get by Proposition \ref{prop:phi-good} that
\begin{align*}
|\mathbb{B}_0(f,g,(\partial_z+\partial_c)F,G)(c)|
\leq &C\|f\|_{L^\infty}\frac{\al^2(c-u(0))}{\phi_1(0,c)}\int_{u(0)}^c|g(z)|dz,
\end{align*}
which implies that
\begin{align*}
&\|\mathbb{B}_0(f,g,(\partial_z+\partial_c)F,G)\|_{L^{\infty}}
\leq \|f\|_{L^\infty}\|g\|_{L^{\infty}},\\
&\big\|\mathbb{B}_0(f,g,(\partial_z+\partial_c)F,G)\big\|_{L^1}
\le C\|f\|_{L^\infty}\|g\|_{L^{1}}.
\end{align*}
Then $L^p$ estimate follows by the interpolation.

Notice that
\begin{align*}
|\al\mathcal{B}(\rho g)|\leq C\frac{\al}{\phi_1(0,c)}\int_{u(0)}^c|g(z)|dz+C\frac{\al}{\phi_1(1,c)}\int_{c}^{u(1)}|g(z)|dz.
\end{align*}
Then $L^p$ estimate of $\mathcal{B}(\rho g)$ can be deduced in a similar way.
\end{proof}

\begin{lemma}\label{lem:B-H1}
There exists a constant $C$ independent of $\al$ such that for any $p\in (1,\infty)$,
\begin{align*}
\|\mathcal{B}(g)\|_{L^p}+\|\pa_c(\rho\mathcal{B}(g))\|_{L^p}+\|\rho^{-1}\mathcal{B}(\rho g)\|_{L^p}+\|\mathcal{B}_1(g)\|_{L^p}\le C\|g\|_{L^p}.
\end{align*}
If $g(u(0))=g(u(1))=0$, then we have
\beno
\|\rho\mathcal{B}(g')\|_{L^p}\le C\|g\|_{L^p}.
\eeno
\end{lemma}

\begin{proof}
The estimate of $\|\mathcal{B}(g)\|_{L^p}$ follows from $L^p$ boundedness of Hardy-Littlewood maximal function.

{\bf Step 1.} Estimate of $\|\pa_c(\rho\mathcal{B}(g))\|_{L^p}$\medskip

It suffices to consider $\mathbb{B}_0$. By the definition of $\mathbb{B}_0$, we have
\begin{align*}
&\pa_c\Big(\rho(c)\mathbb{B}_0\Big(f,g,\phi_1(u^{-1}(\cdot),c),\frac{1}{\phi_1(0,c)^2}\Big)\Big)\\
&=2(c-u(1))f(u(0))\int_{u(0)}^cg(z)\frac{\phi_1(u^{-1}(z),c)}{\phi_1(0,c)^2}dz
+\frac{(u(1)-c)^2f(u(0))g(c)}{\phi_1(0,c)^2}\\
&\quad+(u(1)-c)^2f(u(0))\int_{u(0)}^cg(z)\pa_c\Big(\frac{\phi_1(u^{-1}(z),c)}{\phi_1(0,c)^2}\Big)dz.
\end{align*}
The $L^p$ estimate of the first two terms are obvious. So we only consider the last term, which is bounded by
\beno
C\|f\|_{L^{\infty}}\int_0^{y_c}|g(u(y))|\frac{\al}{\phi_1(0,c)}dy\leq C\|f\|_{L^{\infty}}\frac{\al^2 y_c}{\sinh \al y_c}\int_0^{y_c}|g(u(y))|dy,
\eeno
where we used Proposition \ref{prop:phi}. Thus, $L^\infty$ norm of the last term is bounded by
\begin{align*}
C\sup_{y\in[0,1]}\frac{\al^2 y^2}{\sinh \al y}\|g\|_{L^{\infty}}\|f\|_{L^{\infty}},
\end{align*}
while $L^1$ norm is bounded by
\beno
C\|f\|_{L^{\infty}}\int_0^1|g(u(y))|\int_{y_c}^1\frac{\al^2 y_c}{\sinh \al y_c}dy_cdy
\leq C\|f\|_{L^{\infty}}\|g\|_{L^1}.
\eeno
This implies by the interpolation that
\beno
\|\pa_c(\rho\mathcal{B}(g))\|_{L^p}\leq C\|g\|_{L^p}.
\eeno

{\bf Step 2.} Estimate of $\|\rho^{-1}\mathcal{B}(\rho g)\|_{L^p}$\medskip

Using the fact that
\ben\label{eq:phi1-comp}
\phi_1(u^{-1}(z),c)\leq \phi_1(u^{-1}(c'),c)\quad \textbf{for } |z-c|\leq |c'-c|,
\een
we deduce that
 \begin{align*}
\big|\rho^{-1}\mathcal{B}(\rho g)(c)\big|
\leq &\frac{C}{(c-u(0))^2}\int_{u(0)}^cg(z)(z-u(0))(u(1)-z)dz\\
&+\frac{C}{(u(1)-c)^2}\int_c^{u(1)}g(z)(z-u(0))(u(1)-z)dz\\
\leq &\frac{C}{c-u(0)}\int_{u(0)}^cg(z)dz
+\frac{C}{u(1)-c}\int_c^{u(1)}g(z)dz,
\end{align*}
from which and $L^p$ boundedness of Hardy-Littlewood maximal function, we infer that
\beno
\|\rho^{-1}\mathcal{B}(\rho g)\|_{L^p}\leq C\|g\|_{L^p}.
\eeno

{\bf Step 3.} Estimate of $\|\mathcal{B}_1(g)\|_{L^p}$\medskip

By Proposition \ref{prop:phi-good}, we have
\beno
&&|(\partial_z+\partial_c)\phi_1(u^{-1}(z),c)|\leq C_0\al^2(z-c)^2\phi_1(u^{-1}(z),c),\\
&&\bigg|(\partial_z+\partial_c)\big(\frac{1}{\phi_1(u^{-1}(z),c)^2}\big)\bigg|\leq C_0\al^2(z-c)^2\frac{1}{\phi_1(u^{-1}(z),c)^2},
\eeno
which along with (\ref{eq:phi1-comp}) gives
So we have
\begin{align*}
\mathcal{B}_1(g)\leq \frac{C}{c-u(0)}\int_{u(0)}^c|\varphi(y)|dy+\frac{C}{u(1)-c}\int_c^{u(1)}|\varphi(y)|dy,
\end{align*}
from which and $L^p$ boundedness of Hardy-Littlewood maximal function, we infer that
\beno
\|\mathcal{B}_1(g)\|_{L^p}\leq C\|g\|_{L^p}.
\eeno

{\bf Step 4.} Estimate of $\|\rho\mathcal{B}(g')\|_{L^2}$\medskip

Using the boundary condition of $g$, we get
\begin{align*}
\rho(c)\mathcal{B}(g')=&(u(1)-c)^2(u^{-1})'(u(0))\int_{u(0)}^cg'(z)\frac{\phi_1(u^{-1}(z),c)}{\phi_1(0,c)^2}dz\\
&+(c-u(0))^2(u^{-1})'(u(1))\int_{c}^{u(1)}g'(z)\frac{\phi_1(u^{-1}(z),c)}{\phi_1(1,c)^2}dz\\
=&-(u(1)-c)^2(u^{-1})'(u(0))\int_{u(0)}^cg(z)\frac{\partial_z\phi_1(u^{-1}(z),c)}{\phi_1(0,c)^2}dz\\
&-(c-u(0))^2(u^{-1})'(u(1))\int_{c}^{u(1)}g(z)\frac{\partial_z\phi_1(u^{-1}(z),c)}{\phi_1(1,c)^2}dz\\
&+(u(1)-c)^2(u^{-1})'(u(0))g(c)\frac{1}{\phi_1(0,c)^2}\\
&-(c-u(0))^2(u^{-1})'(u(1))g(c)\frac{1}{\phi_1(1,c)^2}.
\end{align*}
Then by Proposition \ref{prop:phi}, we obtain
\begin{align*}
|\rho(c)\mathcal{B}(g')(c)|\leq & C|g(c)|
+C\int_{0}^{y_c}|g(u(z))|\frac{\al\cosh\al(z-y_c)}{\phi_1(0,c)^2}dz
+C\int_{y_c}^{1}|g(u(z))|\frac{\al\cosh\al(z-y_c)}{\phi_1(1,c)^2}dz.
\end{align*}
with $c=u(y_c)$. Thus, we have
\begin{align*}
\|\rho\mathcal{B}(g')\|_{L^{\infty}}\leq& C\|g\|_{L^{\infty}}+C\|g\|_{L^{\infty}}\sup_{c\in D_0}\Big(\frac{\sinh\al y_c}{\phi_1(0,c)^2}
+\frac{\sinh\al(1-y_c)}{\phi_1(1,c)^2}\Big)\\
\leq& C\|g\|_{L^{\infty}}.
\end{align*}
On the other hand, we have
\begin{align*}
\|\rho(c)\mathcal{B}(g')\|_{L^{1}}\leq& C\|g\|_{L^{1}}
+C\int_0^1\int_{0}^{y_c}|g(z)|\frac{\al\cosh\al(z-y_c)}{\phi_1(0,c)^2}dzdy_c\\
&+C\int_0^1\int_{y_c}^{1}|g(z)|\frac{\al\cosh\al(z-y_c)}{\phi_1(1,c)^2}dzdy_c\\
\leq &C\|g\|_{L^{1}}+C\|g\|_{L^{1}}\int_{0}^{\infty}\frac{z^2\cosh z}{\sinh^2(z)}dz\\
\leq& C\|g\|_{L^{1}}.
\end{align*}
By the interpolation, we get
\beno
\|\rho(c)\mathcal{B}(g')\|_{L^{p}}\leq C\|g\|_{L^{p}}.
\eeno

The proof is completed.
\end{proof}

\begin{lemma}\label{lem:L-H1}
There exists a constant $C$ independent of $\al$ such that for any $p\in (1,\infty)$,
\beno
\|\mathcal{L}_1(g)\|_{L^p}+\|\mathcal{L}_2(g)\|_{L^p}+\|\mathcal{L}_3(g)\|_{L^p}\leq C\|g\|_{L^p}.
\eeno
\end{lemma}

\begin{proof}
The estimate of $\|\mathcal{L}_1(g)\|_{L^2}$ follows from Lemma \ref{lem:T-L2}. Now we consider $\mathcal{L}_2(g)$.
It suffices to consider the first two terms in the definition of $\mathcal{L}_2(g)$. By Proposition \ref{prop:phi-good},
both terms are bounded by
\begin{align*}
C\al^2\int_{u(0)}^{u(1)}\Big|\int_{c}^{c'}|g(z)|dz\Big|\frac{1}{\phi_1(u^{-1}(c'),c)}dc'.
\end{align*}
With this, the proof is the same as $I_1$ in the proof of Lemma \ref{lem:T-L2}.

Next we consider $\mathcal{L}_3(g)$. It suffice to estimate the following terms
\beno
&&\mathbb{T}\Big((u^{-1})', g,(\partial_z+\partial_c)^2\phi_1(u^{-1},c),\frac{1}{\phi_1(u^{-1},c)^2}\Big)(c)\\
&&\qquad+\mathbb{T}\Big((u^{-1})', g,\phi_1(u^{-1},c),(\partial_z+\partial_c)^2\big(\frac{1}{\phi_1(u^{-1},c)^2}\big)\Big)(c)\triangleq \mathcal{L}_2^1(g).
\eeno
By Proposition \ref{prop:phi-good}, we have
\beno
|(\partial_z+\partial_c)^2\phi_1(u^{-1}(z),c)|\leq C\al^2(u^{-1}(z)-u^{-1}(c))^2\cosh\al(u^{-1}(z)-u^{-1}(c)),
\eeno
and
\beno
\bigg|(\partial_z+\partial_c)\big(\frac{1}{\phi_1^2(u^{-1}(z),c)}\big)\bigg|\leq C\al^2(u^{-1}(z)-u^{-1}(c))^2\frac{\cosh\al(u^{-1}(z)-u^{-1}(c))}{\phi_1^3(u^{-1}(z),c)},
\eeno
which imply
\begin{align*}
|\mathcal{L}_2^1(g)(c)|\leq C\al^2\int_{u(0)}^{u(1)}\Big|\int_{c}^{c'}|g(z)|dz\Big|\frac{\cosh\al(u^{-1}(c')-u^{-1}(c))}{\phi_1^2(u^{-1}(c'),c)}dc'.
\end{align*}
Thus, we obtain
\begin{align*}
\|\mathcal{L}_2^1(g)\|_{L^{\infty}}\leq C\int_0^{\infty}\frac{z^3\cosh z}{\sinh^2 z}dz\|g\|_{L^{\infty}}
\leq C\|g\|_{L^{\infty}},
\end{align*}
and
\begin{align*}
\|\mathcal{L}_2^1(g)\|_{L^1}
\leq &C\al^2\int_0^1|g(y)|\int_0^y\int_y^1\frac{\al^2(z-y_c)^2\cosh \al(z-y_c)}{\sinh^2 \al(z-y_c)}dy_cdzdy\\
&+C\al^2\int_0^1|g(y)|\int_0^y\int_y^1\frac{\al^2(z-y_c)^2\cosh \al(z-y_c)}{\sinh^2 \al(z-y_c)}dzdy_cdy\\
\le& C\|g\|_{L^1},
\end{align*}
where we used $\frac{z^2\cosh z}{\sinh^2 z}\leq \frac{C_0}{e^{z/2}}$ for $z\geq 0$. By the interpolation, we get
\beno
\|\mathcal{L}_2^1(g)\|_{L^p}\leq C\|g\|_{L^{p}}.
\eeno

This completes the proof of the lemma.
\end{proof}

\begin{lemma}\label{lem:L-com}
There exists a constant $C$ independent of $\al$ such that for any $p\in (1,\infty)$,
\beno
\|[\mathcal{L}_1,\rho](g')\|_{L^p}+\|\pa_c[\mathcal{L}_1,\rho](g)\|_{L^p}\leq C\|g\|_{L^p}.
\eeno
\end{lemma}

\begin{proof}We have
\begin{align*}
\partial_c[\mathcal{L}_1,\rho]g
=&\mathcal{L}_1((\rho g)')+\mathcal{L}_2(\rho g)-\frac{1}{\rho}\mathcal{B}(\rho g)\\
&-\rho'\mathcal{L}_1(g)-\rho\mathcal{L}_1(g')-\rho\mathcal{L}_2(g)+\mathcal{B}(g)\\
=&[\mathcal{L}_1,\rho](g')+\mathcal{L}_1(\rho'g)+\mathcal{L}_2(\rho g)-\frac{1}{\rho}\mathcal{B}(\rho g)\\
&-\rho'\mathcal{L}_1(g)-\rho\mathcal{L}_2(g)+\mathcal{B}(g).
\end{align*}
Thus, by Lemma \ref{lem:B-H1} and Lemma \ref{lem:L-H1}, it suffices to consider $[\mathcal{L}_1,\rho](g')$.

Since
\begin{align*}
[\mathcal{L}_1,\rho](g')(c)=&
\int_{u(0)}^{u(1)}\frac{\int_c^{c'}(\rho(z)-\rho(c))g'(z)\phi_1(u^{-1}(z),c)dz}{(c'-c)^2}\frac{(u^{-1})'(c')}{\phi_1(u^{-1}(c'),c)^2}dc'\\
=&p.v.\int_{u(0)}^{u(1)}\frac{g(c')}{c'-c}\frac{(u(0)+u(1)-c'-c)(u^{-1})'(c')}{\phi_1(u^{-1}(c'),c)}dc'\\
&-\int_{u(0)}^{u(1)}\frac{\int_c^{c'}(\rho(z)-\rho(c))g(z)\partial_z\phi_1(u^{-1}(z),c)dz}{(c'-c)^2}\frac{(u^{-1})'(c')}{\phi_1(u^{-1}(c'),c)^2}dc'
-\mathcal{L}_1(\rho'g)\\
\triangleq&\Theta_1+\Theta_2-\mathcal{L}_1(\rho'g).
\end{align*}
Let $f(c',c)=(u(0)+u(1)-c'-c)(u^{-1})'(c')$. $\Theta_1$ can be rewritten as
\begin{align*}
&\textrm{p.v.}\int_{u(0)}^{u(1)}\frac{g(c')}{c'-c}\frac{(u(0)+u(1)-c'-c)(u^{-1})'(c')}{\phi_1(u^{-1}(c'),c)}dc'\\
&=(u(0)+u(1)-2c)(u^{-1})'(c)\textrm{p.v.}\int_{u(0)}^{u(1)}\frac{g(c')}{c'-c}dc'\\
&\quad-(u(0)+u(1)-2c)(u^{-1})'(c)\textrm{p.v.}\int_{D_0\cap \{|c'-c|>\frac{1}{\al}\}}\frac{g(c')}{c'-c}dc'\\
&\quad+\int_{u(0)}^{u(1)}\frac{g(c')}{c'-c}\frac{f(c',c)-f(c,c)}{\phi_1(u^{-1}(c'),c)}dc'\\
&\quad+\int_{D_0\cap \{|c'-c|>\frac{1}{\al}\}}\frac{g(c')}{c'-c}\frac{f(c,c)}{\phi_1(u^{-1}(c'),c)}dc'\\
&\quad-f(c,c)\int_{D_0\cap \{|c'-c|\leq\frac{1}{\al}\}}\frac{g(c')}{c'-c}\frac{\phi_1(u^{-1}(c'),c)-1}{\phi_1(u^{-1}(c'),c)}dc'.
\end{align*}
The $L^p$ boundedness of the first two terms follows from $L^p$ boundedness of Hilbert transform. The third term is trivial
since $|f(c',c)-f(c,c)|\leq C|c'-c|$.
The fourth term is bounded by
\begin{align*}
C\int_{[0,1]\cap \{|x-x_0|>\frac{c_0}{\al}\}}|g(u(y))|\frac{\al^2|y-y_c|}{\sinh\al(y-y_c)}dy.
\end{align*}
Thanks to $K_1(z)=\al\frac{\al z}{\sinh \al z}\chi_{[c_0,+\infty]}(\al z)\in L^1$, we get
\begin{align*}
\bigg\|\int_{D_0\cap \{|c'-c|>\frac{1}{\al}\}}\frac{g(c')}{c'-c}\frac{f(c,c)}{\phi_1(u^{-1}(c'),c)}dc'\bigg\|_{L^p}
\leq &C\||g\circ u|\ast K_1\|_{L^p}\\
\leq &C\|g\|_{L^p}.
\end{align*}
The last term  is bounded by
\begin{align*}
C\int_{D_0\cap \{|c'-c|\leq\frac{1}{\al}\}}g(c')\al^2|c'-c|dc'
\end{align*}
Thanks to $K_2(z)=\al^2 z\chi_{[0,1]}(\al z)\in L^1$, we get
\begin{align*}
\bigg\|\int_{D_0\cap \{|c'-c|\leq\frac{1}{\al}\}}\frac{g(c')}{c'-c}\frac{\phi_1(u^{-1}(c'),c)-1}{\phi_1(u^{-1}(c'),c)}dc'\bigg\|_{L^p}
\leq &C_0\||g|\ast K_2\|_{L^p}\\
\leq &C_0\|g\|_{L^p}.
\end{align*}
This shows that
\beno
\|\Theta_1\|_{L^p}\leq C\|g\|_{L^p}.
\eeno

For $\Theta_2$ term, we have
\begin{align*}
|\Theta_2(c)|\leq
C\al\int_0^1\frac{\int_{y_c}^y|g(u(z))|dz}{y-y_c}\frac{\al(y-y_c)}{\sinh\al(y-y_c)}dy,
\end{align*}
from which, it follows that
\begin{align*}
\|\Theta_2(c)\|_{L^{\infty}}\leq &
C\al\int_0^1\frac{\int_{y_c}^y\|g\|_{L^{\infty}}dz}{y-y_c}\frac{\al(y-y_c)}{\sinh\al(y-y_c)}dy\\
\leq &C\al\|g\|_{L^{\infty}}\int_0^1\frac{\al(y-y_c)}{\sinh\al(y-y_c)}dy\\
\leq &C\|g\|_{L^{\infty}}\big\|\frac{z}{\sinh z}\big\|_{L^1}
\leq C\|g\|_{L^{\infty}},
\end{align*}
and
\begin{align*}
\|\Theta_2(c)\|_{L^1}\leq &
C\int_0^1\al\int_0^1\frac{\int_{y_c}^y|g(u(z))|dz}{y-y_c}\frac{\al(y-y_c)}{\sinh\al(y-y_c)}dydy_c\\
\leq &C\int_0^1|g(u(z))|dz
\Big\|\int_0^z\int_z^1\frac{\al^2}{\sinh \al(y_c-y)}dy_cdy\Big\|_{L^{\infty}}\leq C\|g\|_{L^1}.
\end{align*}
By the interpolation, we get
\beno
\|\Theta_2(c)\|_{L^p}\leq C\|g\|_{L^p}.
\eeno
Thus, we obtain
\beno
\|[\mathcal{L}_1,\rho](g')\|_{L^p}\le \|\Theta_1(c)\|_{L^p}+\|\Theta_2(c)\|_{L^p}+\|\mathcal{L}_1(\rho'g)\|_{L^p}\le C\|g\|_{L^p}.
\eeno

This completes the proof of the lemma.
\end{proof}

\section*{Acknowledgement}
Z. Zhang thanks Professor Zhiwu Lin for introducing this question to him and many profitable discussions.
This work was done when W. Zhao was visiting School of Mathematical Science, Peking University. He appreciates the hospitality from PKU.
Z. Zhang is partially supported by NSF of China under Grant 11371037 and 11425103. W. Zhao is partially supported by NSF of China under Grant 11571306.


\begin{thebibliography}{99}

\bibitem{Arnold} V. I. Arnold, {\it On an apriori estimate in the theory of hydrodynamical stability},
Amer. Math. Soc. Transl., 79(1969), 267-269.

\bibitem{Bar} C. Bardos, Y. Guo and W. Strauss, {\it Stable and unstable ideal plane flows}, Chinese Ann. Math. Ser. B,
23(2002), 149-164.

\bibitem{Bed1} J. Bedrossian and N. Masmoudi, {\it Inviscid damping and the asymptotic stability of planar shear flows in the 2D Euler equations}, arXiv:1306.5028.

\bibitem{Bed2} J. Bedrossian, N. Masmoudi and C. Mouhot,
{\it Landau damping: paraproducts and Gevrey regularity}, arXiv:1311.2870.

\bibitem{Bou} F. Bouchet and H. Morita, {\it Large time behavior and asymptotic stability of the 2D Euler and linearized
Euler equations}, Physica D, 239(2010), 948-966.

\bibitem{Che} J.-Y. Chemin, {\it Perfect Incompressible Fluids}, Oxford Lecture Ser. Math. Appl.
14, The Clarendon Press, Oxford University Press, New York, 1998.

\bibitem{Case} K. M. Case, {\it Stability of inviscid plane Couette flow}, Phys. Fluids, 3(1960), 143-148.

\bibitem{Den} S. A. Denisov, {\it Infinite superlinear growth of the gradient for the twodimensional
Euler equation}, Discrete Contin. Dyn. Syst.,  23(2009), 755-764.

\bibitem{Dra} P. G. Drazin and W. H. Reid, {\it Hydrodynamics stability}, Cambridge Monographs on Mech. and Applied Math.,
Cambridge University Press, Cambridge, 1981.


\bibitem{Gre} E. Grenier, {\it On the nonlinear instability of Euler and Prandtl equations},
Comm. Pure Appl. Math.,  53(2000), 1067-1091.

\bibitem{Kis} A. Kiselev and V. Sver\'{a}k, {\it Small scale creation for solutions of the incompressible two-dimensional Euler equation},
Ann. Math., 180(2014), 1205-1220.

\bibitem{Lan} L. Landau, {\it On the vibration of the electronic plasma},  J. Phys. USSR, 10(1946), 25.

\bibitem{Lin-SIAM} Z. Lin, {\it Instabilty of some ideal plane flows}, SIAM J. Math. Anal.,
35(2003), 318-356.

\bibitem{Lin-IN} Z. Lin, {\it Nonlinear instability of ideal plane flows}, Int. Math. Res. Not.,  41(2004), 2147-2178.

\bibitem{LZ-CMP} Z. Lin and C. Zeng, {\it Small BGK waves and nonlinear Landau damping}, Comm. Math. Phys., 306(2011), 291-331.

\bibitem{LZ} Z. Lin and C. Zeng, {\it Inviscid dynamic structures near Couette flow}, Arch. Rat. Mech. Anal., 200(2011), 1075-1097.

\bibitem{Maj} A. J. Majda and A. L. Bertozzi, {\it Vorticity and Incompressible Flow,}
Cambridge Texts Appl. Math. 27, Cambridge University Press, Cambridge, 2002.

\bibitem{Mou} C. Mouhot and C. Villani, {\it On Landau damping}, Acta Math., 207(2011), 29-201.

\bibitem{Orr} W. Orr, {\it Stability and instability of steady motions of a perfect liquid},
Proc. Ir. Acad. Sect. A: Math Astron. Phys. Sci., 27(1907), 9-66.

\bibitem{Ray} L. Rayleigh, {\it On the stability or instability of certain fluid motions}, Proc. London Math.
Soc., 9(1880), 57-70.

\bibitem{Ros} S. I. Rosencrans and D. H. Sattinger, {\it On the spectrum of an operator occurring in the theory of Hydrodynamics stability},
J. Math. Phys., 45(1966), 289-300.

\bibitem{Sch} P. J. Schmid and D. S. Henningson, {\it Stability and transition in shear flows}, Applied
Mathematical Sciences Vol. 142,  Springer-Verlag, New York, 2001.

\bibitem{Ste} S. A. Stepin, {\it Nonself-adjoint Friedrichs models in Hydrodynamics stability}, Functional Analysis and its Applications, 29(1995),
91-101.

\bibitem{Vish} M. M. Vishik and S. Friedlander, {\it Nonlinear instability in two dimensional ideal fluids: the case of a dominant eigenvalue},
 Comm. Math. Phys., 243(2003), 261-273.

\bibitem{Zill}  C. Zillinger, {\it Linear inviscid damping for monotone shear flows}, arXiv:1410.7341.

\bibitem{Zill2} C. Zillinger, {\it Linear inviscid damping for monotone shear flows in a finite periodic channel, boundary effects,
blow-up and critical Sobolev regularity}, arXiv:1506.04010.





\end{thebibliography}
\end{document}